\documentclass[10pt,reqno]{amsart}
\pdfoutput=1

\usepackage{adjustbox,array,bigints,cancel,color,comment,extpfeil,mathdots,mathrsfs,mathtools,MnSymbol,multicol,scalerel,setspace,tcolorbox,tikz,tikz-cd,upgreek}
\usepackage[fontsize = 10pt]{fontsize}
\usepackage[left = 2.5cm, right = 2.5cm, top = 2.5cm, bottom = 2.5cm]{geometry}
\usetikzlibrary{patterns}

\numberwithin{equation}{section}

\theoremstyle{plain}
\newtheorem{Cl}[equation]{Claim}

\newtheorem{Lem}[equation]{Lemma}

\newtheorem{Prop}[equation]{Proposition}

\newtheorem{Thm}[equation]{Theorem}

\theoremstyle{definition}

\newtheorem{Conds}[equation]{Conditions}

\newtheorem{Defn}[equation]{Definition}
\newtheorem{Ex}[equation]{Example}

\theoremstyle{remark}
\newtheorem{Rk}[equation]{Remark}

\theoremstyle{plain}
\newtheorem*{Cl*}{Claim}
\newtheorem*{Conj*}{Conjecture}
\newtheorem*{Lem*}{Lemma}
\newtheorem*{Prop*}{Proposition}
\newtheorem*{Q*}{Question}
\newtheorem*{Schol*}{Scholium}
\newtheorem*{SubCl*}{Subclaim}
\newtheorem*{Thm*}{Theorem}

\theoremstyle{definition}
\newtheorem*{Cond*}{Condition}
\newtheorem*{Cstr*}{Construction}
\newtheorem*{Defn*}{Definition}
\newtheorem*{Ex*}{Example}
\newtheorem*{Exs*}{Examples}
\newtheorem*{Md*}{Method}
\newtheorem*{Nt*}{Notation}
\newtheorem*{Pty*}{Property}

\theoremstyle{remark}

\newtheorem*{Rk*}{Remark}
\newtheorem*{Rks*}{Remarks}
\newtheorem*{A-d}{Aside}

\newcommand{\cag}{\begin{equation}\begin{gathered}}
\newcommand{\caag}{\end{gathered}\end{equation}}
\newcommand{\caw}{\begin{equation*}\begin{gathered}}
\newcommand{\caaw}{\end{gathered}\end{equation*}}
\newcommand{\e}{\begin{equation}\begin{aligned}}
\newcommand{\ee}{\end{aligned}\end{equation}}
\newcommand{\ew}{\begin{equation*}\begin{aligned}}
\newcommand{\eew}{\end{aligned}\end{equation*}}

\newcommand{\bcd}{\begin{tikzcd}}
\newcommand{\ecd}{\end{tikzcd}}
\newcommand{\bma}{\begin{matrix}}
\newcommand{\ema}{\end{matrix}}
\newcommand{\bpm}{\begin{pmatrix}}
\newcommand{\epm}{\end{pmatrix}}
\newcommand{\bvm}{\begin{vmatrix}}
\newcommand{\evm}{\end{vmatrix}}

\newcommand{\nts}{\begin{tcolorbox}}
\newcommand{\ntss}{\end{tcolorbox}}

\newcommand{\clref}[1]{Claim \ref{#1}}

\newcommand{\condsref}[1]{Conditions \ref{#1}}
\newcommand{\cref}[1]{Corollary \ref{#1}}

\newcommand{\eref}[1]{eqn.\hspace{0.6mm}(\ref{#1})}
\newcommand{\erefs}[1]{eqns.\hspace{0.6mm}(\ref{#1})}
\newcommand{\exref}[1]{Example \ref{#1}}

\newcommand{\lref}[1]{Lemma \ref{#1}}

\newcommand{\pref}[1]{Proposition \ref{#1}}

\newcommand{\rref}[1]{Remark \ref{#1}}

\newcommand{\sref}[1]{\S\ref{#1}}

\newcommand{\tref}[1]{Theorem \ref{#1}}
\newcommand{\trefs}[1]{Theorems \ref{#1}}

\DeclareMathSizes{10}{10}{8}{7}

\newcommand{\mss}[1]{\mbox{\scriptsize \(#1\)}}

\newcommand{\mns}[1]{\mbox{\normalsize \(#1\)}}
\newcommand{\mla}[1]{\mbox{\large \(#1\)}}

\newcommand{\bb}[1]{\mathbb{#1}}
\newcommand{\cal}[1]{\mathscr{#1}}
\newcommand{\fr}[1]{\mathfrak{#1}}

\newcommand{\mc}[1]{\mathcal{#1}}

\newcommand{\ul}[1]{\underline{#1}}

\newcommand{\as}{\hspace{5mm}\text{as}\hspace{5mm}}
\newcommand{\et}{\hspace{5mm}\text{and}\hspace{5mm}}
\newcommand{\hs}[1]{\hspace{#1}}
\newcommand{\on}{\hspace{5mm}\text{on}\hspace{5mm}}
\newcommand{\vs}[1]{\vspace{#1}}

\DeclareMathSymbol{\Alpha}{\mathalpha}{operators}{"41}
\DeclareMathSymbol{\Beta}{\mathalpha}{operators}{"42}
\DeclareMathSymbol{\Epsilon}{\mathalpha}{operators}{"45}
\DeclareMathSymbol{\Zeta}{\mathalpha}{operators}{"5A}
\DeclareMathSymbol{\Eta}{\mathalpha}{operators}{"48}
\DeclareMathSymbol{\Iota}{\mathalpha}{operators}{"49}
\DeclareMathSymbol{\Kappa}{\mathalpha}{operators}{"4B}
\DeclareMathSymbol{\Mu}{\mathalpha}{operators}{"4D}
\DeclareMathSymbol{\Nu}{\mathalpha}{operators}{"4E}
\DeclareMathSymbol{\Omicron}{\mathalpha}{operators}{"4F}
\DeclareMathSymbol{\Rho}{\mathalpha}{operators}{"50}
\DeclareMathSymbol{\Tau}{\mathalpha}{operators}{"54}
\DeclareMathSymbol{\Chi}{\mathalpha}{operators}{"58}
\DeclareMathSymbol{\omicron}{\mathord}{letters}{"6F}

\newcommand{\al}{\alpha}
\newcommand{\be}{\beta}
\newcommand{\ga}{\gamma}
\newcommand{\de}{\delta}
\newcommand{\ep}{\varepsilon}

\newcommand{\ze}{\zeta}
\renewcommand{\th}{\theta}
\newcommand{\vth}{\vartheta}
\newcommand{\io}{\iota}
\newcommand{\ka}{\kappa}
\newcommand{\la}{\lambda}
\newcommand{\vpi}{\varpi}
\newcommand{\rh}{\rho}

\newcommand{\si}{\sigma}

\newcommand{\ta}{\tau}
\newcommand{\up}{\upsilon}
\newcommand{\ph}{\phi}
\newcommand{\vph}{\upvarphi}

\newcommand{\ch}{\chi}

\newcommand{\om}{\omega}

\newcommand{\Ga}{\Gamma}
\newcommand{\De}{\Delta}

\newcommand{\Th}{\Theta}

\newcommand{\La}{\Lambda}

\newcommand{\Si}{\Sigma}

\newcommand{\Ph}{\Phi}

\newcommand{\Om}{\Omega}

\newcommand{\<}{\langle}
\newcommand{\?}{\rangle}

\newcommand{\codim}{\operatorname{codim}}

\newcommand{\ds}{\oplus}
\newcommand{\End}{\operatorname{End}}

\newcommand{\Id}{\operatorname{Id}}

\newcommand{\Ker}{\operatorname{Ker}}

\newcommand{\lqt}[2]{\left.\raisebox{-1mm}{\(#2\)}\middle\backslash\raisebox{1mm}{\(#1\)}\right.}

\newcommand{\rqt}[2]{\left.\raisebox{1mm}{\(#1\)}\middle/\raisebox{-1mm}{\(#2\)}\right.}
\newcommand{\ts}{\otimes}

\newcommand{\x}{\times}

\newcommand{\del}{\partial}

\newcommand{\Hocl}[1]{\overset{\circ}{H^{#1}}_{\kern-1.9mm\cl}}

\newcommand{\lop}{\left\|\kern-1.30mm\left\|}
\newcommand{\op}{\|\kern-1.30mm\|}
\newcommand{\rop}{\right\|\kern-1.30mm\right\|}
\newcommand{\SI}{\operatorname{\cal{I}\kern-1.5pt nd}}

\newcommand{\cc}{\subseteq}

\newcommand{\es}{\emptyset}

\newcommand{\mt}{\mapsto}

\newcommand{\osr}{\backslash}
\newcommand{\oto}[1]{\xrightarrow{#1}}
\newcommand{\pc}{\subset}

\newcommand{\yy}{\supseteq}

\newcommand{\B}{\mathrm{B}}

\newcommand{\CH}{\mathcal{H}_3}

\newcommand{\cl}{\mathrm{closed}}

\newcommand{\dd}{\mathrm{d}}
\newcommand{\diam}{\operatorname{diam}}

\newcommand{\dR}[1]{H^{#1}_{\operatorname{dR}}}
\newcommand{\E}{\mathrm{Eucl}}

\newcommand{\emb}{\hookrightarrow}

\newcommand{\hk}{\righthalfcup}
\newcommand{\Hol}{\operatorname{Hol}}
\newcommand{\Hs}{\raisebox{1pt}{\mss{\bigstar}}}
\newcommand{\itr}[1]{\overset{\circ}{#1}}

\newcommand{\M}{\mathrm{M}}

\newcommand{\OP}[2]{\mathcal{O}_{\mathbb{CP}^{#1}}(#2)}

\newcommand{\s}{\odot}

\renewcommand{\ss}[2][{}]{\bigodot{\hspace{-1mm}}^{#2}_{#1}\hspace{0.6mm}}

\newcommand{\T}{\mathrm{T}}
\newcommand{\tl}{{\mns{\sim}}}

\newcommand{\w}{\wedge}
\newcommand{\ww}[2][{}]{\bigwedge{\hspace{-1mm}}^{#2}_{\hspace{1mm}#1}\hspace{0.1mm}}

\newcommand{\0}{\infty}
\newcommand{\1}{\cdot}

\newcommand{\ca}[1]{\breve{#1}}
\renewcommand{\ge}{\geqslant}
\newcommand{\gl}{\hspace{0.4mm}\raisebox{0.8mm}{\(>\)}\kern-1.8mm\raisebox{-0.8mm}{\(<\)}\hspace{0.4mm}}
\newcommand{\gle}{\hspace{0.4mm}\raisebox{1.2mm}{\(\ge\)}\kern-1.8mm\raisebox{-1.2mm}{\(\le\)}\hspace{0.4mm}}
\renewcommand{\le}{\leqslant}
\newcommand{\pt}{\bullet}

\newcommand{\g}{\(\mathrm{G}_2\)}
\newcommand{\Gg}{\mathrm{G}}
\newcommand{\GL}{\operatorname{GL}}

\newcommand{\SL}{\operatorname{SL}}

\newcommand{\SO}{\operatorname{SO}}
\newcommand{\Sp}{\operatorname{Sp}}
\newcommand{\Spin}{\operatorname{Spin}}
\newcommand{\Stab}{\operatorname{Stab}}
\newcommand{\SU}{\operatorname{SU}}

\newcommand{\rmm}{Riemannian metric}

\newcommand{\sqfs}{stratified quasi-Finslerian structure}
\newcommand{\srm}{stratified Riemannian metric}
\newcommand{\srsm}{stratified Riemannian semi-metric}

\newcommand{\Wlg}{Without loss of generality}
\newcommand{\wlg}{without loss of generality}
\newcommand{\Wrt}{With respect to}
\newcommand{\wrt}{with respect to}

\newcommand{\EH}{Eguchi--Hanson}
\newcommand{\GH}{Gromov--Hausdorff}

\newcommand{\ol}[1]{\overline{#1}}
\newcommand{\h}{\widehat}
\newcommand{\lt}{\left}
\newcommand{\m}{\middle}

\newcommand{\rt}{\right}

\newcommand{\tld}{\widetilde}

\title[Unboundedness above of the Hitchin functional on \g\ 3-forms and collapsing results]{Unboundedness above of the Hitchin functional on \g\ 3-forms and associated collapsing results}
\author{Laurence H. Mayther}

\begin{document}\fontsize{10pt}{12pt}\selectfont
\begin{abstract}
\footnotesize{This paper uses scaling arguments to prove the unboundedness above of the Hitchin functional on closed \g\ 3-forms for two explicit closed 7-manifolds.  The first manifold is the product \(X \x S^1\) (where \(X\) is the Nakamura manifold constructed by de Bartolomeis--Tomassini) equipped with a 4-dimensional family of closed \g\ 3-forms and is inspired by a short paper of Fern\'{a}ndez.  The second is the manifold recently constructed by Fern\'{a}ndez--Fino--Kovalev--Muñoz.  In the latter example, careful resolution of singularities is required, in order to ensure that the rescaled forms are cohomologically constant.  By combining suitable geometric estimates with a general collapsing theorem for orbifolds recently obtained by the author, explicit descriptions of the large volume limits of both manifolds are also obtained.  The proofs in this paper are notable for not requiring explicit solution of the Laplacian flow evolution PDE for closed \g-structures, thereby allowing treatment of manifolds which lack the high degree of symmetry generally required for Laplacian flow to be explicitly soluble.}
\end{abstract}
\maketitle

\section{Introduction}

Let \(\bb{A}\) be an oriented, 7-dimensional, real vector space.  A 3-form \(\vph \in \ww{3}\bb{A}^*\) is termed a \g\ 3-form (or, simply, of \g-type) if there exists a correctly oriented basis \(\lt(\th^1,...,\th^7\rt)\) of \(\bb{A}^*\) such that:
\ew
\vph = \th^{123} + \th^{145} + \th^{167} + \th^{246} - \th^{257} - \th^{347} - \th^{356}.
\eew
In this case \(\Stab_{\GL_+(\bb{A})}(\vph) \cong \Gg_2\).  The set \(\ww[+]{3}\bb{A}^*\) of \g\ 3-forms on \(\bb{A}\) is open in \(\ww{3}\bb{A}^*\) and thus \g\ 3-forms are stable, in the sense of \cite{SF&SM}.  Now let \(\M\) be an oriented 7-manifold.  A 3-form \(\ph \in \Om^3(\M)\) is termed a \g\ 3-form if, for all \(x \in \M\), \(\ph|_x \in \ww[+]{3}\T^*_x\M\).  Such \(\ph\) is equivalent to a \g-structure on \(\M\) and thus, since \(\Gg_2 \pc \SO(7)\), \(\ph\) induces a metric \(g_\ph\) and volume form \(vol_\ph\) on \(\M\).

Now suppose that \(\M\) is closed and \(\dd\ph = 0\), and write \([\ph]_+\) for the subset of the de Rham class of \(\ph\) consisting of \g\ 3-forms; note that \([\ph]_+ \pc [\ph]\) is open in the \(C^0\)-topology, by the stability of \g\ 3-forms.  The Hitchin functional on \([\ph]_+\), defined in \cite{TGo3Fi6&7D}, is the map:
\ew
\bcd[row sep = 0pt]
\CH: [\ph]_+ \ar[r] & (0,\0)\\
\ph' \ar[r, maps to] & \bigint_\M vol_{\ph'}
\ecd
\eew
Interest in the functional \(\CH\) stems primarily from the observation that its critical points are precisely those \(\ph' \in [\ph]_+\) which are torsion-free, i.e.\ which satisfy \(\dd \Hs_{\ph'}\ph' = 0\), a condition which is equivalent to \(\Hol(g_{\ph'}) \cc \Gg_2\) by a well-known result of Fern\'{a}ndez--Gray \cite{RMwSGG2}.  In \cite{RoG2S}, Bryant used this observation to propose a new possible construction of metrics with holonomy contained in \g\ on closed 7-manifolds, via the gradient flow of the functional \(\CH\) (known as Laplacian flow).  Specifically, if a flow line existed for all time, and the functional \(\CH\) were bounded above, then one might hope that, as \(t \to \0\), the flow line would converge to a torsion-free \g\ 3-form and hence produce a metric with holonomy contained in \g.  This led Bryant to pose the general question of whether or not the functional \(\CH\) is bounded above.

At present, answering this question in general appears intractable, however for manifolds which cannot admit torsion-free \g\ 3-forms, the question can be more tractable (see, e.g.\ \cite[\S6, Ex.\ 2]{RoG2S}).  This paper considers two, specific examples of such manifolds and proves that, in each case, the functional \(\CH\) is unbounded above.

The first example is inspired by Fern\'{a}ndez' short paper \cite{AFoCSG2CM} and consists of the closed 7-manifold \(N = X \x S^1\), where \(X\) is the Nakamura manifold constructed by de Bartolomeis--Tomassini \cite{OSGCYM}, equipped with a 4-dimensional family of closed \g\ 3-forms \(\ph(\al,\be,\la)\) on \(N\) described in \sref{N-mfld}, where \((\al,\be,\la) \in \lt(\bb{R} \osr \{0\}\rt)^2 \x \lt(\bb{C}\osr\{0\}\rt)\).

\begin{Thm}\label{UA}
The map:
\ew
\bcd[row sep = 0pt]
(\bb{R} \osr \{0\})^2 \x (\bb{C}\osr\{0\}) \ar[r] & \dR{3}(N)\\
(\al,\be,\la) \ar[r, maps to] & \lt[\ph(\al,\be,\la)\rt]
\ecd
\eew
is injective and, for all \((\al,\be, \la) \in (\bb{R} \osr \{0\})^2 \x (\bb{C}\osr\{0\})\), the functional:
\ew
\CH:\lt[\ph(\al,\be,\la)\rt]_+ \to (0,\infty)
\eew
is unbounded above.
\end{Thm}

The proof of \tref{UA} uses a direct scaling argument, described in \sref{UB-Md}, to construct explicit families \(\ph(\al,\be,\la;\mu)_{\mu \in [1,\0)}\) of closed \g\ 3-forms in the classes \([\ph(\al,\be,\la)]_+\) with unbounded volume.  The large volume limits of these families are themselves of independent interest.  Indeed, it is known that \(X\) splits as a product \(X \cong X' \x S^1\).  Writing \(\ell = \log \frac{3 + \sqrt{5}}{2}\) for the parameter introduced in \cite{OSGCYM}, there is a natural fibration \(X' \to \rqt{\bb{R}}{\ell\bb{Z}}\) which lifts to define a fibration \(\fr{p}: N \cong X' \x S^1 \x S^1 \to \rqt{\bb{R}}{\ell\bb{Z}}\).

\begin{Thm}\label{CT}
Let \((\al,\be,\la) \in (\bb{R}\osr\{0\})^2 \x (\bb{C}\osr\{0\})\) and let \((N,\ph(\al,\be,\la;\mu))_{\mu \in [1,\infty)}\) be the family constructed in the proof of \tref{UA}.  Then the large volume limit of \((N,\ph(\al,\be,\la;\mu))\) corresponds to an adiabatic limit of the fibration \(\fr{p}\).  Specifically:
\ew
(N,\mu^{-12}\ph(\al,\be,\la;\mu)) \to \lt(\rqt{\bb{R}}{\ell\bb{Z}}, \al^2 \lt(\la\ol{\la}\rt)^{-\frac{2}{3}} g_\E\rt) \as \mu \to \infty,
\eew
where the convergence is in the Gromov--Hausdorff sense.
\end{Thm}

I remark that \tref{CT} also provides evidence that the functional \(\CH\) can be of significant geometric interest, even on manifolds which do not admit torsion-free \g\ 3-forms.  To the author's knowledge, \tref{CT} is one of the first illustrations of this possibility to have appeared in the literature.

The second example considered in this paper is the manifold \(\ca{\M}\) with closed \g\ 3-form \(\ca{\ph}\) constructed by Fern\'{a}ndez--Fino--Kovalev--Muñoz in \cite{ACG2CMwFBNb1eq1} and described in \sref{FFKM-constr}:
\begin{Thm}\label{FFKM-UA}
The Hitchin functional:
\ew
\CH:\lt[\ca{\ph}\rt]_+ \to (0,\infty)
\eew
is unbounded above.
\end{Thm}

In comparison to \(N\), the construction of \(\ca{\M}\) is significantly more complicated and requires a `resolution of singularities' argument, inspired by Joyce's Kummer-type construction of \g-manifolds \cite{CR7MwHG2I, CR7MwHG2II, CMwSH}.  The more complicated nature of the manifold \(\ca{\M}\) means that, to the author's knowledge, the unboundedness of \(\CH\) on the space \(\lt[\ca{\ph}\rt]_+\) cannot be proved using any existing techniques in the literature.

As for the manifold \(N\), the proof of \tref{FFKM-UA} produces an explicit family \(\lt(\ca{\ph}^\mu\rt)_{\mu \in [0,\0)}\) of \g\ 3-forms in \(\lt[\ca{\ph}\rt]_+\) with unbounded volume and, again, the large volume limit of the family \(\lt(\ca{\M},\ca{\ph}^\mu\rt)\) can be described explicitly.  By \cite[\S9]{ACG2CMwFBNb1eq1}, there is a natural fibration \(\pi: \ca{\M} \to \lt(\lqt{\bb{T}^2}{\{\pm1\}}\rt) \x S^1\).

\begin{Thm}\label{FFKM-CT}
Let \(\lt(\ca{\M},\ca{\ph}^\mu\rt)_{\mu \in [1,\infty)}\) be the family constructed in the proof of \tref{FFKM-UA}.  Then the large volume limit of \(\lt(\ca{\M}, \ca{\ph}^\mu\rt)\) corresponds to an adiabatic limit of the fibration \(\pi\).  Specifically, let \(B\) denote the orbifold \(\lt(\lqt{\bb{T}^2}{\{\pm1\}}\rt) \x S^1\).  Then:
\ew
\lt(\ca{\M}, \mu^{-6}\ca{\ph}^\mu\rt) \to (B,d) \as \mu \to \infty
\eew
in the Gromov--Hausdorff sense, for a suitable metric (i.e.\ distance function) \(d\) on \(B\).
\end{Thm}

For an explicit description of the metric \(d\), the reader is directed to \tref{FFKM-CT2} (see also \eref{mcL}).  Essential to the proof of \tref{FFKM-CT} is a general collapsing result for families of Riemannian metrics on closed manifolds (or, more generally, orbifolds) proven by the author in \cite{AGCRfOvSF}.  I, also, remark that \tref{FFKM-CT} again provides evidence that the functional \(\CH\) can be of significant geometric interest, even on manifolds which do not admit torsion-free \g\ 3-forms.

The basic principle underlying the proofs of both \trefs{UA} and \ref{FFKM-UA} is a direct scaling argument, described in \sref{UB-Md} (see, in particular, \pref{UB-Md-Prop}); the argument does not rely on the solubility of the Laplacian flow equation (see \rref{LF-Rk}).  To the author's knowledge, this scaling argument has not been adopted previously in the literature and it is expected that this technique can be used to prove the unboundedness above of \(\CH\) on other explicit examples of closed 7-manifolds with closed \g\ 3-forms.

The results of the present paper were obtained during the author's doctoral studies, which where supported by EPSRC Studentship 2261110.

\section{Unboundedness of \(\CH\) via scaling}\label{UB-Md}

I begin with an algebraic lemma:
\begin{Lem}\label{scaling-lem}
(1) Let \(\lt(\th^1,...,\th^7\rt)\) denote the canonical basis of \(\lt(\bb{R}^7\rt)^*\) and define:
\ew
\ph =&\ \th^{123} + \th^{145} + \th^{167} + \th^{246} - \th^{257} - \th^{347} - \th^{356}\\
=&\ \ph_1 \hs{1.7mm} + \ph_2 \hs{2.6mm} + \ph_3 \hs{2.6mm} + \ph_4 \hs{2.6mm} + \ph_5 \hs{2.6mm} + \ph_6 \hs{2.6mm} + \ph_7.
\eew
Then for all \(\la_1,...,\la_7 \in (0,\infty)\):
\ew
\ph_{(\la_1,...,\la_7)} = \sum_{i=1}^7 \la_i \ph_i
\eew
is of \g-type and:
\ew
vol_{\ph_{(\la_1,...,\la_7)}} = \lt(\prod_{i=0}^7 \la_i\rt)^\frac{1}{3}vol_\ph.
\eew

(2) Let \(\bb{F}\) be a 4-dimensional real vector space equipped with a complex structure \(J\).  Let \(\om\) be a real, positive \((1,1)\)-form on \(\bb{F}\) and let \(\Om\) be a non-zero complex \((2,0)\)-form on \(\bb{F}\).   Define a constant \(\nu>0\) by the equation:
\e\label{nu}
2\om^2 = \nu^2\Om \w \ol{\Om}.
\ee
Then given a 3-dimensional real vector space \(\bb{G}\) with basis \(\lt(g^1,g^2,g^3\rt)\) of \(\bb{G}^*\), the 3-form on \(\bb{F}\ds\bb{G}\) defined by:
\ew
\ph' = g^{123} + g^1\w\om - g^2\w\fr{Re}\Om + g^3\w\fr{Im}\Om
\eew
is of \g-type.   Moreover:
\caw
g_{\ph'} = \nu^\frac{4}{3}\lt(g^1\rt)^{\ts2} + \nu^{-\frac{2}{3}}\lt[\lt(g^2\rt)^{\ts2} + \lt(g^3\rt)^{\ts2}\rt] + \nu^{-\frac{2}{3}}g_\om;\\
vol_{\ph'} = \frac{\nu^\frac{2}{3}}{4} g^{123} \w \Om \w \ol{\Om};\\
\Hs_{\ph'}\ph' = \frac{\nu^\frac{2}{3}}{4}\Om \w \ol{\Om} + \nu^{-\frac{4}{3}}g^{23} \w \om + \nu^\frac{2}{3} g^{13} \w \fr{Re}\Om + \nu^\frac{2}{3} g^{12} \w \fr{Im}\Om,
\caaw
where \(g_\om\) is the metric on \(\bb{F}\) induced by \(J\) and the real, positive \((1,1)\)-form \(\om\).
\end{Lem}

\begin{proof}
Begin with (1).  Let \(\mu_1,...,\mu_7 \in (0,\infty)\) be chosen later, define \(\vth^i = \mu_i \th^i\) for all \(i\) and consider the \g\ 3-form:
\ew
\vph(\mu_1,...,\mu_7) =&\ \hs{2.1mm} \vth^{123} \hs{2.1mm} + \hs{2.1mm} \vth^{145} \hs{2.1mm} + \hs{2.1mm} \vth^{167} \hs{2.1mm} + \hs{2.1mm} \vth^{246} \hs{2.1mm} - \hs{2.1mm} \vth^{257} \hs{2.1mm} - \hs{2.1mm} \vth^{347} \hs{2.1mm} - \hs{2.1mm} \vth^{356}\\
=&\ \mu_{123}\ph_1 + \mu_{145}\ph_2 + \mu_{167}\ph_3 + \mu_{246}\ph_4 + \mu_{257}\ph_5 + \mu_{347}\ph_6 + \mu_{356}\ph_7,
\eew
where \(\mu_{ijk} = \mu_i\mu_j\mu_k\).  Clearly:
\e\label{scaling-vol}
vol_{\vph(\mu_1,...,\mu_7)} = \mu_{1234567}\th^{1234567} = \mu_{1234567}vol_\ph.
\ee
I claim that \(\vph(\mu_1,...,\mu_7) = \ph_{(\la_1,...,\la_7)}\) for suitable \(\mu_i\).  Indeed, this equation is equivalent to the system of equations:
\cag\label{scaling-sys}
\mu_{123} = \la_1 \hs{5mm} \mu_{145} = \la_2 \hs{5mm} \mu_{167} = \la_3\\
\mu_{246} = \la_4 \hs{5mm} \mu_{257} = \la_5 \hs{5mm} \mu_{347} = \la_6\\
\mu_{356} = \la_7,
\caag
and taking \(\log\) (which is possible since all \(\mu_i\) and \(\la_i\) are positive) yields the invertible linear system:
\e\label{lin-sys}
\bpm
1 & 1 & 1 & 0 & 0 & 0 & 0\\
1 & 0 & 0 & 1 & 1 & 0 & 0\\
1 & 0 & 0 & 0 & 0 & 1 & 1\\
0 & 1 & 0 & 1 & 0 & 1 & 0\\
0 & 1 & 0 & 0 & 1 & 0 & 1\\
0 & 0 & 1 & 1 & 0 & 0 & 1\\
0 & 0 & 1 & 0 & 1 & 1 & 0
\epm
\bpm
\log \mu_1\\
\log \mu_2\\
\log \mu_3\\
\log \mu_4\\
\log \mu_5\\
\log \mu_6\\
\log \mu_7
\epm
=
\bpm
\log \la_1\\
\log \la_2\\
\log \la_3\\
\log \la_4\\
\log \la_5\\
\log \la_6\\
\log \la_7
\epm.
\ee
Taking the product of the equations in \eref{scaling-sys} yields:
\ew
\prod_{i=1}^7 \la_i = \lt(\prod_{i=1}^7 \mu_i\rt)^3.
\eew
The formula for \(vol_{\ph_{(\la_1,...,\la_7)}}\) now follows from \eref{scaling-vol}.

Now let \(\bb{F}\), \(\bb{G}\), \(J\), \(\om\), \(\Om\) and \(\ph'\) be as in (2).  Since \(\om\) is a positive \((1,1)\)-form and \(\Om\) is a non-zero \((2,0)\)-form \wrt\ \(J\), one can choose a basis \(\lt(f^1,f^2,f^3,f^4\rt)\) of \(\bb{F}^*\) such that \(f^1 + if^2\) and \(f^3 + if^4\) are \((1,0)\)-forms \wrt\ \(J\) and:
\ew
\om = f^{12} + f^{34} \et \Om = \nu^{-1} \lt(f^1 + if^2\rt)\w\lt(f^3 + if^4\rt),
\eew
where \(\nu\) is defined in \eref{nu}.  Consider the (correctly oriented) basis:
\ew
\lt(\th^1, \th^2, \th^3, \th^4, \th^5, \th^6, \th^7\rt) = \lt(f^1, -f^2, -g^1, -g^2, -f^3, -f^4, -g^3\rt)
\eew
of \(\lt(\bb{F} \ds \bb{G}\rt)^*\); then \wrt\ this basis:
\ew
\ph' = \th^{123} + \nu^{-1}\th^{145} + \nu^{-1}\th^{167} + \nu^{-1}\th^{246} - \nu^{-1}\th^{257} - \th^{347} - \th^{356}.
\eew
This is of \g-type by (1).  The explicit formulae for \(g_\ph\), \(vol_\ph\) and \(\Hs_\ph\ph\) follow by solving the linear system in \eref{lin-sys} explicitly to obtain the `\g\ basis':
\ew
\lt(\vth^1, \vth^2, \vth^3, \vth^4, \vth^5, \vth^6, \vth^7\rt) = \lt(\nu^{-\frac{1}{3}}\th^1, \nu^{-\frac{1}{3}}\th^2, \nu^\frac{2}{3}\th^3, \nu^{-\frac{1}{3}}\th^4, \nu^{-\frac{1}{3}}\th^5, \nu^{-\frac{1}{3}}\th^6, \nu^{-\frac{1}{3}}\th^7\rt).
\eew

\end{proof}

Applying \lref{scaling-lem} to manifolds yields the following unboundedness result for the functional \(\CH\):

\begin{Prop}\label{UB-Md-Prop}
(1) Let \(\M\) be a closed, parallelisable 7-manifold, let \(\th^1,...,\th^7\) be a basis of 1-forms and let \(\ph\) be the \g\ 3-form:
\ew
\ph =&\ \th^{123} + \th^{145} + \th^{167} + \th^{246} - \th^{257} - \th^{347} - \th^{356}\\
=&\ \ph_1 \hs{1.7mm} + \ph_2 \hs{2.6mm} + \ph_3 \hs{2.6mm} + \ph_4 \hs{2.6mm} + \ph_5 \hs{2.6mm} + \ph_6 \hs{2.6mm} + \ph_7.
\eew
Suppose that \(\dd\ph = 0\) and that there exists \(I \cc \{1,...,7\}\) such that \(\ph_I = \sum_{i \in I} \ph_i\) is exact.  Then for all \(\la \ge 0\), \(\ph(\la) = \ph + \la\ph_I\) is a closed \g\ 3-form in the same cohomology class as \(\ph\) satisfying \(\CH(\ph(\la)) = (1 + \la)^{\frac{|I|}{3}}\CH(\ph)\).  In particular:
\ew
\sup_{\ph' \in [\ph]_+} \CH(\ph') = \infty.
\eew

(2) Let \(\M\) be a closed, oriented 7-manifold, let \(\T\M \cong \bb{R}^3 \ds \mc{F}\) for some rank 4 distribution \(\mc{F}\) on \(\M\) (such a splitting always exists by \cite[Table 1]{KTOttEoVF}), let \((g^1,g^2,g^3)\) be a basis of 1-forms for the trivial bundle \(\lt(\bb{R}^3\rt)^* \pc \lt(\bb{R}^3 \ds \mc{F}\rt)^* \cong \T^*\M\), let \(J\) be a section of \(\End(\mc{F})\) satisfying \(J^2 = -\Id\), let \(\om\) be a real, positive \((1,1)\)-form on \(\mc{F}\) and \(\Om\) be a non-zero \((2,0)\)-form on \(\mc{F}\) (\wrt\ \(J\)), and let \(\ph'\) be the \g\ 3-form:
\ew
\ph' = g^{123} + g^1\w\om - g^2\w\fr{Re}\Om + g^3\w\fr{Im}\Om.
\eew
Suppose that \(\dd\ph' = 0\) and that \(g^1 \w \om\) is exact.  Then for all \(\la \ge 0\), \(\ph'(\la) = \ph' + \la g^1 \w \om\) is a closed \g\ 3-form in the same cohomology class as \(\ph'\), satisfying \(\CH(\ph'(\la)) = (1 + \la)^\frac{2}{3}\CH(\ph')\).  In particular:
\ew
\sup_{\ph'' \in [\ph']_+} \CH(\ph') = \infty.
\eew

\noindent Likewise, if \(g^2 \w \fr{Re}\Om - g^3 \w \fr{Im}\Om\) is exact, then for all \(\la > 0\), \(\ph''(\la) = \ph' - \la\lt(g^2 \w \fr{Re}\Om - g^3 \w \fr{Im}\Om\rt)\) is a closed \g\ 3-form in the same cohomology class as \(\ph'\), satisfying \(\CH(\ph''(\la)) = (1 + \la)^\frac{4}{3}\CH(\ph')\), so that once again:
\ew
\sup_{\ph'' \in [\ph']_+} \CH(\ph') = \infty.
\eew
\end{Prop}
~

\section{The unboundedness above of \(\CH\) on \(\lt(N,\ph(\al,\be,\la)\rt)\)}\label{N-mfld}

I begin by recalling the construction of the Nakamura manifold \(X\) from \cite{OSGCYM}.  Define a product \(*\) on \(\bb{C}^3\) via the formula:
\ew
\lt(u^1,u^2,u^3\rt) * \lt(w^1,w^2,w^3\rt) = \lt(u^1 + w^1, e^{-w^1}u^2 + w^2, e^{w^1}u^3 + w^3\rt).
\eew
Then \(\lt(\bb{C}^3,*\rt)\) is a complex, soluble, non-nilpotent Lie group, which I denote \(H\).  Equivalently, one may identify:
\e\label{H}
H = \lt\{ \bpm e^{w^1} & 0 & 0 & 0\\ 0 & e^{-w^1} & 0 & 0\\ 0 & 0 & 1 & 0\\ w^3 & w^2 & w^1 & 1 \epm \in \SL(4;\bb{C}) ~\m|~ (w^1,w^2,w^3) \in \bb{C}^3 \rt\}.
\ee
A basis of (complex) right-invariant 1-forms on \(H\) is given by:
\ew
\Th^1 = \dd w^1, \hs{5mm} \Th^2 = e^{w^1}\dd w^2 \et \Th^3 = e^{-w^1}\dd w^3.
\eew
Let \(\ell = \log \frac{3 + \sqrt{5}}{2}\), \(m = \frac{\sqrt{5} - 1}{2}\) and define \(\De \pc H\) to be the uniform (i.e.\ discrete and co-compact) subgroup generated by the six elements\footnote{These formulae differ from those in \cite{OSGCYM}, as the author of this paper has discovered an error {\it op.\ cit.}, which has been corrected in the formulae here presented.}:
\begin{alignat*}{3}
e_1 = (\ell,&0,0), & \hs{5mm} e_2 = (2\pi i, 0, &0), & \hs{5mm} e_3 = (0, &-m, 1)\\
e_4 = (0, &1, m), & \hs{5mm} e_5 = (0, -2\pi i m, &2\pi i), & \hs{5mm} e_6 = (0, &2\pi i, 2\pi i m).
\end{alignat*}
The quotient \(X = \rqt{H}{\De}\) is a compact soluble manifold called a Nakamura manifold (the first such examples being constructed by Nakamura \cite{CPM&tSD}).  Explicitly, write \(P = \bpm -m & 1\\ 1 & m \epm\); then clearly:
\e\label{4-torus}
\rqt{H}{\<e_2, e_3,e_4,e_5,e_6\?} &\cong \rqt{\bb{C}}{2\pi i\bb{Z}} \x \rqt{\bb{C}^2}{P\lt(\bb{Z}^2 + 2\pi i\bb{Z}^2\rt)}\\
&\cong \lt(\bb{R} \x S^1\rt) \x \fr{T},
\ee
where \(\fr{T}\) denotes the complex torus \(\rqt{\bb{C}^2}{P\lt(\bb{Z}^2 + 2\pi i\bb{Z}^2\rt)}\).  Moreover, from the equation:
\ew
P \bpm 2 & 1\\ 1 & 1 \epm P^{-1} = \bpm e^{-\ell} & 0\\ 0 & e^\ell \epm
\eew
it follows that the linear map \(\bpm e^{-\ell} & 0\\ 0 & e^\ell \epm\) on \(\bb{C}^2\) descends to define a map \(\La\) on \(\fr{T}\).  One can then write:
\ew
X = \lt(\rqt{\lt(\bb{R} \x \fr{T}\rt)}{\<T\?}\rt) \x S^1
\eew
where \(T\) is the automorphism given by \((w,p) \in \bb{R} \x \fr{T} \mt (w + \ell, \La(p)) \in \bb{R} \x \fr{T}\).  The right-invariant 1-forms \(\Th^i\) descend to a basis of (complex) 1-forms on \(X\), again denoted \(\Th^i\), which satisfy:
\e\label{dRI}	
\dd\Th^1 = 0, \hs{5mm} \dd\Th^2 = \Th^1 \w \Th^2, \hs{5mm} \dd\Th^3 = -\Th^1 \w \Th^3.
\ee
Write \(g^1 = \fr{Re}\Th^1\) and \(g^2 = \fr{Im}\Th^1\), so that, in particular, \(\dd g^1 = \dd g^2 = 0\).

Now consider the manifold \(N = X \x S^1\).  I begin by establishing the following result, which is not proved in the literature, but which nevertheless appears known to some authors (cf.\ \cite[\S6, Ex.\ 2]{RoG2S}):
\begin{Prop}
The manifold \(N\) admits no torsion-free \g-structures.
\end{Prop}

\begin{proof}
The argument is largely topological in nature.   It follows from \cite[Thm.\ 4.1]{OSGCYM} that \(b^1(N) = 3\).  Thus, since metrics with holonomy contained in \g\ are automatically Ricci-flat \cite[Prop.\ 11.8]{RG&HG}, if \(N\) admitted a torsion-free \g-structure then, by applying Bochner's technique as in \cite[Thm.\ 3.5.4 and 3.5.5]{CMwSH}, the universal cover of \(N\) would be homeomorphic to \(\bb{R}^3\x\tld{N}\) for some simply connected, closed 4-manifold \(\tld{N}\).  However the universal cover of \(X\) is \(H \overset{homeo}{\cong} \bb{C}^3\) (a fact which holds more generally for any complex soluble manifold \cite[p.\ 86]{CPM&tSD}) and thus the universal cover of \(N\) is \(\bb{R}^7\), not \(\bb{R}^3\x\tld{N}\).  Thus no torsion-free \g-structures on \(N\) can exist.

\end{proof}

\(N\) does, however, admit closed \g\ 3-forms.  Consider the (complex) rank 2 distribution on \(X\) given by \(\mc{F} = \Ker_{\bb{C}} \Th^1\) and define \((1,1)\) and \((2,0)\)-forms on \(\mc{F}\) by:
\ew
\om = \frac{i}{2}\lt[\Th^2 \w \ol{\Th}^2 + \Th^3 \w \ol{\Th}^3\rt], \hs{5mm} \rh = \frac{i}{2}\lt[\Th^2 \w \ol{\Th}^2 - \Th^3 \w \ol{\Th}^3\rt] \et \Om = \Th^2 \w \Th^3.
\eew
Then \(\om\) is real and positive, \(\rh\) is real and \((\om,\rh,\Om)\) satisfies:
\e\label{dSU2}
\dd\om = 2 g^1 \w \rh, \hs{5mm} \dd\rh = 2g^1 \w \om \et \dd\Om = 0.
\ee
Write \(g^3\) for the canonical 1-form on \(S^1\).  For each \(\al \in \bb{R}\osr\{0\}\), \(\be \in \bb{R}\osr\{0\}\) and \(\la \in \bb{C}\osr\{0\}\) define a 3-form on \(N\) by:
\ew
\ph(\al,\be,\la) = \al\be g^{123} + \al g^1 \w \om - \be g^2 \w \fr{Re}\la\Om + g^3 \w \fr{Im}\la\Om.
\eew
\(\ph(\al,\be,\la)\) defines a \g-structure on \(N\), by \lref{scaling-lem} (applied to the forms \(\al g^1, \be g^2, g^3, \om, \la\Om\)).   Moreover \(\dd\ph(\al,\be,\la) = 0\), by \eref{dSU2}.

\begin{Rk}
The construction of \(\ph(\al,\be,\la)\) above was inspired by the \g\ 3-forms defined by Fern\'{a}ndez \cite{AFoCSG2CM}.  Indeed, in the case \(\la = 1\) and \(\al = \be \in \lt\{\ka \in \bb{R}~\m|~e^\frac{1}{\ka} + e^{-\frac{1}{\ka}} \in \bb{Z}\osr\{2\}\rt\}\), the forms \(\ph(\al,\be,\la)\) are closely related to Fernandez' definition: the key differences are firstly that, for \(H\) as defined above, Fern\'{a}ndez considers a left-quotient of \(H\) and thus constructs left-invariant \g\ 3-forms rather than right-invariant \g\ 3-forms, and secondly that Fern\'{a}ndez reverses the roles of \(g^2\) and the canonical 1-form on \(S^1\); this arises since \cite{AFoCSG2CM} uses the opposite convention for the orientation of \g-structures to the one used in this paper; see \cite[\S2.1]{SNoG2aS7G} for a discussion of the two conventions.  Also, Fern\'{a}ndez' treatment in \cite{AFoCSG2CM} focuses almost exclusively on the manifold \(N\) from the perspective of real differential geometry, and thus does not notice the natural \(\SU(2)\)-structure (with torsion) underlying the construction {\it op.\ cit.}.
\end{Rk}

I now prove \tref{UA}.  Recall the statement of the theorem:\vs{3mm}

\noindent{\bf Theorem \ref{UA}.}
\em The map:
\ew
\bcd[row sep = 0pt]
(\bb{R} \osr \{0\})^2 \x (\bb{C}\osr\{0\}) \ar[r] & \dR{3}(N)\\
(\al,\be,\la) \ar[r, maps to] & \lt[\ph(\al,\be,\la)\rt]
\ecd
\eew
is injective and, for all \((\al,\be, \la) \in (\bb{R} \osr \{0\})^2 \x (\bb{C}\osr\{0\})\), the functional:
\ew
\CH:\lt[\ph(\al,\be,\la)\rt]_+ \to (0,\infty)
\eew
is unbounded above.\vs{2mm}

\em\begin{proof}
Since each of \(\lt(\fr{Re}\Om\rt)^2\), \(g^{13} \w \fr{Re}\Om\), \(-g^{13} \w \fr{Im}\Om\), \(g^{12} \w \fr{Re}\Om\) and \(g^{12} \w \fr{Im}\Om\) are closed, there is a map:
\ew
\bcd[row sep = 0pt]
\dR{3}(N) \ar[r, ^^22 \ch ^^22] & \bb{R}^5\\
{[\xi]} \ar[r, maps to] & \bpm \bigintsss_N \xi \w \lt(\fr{Re}\Om\rt)^2 \vs{1mm} \\ \bigintsss_N \xi \w g^{13} \w \fr{Re}\Om \vs{1mm} \\ \bigintsss_N \xi \w -g^{13} \w \fr{Im}\Om \vs{1mm} \\ \bigintsss_N \xi \w g^{12} \w \fr{Im}\Om \vs{1mm} \\ \bigintsss_N \xi \w g^{12} \w \fr{Re}\Om \epm
\ecd
\eew
A direct calculation shows that, writing \(A = \bigintsss_N g^{123} \w \lt(\fr{Re}\Om\rt)^2 = \bigintsss_N g^{123} \w \lt(\fr{Im}\Om\rt)^2 > 0\), one has:
\ew
\ch\lt([\ph(\al,\be,\la)]\rt) = A\cdot\bpm \al\be \\ \fr{Re}\la \\ \fr{Im}\la \\ \be\fr{Re}\la \\ \be\fr{Im}\la \epm.
\eew
Thus the composite \((\bb{R} \osr \{0\})^2 \x (\bb{C}\osr\{0\}) \oto{\ph(-,-,-)} \dR{3}(N) \oto{\ch} \bb{R}^5\) is injective and hence so too is \((\bb{R} \osr \{0\})^2 \x (\bb{C}\osr\{0\}) \oto{\ph(-,-,-)} \dR{3}(N)\).

Finally, \(g^1 \w \om = \frac{1}{2}\dd \rh\) is exact, by \eref{dSU2}.  Thus the unboundedness of \(\CH\) on the classes \(\lt[\ph(\al,\be,\la)\rt]_+\) follows immediately from \pref{UB-Md-Prop}(2); in particular, writing:
\ew
\ph(\al,\be,\la;\mu) = \al\be g^{123} + \al \mu^6 g^1 \w \om - \be g^2 \w \fr{Re}\la\Om + g^3 \w \fr{Im}\la\Om
\eew
for \(\mu \ge 1\), \pref{UB-Md-Prop}(2) shows that:
\ew
\CH\lt(\ph(\al,\be,\la;\mu)\rt) = \mu^4\CH\lt(\ph(\al,\be,\la)\rt) \to \infty \as \mu \to \infty,
\eew
completing the proof.

\end{proof}

\section{The large volume limit of \(\lt(N,\ph(\al,\be,\la;\mu)\rt)\)}

The aim of this section is to describe the geometry of \(\lt(N,\ph(\al,\be,\la;\mu)\rt)\) as \(\mu \to \0\).  Recall the group \(H\) defined in \eref{H} and the uniform subgroup \(\De\).  The subgroup \(K \pc H\) corresponding to \(w^1 = 0\) is a connected, normal, Abelian Lie subgroup of \(H\), which is maximal nilpotent since \(H\) is non-nilpotent and \(\codim(K,H) = 1\).  By Mostow's Theorem \cite[p.\ 87]{CPM&tSD}, there is a fibration:
\ew
\fr{f}:X = \rqt{H}{\De} \to \rqt{\lt(\rqt{H}{K}\rt)}{\lt(\rqt{\De \1 K}{K}\rt)}
\eew
with fibre \(\rqt{K}{\De \cap K}\).  Explicitly, recall that \(X \cong \lt(\rqt{\lt(\bb{R} \x \fr{T}\rt)}{\<T\?}\rt) \x S^1\), where \(\fr{T}\) is the 4-torus defined in \eref{4-torus} and \(T\) is the automorphism given by \((w,p) \in \bb{R} \x \fr{T} \mt (w + \ell, \La(p)) \in \bb{R} \x \fr{T}\).  Then \(\fr{f}\) is simply the natural projection:
\ew
\lt(\rqt{\lt(\bb{R} \x \fr{T}\rt)}{\<T\?}\rt) \x S^1 \oto{proj} \lt(\rqt{\bb{R}}{\ell\bb{Z}}\rt) \x S^1,
\eew
with fibre \(\fr{T}\).  Using \(\fr{f}\), define a fibration \(\fr{p}:N \to \rqt{\bb{R}}{\ell\bb{Z}}\) via:
\ew
\fr{p}:N = X \x S^1 \oto{proj_1} X \oto{\fr{f}} \lt(\rqt{\bb{R}}{\ell\bb{Z}}\rt) \x S^1 \oto{proj_1} \rqt{\bb{R}}{\ell\bb{Z}}.
\eew

I now prove \tref{CT}.  Recall the statement of the theorem:\vs{3mm}

\noindent{\bf Theorem \ref{CT}.}
\em Let \((\al,\be,\la) \in (\bb{R}\osr\{0\})^2 \x (\bb{C}\osr\{0\})\) and let \((N,\ph(\al,\be,\la;\mu))_{\mu \in [1,\infty)}\) be the family constructed in the proof of \tref{UA}.  Then the large volume limit of \((N,\ph(\al,\be,\la;\mu))\) corresponds to an adiabatic limit of the fibration \(\fr{p}\).  Specifically:
\ew
(N,\mu^{-12}\ph(\al,\be,\la;\mu)) \to \lt(\rqt{\bb{R}}{\ell\bb{Z}}, \al^2 \lt(\la\ol{\la}\rt)^{-\frac{2}{3}} g_\E\rt) \as \mu \to \infty,
\eew
where the convergence is in the Gromov--Hausdorff sense.\vs{2mm}\em

The proof uses the following convergence result, which is a special case of \cite[Thm.\ 1.2]{AGCRfOvSF} (for brevity of notation, given symmetric bilinear forms \(g\) and \(h\), write \(g \ge h\) to mean that \(g - h\) is non-negative definite):
\begin{Thm}\label{Cvgce-Thm}
Let \(E\) and \(B\) be closed manifolds and let \(\pi: E \to B\) be a submersion.  Let \(g^\mu\) be a family of \rmm s on \(E\) and let \(g\) be a \rmm\ on \(B\).  If \(g^\mu \to \pi^*g\) uniformly and there exist constants \(\La_\mu \ge 0\) such that:
\ew
\lim_{\mu \to \infty} \La_\mu = 1 \et g^\mu \ge \La_\mu^2 \pi^*g \text{ for all } \mu \in [1,\infty),
\eew
then \(\lt(E,g^\mu\rt)\) converges to \(\lt(B,g\rt)\) in the \GH\ sense as \(\mu \to \0\).
\end{Thm}

\begin{proof}[Proof of \tref{CT}]
By \lref{scaling-lem} applied to the forms \(\al g^1, \be g^2, g^3, \mu^6\om, \la\Om\) (so that \(\nu = \frac{\mu^6}{\lt(\la\ol{\la}\rt)^\frac{1}{2}}\)) one may compute that:
\ew
g_{\ph(\al,\be,\la;\mu)} = \frac{\mu^8\al^2}{\lt(\la\ol{\la}\rt)^\frac{2}{3}} \lt(g^1\rt)^{\ts2} + \mu^2\lt(\la\ol{\la}\rt)^\frac{1}{3}g_\om + \frac{\lt(\la\ol{\la}\rt)^\frac{1}{3}}{\mu^4}\lt[\be^2\lt(g^2\rt)^{\ts2} + \lt(g^3\rt)^{\ts2}\rt].
\eew
Rescaling the \g\ 3-forms \(\ph(\al,\be,\la;\mu) \mt \mu^{-12}\ph(\al,\be,\la;\mu)\), one finds that:
\ew
g_{\mu^{-12}\ph(\al,\be,\la;\mu)} &= \frac{\al^2}{\lt(\la\ol{\la}\rt)^\frac{2}{3}} \lt(g^1\rt)^{\ts2} + \frac{\lt(\la\ol{\la}\rt)^\frac{1}{3}}{\mu^6}g_\om + \frac{\lt(\la\ol{\la}\rt)^\frac{1}{3}}{\mu^{12}}\lt[\be^2\lt(g^2\rt)^{\ts2} + \lt(g^3\rt)^{\ts2}\rt]\\
&\to \al^2\lt(\la\ol{\la}\rt)^{-\frac{2}{3}} \lt(g^1\rt)^{\ts2} = \fr{p}^* \lt[\al^2\lt(\la\ol{\la}\rt)^{-\frac{2}{3}} g_\E \rt] \text{ uniformly as } \mu \to \infty,
\eew
where \(g_\E\) denotes the Euclidean metric on \(\rqt{\bb{R}}{\ell\bb{Z}}\).  Moreover:
\ew
g_{\mu^{-12}\ph(\al,\be,\la;\mu)} \ge \fr{p}^* \lt[\al^2\lt(\la\ol{\la}\rt)^{-\frac{2}{3}} g_\E \rt] \hs{3mm} \text{for all } \mu.
\eew
The result now follows at once from \tref{Cvgce-Thm}.

\end{proof}

\section{The unboundedness above of \(\CH\) on \(\lt(\ca{\M},\ca{\ph}\rt)\)}\label{FFKM}

\subsection{Preliminaries: basic properties of orbifolds}

The material in this subsection is largely based on \cite[\S\S1.1--1.3]{O&ST} and \cite[\S14.1]{THKLFPFftScDO}.  Let \(E\) be a topological space.  
\begin{Defn}
An \(n\)-dimensional orbifold chart \(\Xi\) is the data of a connected, open neighbourhood \(U\) in \(E\), a finite subgroup \(\Ga \pc \GL(n;\bb{R})\), a connected, \(\Ga\)-invariant open neighbourhood \(\tld{U}\) of \(0 \in \bb{R}^n\) and a homeomorphism \(\ch: \lqt{\tld{U}}{\Ga} \to U\).  Write \(\tld{\ch}\) for the composite \(\tld{U} \oto{quot} \lqt{\tld{U}}{\Ga} \oto{\ch} U\).  Say that \(\Xi\) is centred at \(e \in E\) if \(e = \tld{\ch}(0)\).   In this case, \(\Ga\) is called the orbifold group of \(e\), denoted \(\Ga_e\).  \(e\) is called a smooth point if \(\Ga_e = 0\), and a singular point if \(\Ga_e \ne 0\).

Now consider two orbifold charts \(\Xi_1 = \lt(U_1,\Ga_1,\tld{U}_1,\ch_1\rt)\) and \(\Xi_2 = \lt(U_2,\Ga_2,\tld{U}_2,\ch_2\rt)\) with \(U_1 \cc U_2\).   An embedding of \(\Xi_1\) into \(\Xi_2\) is the data of a smooth, open embedding \(\io_{12}:\tld{U}_1 \emb \tld{U}_2\) and a group isomorphism \(\la_{12}: \Ga_1 \to \Stab_{\Ga_2}(\io_{12}(0))\) such that for all \(x \in \tld{U}_1\) and all \(\si \in \Ga_1\): \(\io_{12}(\si \cdot x) = \la_{12}(\si) \cdot \io_{12}(x)\), and such that the following diagram commutes:
\ew
\bcd[column sep = 13mm, row sep = 3mm]
\tld{U}_1 \ar[rr, ^^22 \mla{\io_{12}} ^^22] \ar[dd, ^^22 \mla{\tld{\ch}_1} ^^22] & & \tld{U}_2 \ar[dd, ^^22 \mla{\tld{\ch}_2} ^^22]\\
& & \\
U_1 \ar[rr, ^^22\mla{incl}^^22, hook] &  & U_2
\ecd
\eew

Now let \(\Xi_1\) and \(\Xi_2\) be arbitrary.  \(\Xi_1\) and \(\Xi_2\) are compatible, if for every \(e \in U_1 \cap U_2\) there exists a chart \(\Xi_e = \lt(U_e,\Ga_e,\tld{U}_e,\ch_e\rt)\) centred at \(e\) together with embeddings \((\io_{e1},\la_{e1}): \Xi_e \emb \Xi_1\) and \((\io_{e2},\la_{e2}): \Xi_e \emb \Xi_2\).  If \(U_1 \cap U_2 = \es\), then \(\Xi_1\) and \(\Xi_2\) are automatically compatible, however if \(U_1 \cap U_2 \ne \es\) and \(\Xi_1\) and \(\Xi_2\) are compatible, then \(\Xi_1\) and \(\Xi_2\) have the same dimension; moreover, if \(\Xi_1\) and \(\Xi_2\) are centred at the same point \(e \in E\), then \(\Ga_1 \cong \Ga_2\) and therefore the orbifold group \(\Ga_e\) is well-defined up to isomorphism.

An orbifold atlas for \(E\) is a collection of compatible orbifold charts \(\fr{A}\) which is maximal in the sense that if a chart \(\Xi\) is compatible with every chart in \(\fr{A}\), then \(\Xi \in \fr{A}\).  An orbifold is a connected, Hausdorff, second-countable topological space \(E\) equipped with an orbifold atlas \(\fr{A}\).  Every chart of \(E\) has the same dimension \(n\); call this the dimension of the orbifold.
\end{Defn}

Let \(E\) be an \(n\)-orbifold.  Given any chart \(\Xi = \lt(U, \Ga, \tld{U}, \ch\rt)\) for \(E\), the action of \(\Ga\) on \(\tld{U}\) naturally lifts to an action of \(\Ga\) on \(\T\tld{U}\) by bundle automorphisms.  Given a second chart \(\Xi' = (U',\Ga',\tld{U}',\ch')\) embedding into \(\Xi\), the map \(\tld{U}' \emb \tld{U}\) induces an equivariant embedding of bundles \(\T\tld{U}' \emb \T\tld{U}\).  Define:
\ew
\T E = \rqt{\lt[\coprod_\Xi \lt(\lqt{\T\tld{U}}{\Ga}\rt)\rt]}{\tl}
\eew
where the quotient by \(\tl\) denotes that one should glue along the embeddings \(\T\tld{U}' \emb \T\tld{U}\).  The resulting space \(\T E\) is an orbifold, and the natural projection \(\pi: \T E \to E\) gives \(\T E\) the structure of an orbifold vector bundle over \(E\) in the sense of \cite[p.\ 173]{THKLFPFftScDO}.  \(\T E\) is termed the tangent bundle of \(E\).  Given \(e \in E\), the tangent space at \(e\), denoted \(\T_eE\), is the preimage of \(e\) under the map \(\T E \to E\).  It may be identified with the quotient space \(\lqt{\bb{R}^n}{\Ga_e}\), where \(\Ga_e\) is the orbifold group at \(e\).  A section of the tangent bundle is then simply a continuous map \(X: E \to \T E\) such that, for each local chart \(\Xi = \lt(U, \Ga, \tld{U}, \ch\rt)\) for \(E\), there is a smooth, \(\Ga\)-invariant, vector field \(\tld{X}: \tld{U} \to \T \tld{U}\) such that the following diagram commutes:
\ew
\bcd[column sep = 13mm, row sep = 3mm]
\tld{U} \ar[rr, ^^22 \mla{\tld{X}} ^^22] \ar[dd, ^^22 \mla{\tld{\ch}} ^^22] & & \T \tld{U} \ar[dd]\\
& &\\
U \ar[rr, ^^22 \mla{X} ^^22] &  & \pi^{-1}(U)
\ecd
\eew
Note, in particular, that \(\tld{X}|_0\) is invariant under the action of \(\Ga_e\); thus, given a 1-form \(\al \in \T^*_e E \cong \lqt{\bb{R}^n}{\Ga_e}\), the evaluation \(\al\lt(\tld{X}|_0\rt)\) is well-defined.  In a similar way, one may define the cotangent bundle of an orbifold, tensor bundles, bundles of exterior forms etc., with sections of these bundles defined analogously.  \(E\) is orientable if \(\ww{n}\T^* E\) is trivial, i.e.\ has a non-vanishing section (in particular, all the orbifold groups of \(E\) lie in \(\GL_+(n;\bb{R})\)); an orientation is simply a choice of trivialisation.  A \g\ 3-form on an oriented orbifold \(E\) is a section of \(\ww{3}\T^* E\) such that each local representation \(\tld{\ph}\) of \(\ph\) is a \g\ 3-form in the classical (manifold) sense.  In particular, the orbifold group \(\Ga_e\) at each point lies in the stabiliser of \(\tld{\ph}|_0\) and thus is isomorphic to a subgroup of \g.  Riemannian metrics and other geometric structures can be defined similarly.  Given a closed (orbifold) \g\ 3-form \(\ph\), the functional \(\CH: [\ph]_+ \to (0,\0)\) is defined as in the manifold case, by integrating the volume form induced by \(\ph\) over all of \(\M\); see \cite[\S2.1]{O&ST} for further details on de Rham cohomology and integration of forms on orbifolds.

\subsection{The construction of \(\lt(\ca{\M},\ca{\ph}\rt)\)}\label{FFKM-constr}

For full details of the construction, see \cite{ACG2CMwFBNb1eq1}.

Let \(G = \lt\{ \bpm A_1 & 0\\ 0 & A_2 \epm \in \SL(12;\bb{R}) \rt\}\), where, for \((x^1,..,x^7)\in\bb{R}^7\):
\ew
A_1 =
\bpm
1 & -x^2 & x^1 & x^4 & -x^1x^2 & x^6\\
0 & 1 & 0 & -x^1 & x^1 & \frac{1}{2}(x^1)^2\\
0 & 0 & 1 & 0 & -x^2 & -x^4\\
0 & 0 & 0 & 1 & 0 & 0\\
0 & 0 & 0 & 0 & 1 & x^1\\
0 & 0 & 0 & 0 & 0 & 1
\epm
\et
A_2 =
\bpm
1 & -x^3 & x^1 & x^5 & -x^1x^3 & x^7\\
0 & 1 & 0 & -x^1 & x^1 & \frac{1}{2}(x^1)^2\\
0 & 0 & 1 & 0 & -x^3 & -x^5\\
0 & 0 & 0 & 1 & 0 & 0\\
0 & 0 & 0 & 0 & 1 & x^1\\
0 & 0 & 0 & 0 & 0 & 1
\epm.
\eew
Write \(\Ga\pc G\) for the uniform subgroup corresponding to \((x^1,...,x^7)\in 2\bb{Z}\x\bb{Z}^6\) and define \(\M = \lqt{G}{\Ga}\), a closed nilmanifold.  \(G\) admits a basis of left-invariant 1-forms given by:
\cag\label{FFKM1forms}
\th^1 = \dd x^1, \hs{5mm} \th^2 = \dd x^2, \hs{5mm} \th^3 = \dd x^3, \hs{5mm} \th^4 = \dd x^4 - x^2 \dd x^1\\
\th^5 = d x^5 - x^3\dd x^1, \hs{5mm} \th^6 = \dd x^6 + x^1\dd x^4, \hs{5mm} \th^7 = \dd x^7 + x^1\dd x^5
\caag
which descend to define a basis of 1-forms on \(\M\) (also denoted \(\th^i\)) satisfying:
\e\label{FFKMform}
\dd\th^i = 0 \ (i = 1,2,3), \hs{5mm} \dd\th^4 = \th^{12}, \hs{5mm} \dd\th^5 = \th^{13}, \hs{5mm} \dd\th^6 = \th^{14} \et \dd\th^7 = \th^{15}.
\ee
Define a closed \g\ 3-form on \(\M\) by:
\ew
\vph = \th^{123} + \th^{145} + \th^{167} - \th^{246} + \th^{257} + \th^{347} + \th^{356}.
\eew
(N.B.\ a \g-basis for \(\vph\) is given by \(\lt(\th^1,-\th^2,-\th^3,\th^4,\th^5,\th^6,\th^7\rt)\).)  Then \(\M\) admits a (non-free) involution \(\mc{I}\) given by:
\e\label{FFKMInv}
\mc{I}: \Ga\cdot(x^1,x^2,x^3,x^4,x^5,x^6,x^7) \mt \Ga\cdot(-x^1,-x^2,x^3,x^4,-x^5,-x^6,x^7)
\ee
which preserves \(\vph\) and hence \(\vph\) descends to define a closed (orbifold) \g\ 3-form \(\h{\vph}\) on \(\h{\M} = \lqt{\M}{\mc{I}}\).

Let \(\h{S}\) denote the singular locus of \(\h{\M}\) and write \(S\) for the preimage of \(\h{S}\) under the natural projection \(\M \to \h{\M}\). By \eref{FFKMInv} (see \S5 of the arXiv version of \cite{ACG2CMwFBNb1eq1}; the journal version contains an error in this regard) \(S = \coprod_{{\bf a} \in \fr{A}} S_{\bf a}\) where \({\bf a} = \lt(a^1,a^2,a^5,a^6\rt) \in \fr{A} = \{0,1\}\x\lt\{0,\tfrac{1}{2}\rt\}^3\) and:
\ew
S_{\bf a} =
\begin{dcases*}
\lt\{\Ga \cdot \lt(0,a^2,x^3,x^4,a^5,a^6,x^7\rt)~\m|~ x^3,x^4,x^7 \in \bb{R}\rt\} & if \(a^1 = 0\)\\
\lt\{\Ga \cdot \lt(1,a^2,x^3,x^4,a^5,\tfrac{3}{2}a^2 + a^6 - x^4,x^7\rt)~\middle|~ x^3,x^4,x^7 \in \bb{R}\rt\} & if \(a^1 = 1\).
\end{dcases*}
\eew
Similarly, write \(\h{S} = \coprod_{{\bf a} \in \fr{A}} \h{S}_{\bf a}\) where each \(\h{S}_{\bf a}\) is the image of \(S_{\bf a}\) under the projection \(\M \to \h{\M}\).  The map:
\e\label{Ph0}
\Ph_{\bf 0}: \bb{T}^3 \x \B^4_\ep &\to \M\\
\lt[\lt(y^3,y^4,y^7\rt) + \bb{Z}^3,\lt(y^1,y^2,y^5,y^6\rt)\rt] &\mt \Ga \cdot \lt(y^1,y^2,y^3,y^4,y^5 + y^1y^3, y^6, y^7 - \tfrac{1}{2}\lt(y^1\rt)^2y^3\rt)
\ee
defines an embedding onto an open neighbourhood of \(S_{\bf 0}\) identifying \(\bb{T}^3\) with \(S_{\bf 0}\) and \(\mc{I}\) with \(\Id_{\bb{T}^3} \x -\Id_{\B^4_\ep}\) for \(\ep>0\) sufficiently small.  Similarly, for each \({\bf a} = \lt(0,a^2,a^5,a^6\rt)\), there is an embedding \(\Ph_{\bf a}: \bb{T}^3 \x \B^4_\ep \to \M\) given by \(\Ph_{\bf a} = f_{\bf a} \circ \Ph_{\bf 0}\), where \(f_{\bf a}\) is a diffeomorphism of \(\M\) induced by a left-translation of \(G\), commuting with \(\mc{I}\) and mapping \(S_{\bf 0}\) to \(S_{\bf a}\).  For the other components of the singular locus, define a lattice \(\La = \bb{Z}\cdot \bpm 1\\ 0\\ -\frac{1}{2} \epm + \bb{Z} \cdot \bpm 0\\ 1\\ 0 \epm + \bb{Z} \cdot \bpm 0\\ 0\\ 1 \epm \pc \bb{R}^3\) and write \(\tld{\bb{T}^3} = \rqt{\bb{R}^3}{\La}\).  Writing \({\bf 1} = (1,0,0,0)\), the map:
\e\label{Ph1}
\Ph_{\bf 1} : \tld{\bb{T}^3} \x \B^4_\ep &\to \M\\
\lt[\lt(y^3,y^4,y^7\rt) + \La,\lt(y^1,y^2,y^5,y^6\rt)\rt] &\mt \Ga \cdot \lt(y^1 + 1,y^2,y^3,y^4,y^5 + y^1y^3, y^6 - y^4, y^7 - y^5 - y^1y^3 - \tfrac{1}{2}\lt(y^1\rt)^2y^3\rt)
\ee
defines an embedding onto an open neighbourhood of \(S_{\bf 1}\) for \(\ep>0\) sufficiently small identifying \(\tld{\bb{T}^3}\) with \(S_{\bf 1}\) and \(\mc{I}\) with \(\Id_{\tld{\bb{T}^3}} \x -\Id_{\B^4_\ep}\).  Similarly, for each \({\bf a} = \lt(1,a^2,a^5,a^6\rt)\), there is an embedding \(\Ph_{\bf a}: \tld{\bb{T}^3} \x \B^4_\ep \to \M\) given by \(\Ph_{\bf a} = g_{\bf a} \circ \Ph_{\bf 1}\), where \(g_{\bf a}\) is a diffeomorphism of \(\M\) induced by a left-translation of \(G\), commuting with \(\mc{I}\) and mapping \(S_{\bf 1}\) to \(S_{\bf a}\).

Let \(T\) denote either \(\bb{T}^3\) or \(\tld{\bb{T}^3}\), as appropriate, and write \(\h{\Ph}_{\bf a}\) for the map \(T \x \lt(\lqt{\B^4_\ep}{\{\pm1\}}\rt) \to \h{\M}\) induced by \(\Ph_{\bf a}\), i.e.\ for the bottom map in the commutative diagram:
\ew
\bcd
T \x \B^4_\ep \ar[r, "\Ph_{\bf a}"] \ar[d, "quot"] & \M \ar[d, "quot"]\\
T \x \lt(\lqt{\B^4_\ep}{\{\pm1\}}\rt) \ar[r, "\h{\Ph}_{\bf a}"] & \h{\M}
\ecd
\eew
For each \({\bf a} = \lt(a^1, a^2,a^5,a^6\rt) \in \fr{A}\), define \(U_{\bf a} = \Ph_{\bf a}\lt(T \x \B^4_\ep\rt)\) and \(\h{U}_{\bf a} = \h{\Ph}_{\bf a}\lt[T \x \lt(\lqt{\B^4_\ep}{\{\pm1\}}\rt)\rt]\); \wlg\ one can assume that the \(U_{\bf a}\) are disjoint.  Then, shrinking \(\ep > 0\) still further if necessary, there exists a closed orbifold \g\ 3-form \(\h{\ph}\) on \(\h{\M}\) such that \(\h{\ph} = \h{\vph}\) on \(\lt.\h{\M}\m\osr\coprod_{{\bf a}\in\fr{A}} \h{U}_{\bf a}\rt.\) and on each \(\h{W}_{\bf a} = \h{\Ph}_{\bf a}\lt[T \x \lt(\lqt{\B^4_{\ep/2}}{\{\pm1\}}\rt)\rt]\) one has:
\ew
\h{\ph} = \dd y^{123} + \dd y^{145} + \dd y^{167} - \dd y^{246} + \dd y^{257} + \dd y^{347} + \dd y^{356},
\eew
where the \(y^i\) are defined in \erefs{Ph0} and \eqref{Ph1}, and the associated discussion.  Identify \(\lqt{\B^4_{\ep/2}}{\{\pm1\}}\ \pc \lqt{\bb{C}^2}{\{\pm1\}}\) by writing \(w^1 = y^1 + i y^2\) and \(w^2 = y^5 + i y^6\) and define:
\ew
\h{\om} = \frac{i}{2}\lt(\dd w^1 \w \dd \ol{w}^1 + \dd w^2 \w \dd \ol{w}^2\rt) \et \h{\Om} = \dd w^1 \w \dd w^2.
\eew
Then on \(\h{W}_{\bf a}\), one has:
\ew
\h{\ph} = \dd y^{347} + \dd y^3 \w \om - \dd y^4 \fr{Re}\Om + \dd y^7 \fr{Im}\Om.
\eew
Now recall the space \cite[\S2]{ADLatCEHS}:
\ew
\tld{X} = \OP{1}{-2} = \lt\{\lt(\lt(U^1,U^2\rt),\lt[W^1:W^2\rt]\rt) \in \bb{C}^2 \x \bb{CP}^1~\m|~U^1\lt(W^2\rt)^2 = U^2\lt(W^1\rt)^2\rt\} \cong\T^*\bb{CP}^1,
\eew
together with the continuous (non-smooth) blow-up map \(\rh:\OP{1}{-2}\to\lqt{\bb{C}^2}{\{\pm1\}}\) given by:
\ew
\lt(\lt(U^1,U^2\rt),\lt[W^1:W^2\rt]\rt) \mt \pm\lt(\sqrt{U^1},\sqrt{U^2}\rt),
\eew
where the square-roots on the right-hand side are constrained only by the condition \(\sqrt{U^1}W^2 = \sqrt{U^2}W^1\), and write \(\fr{E} = \rh^{-1}(\{0\})\) for the exceptional divisor.  Using the map \(\rh\), identify the spaces \(\lt.\tld{X}\m\osr\fr{E}\rt.\) and \(\lt(\lqt{\bb{C}^2}{\{\pm1\}}\rt)\osr\{0\}\).  It can be shown that the form \(\h{\Om}\) on \(\lt.\tld{X}\m\osr\fr{E}\rt.\) extends over all of \(\tld{X}\) to define a smooth, closed, non-zero \((2,0)\)-form \(\tld{\Om}\), however the form \(\h{\om}\) on \(\lt.\tld{X}\m\osr\fr{E}\rt.\) cannot be extended over \(\fr{E}\).  Instead, one considers the so-called \EH\ metrics \(\tld{\om}_t\) on \(\tld{X}\) defined as follows: let \(r^2 = \lt|w^1\rt|^2 + \lt|w^2\rt|^2\) denote the distance squared from the origin in \(\lqt{\bb{C}^2}{\{\pm1\}}\) and define:
\e\label{EH-defn}
\tld{\om}_t = \frac{1}{4}\dd\dd^c\lt[\sqrt{r^4 + t^4} + t^2\log\lt(\frac{r^2}{\sqrt{r^4 + t^4} + t^2}\rt)\rt] \text{ on } \lt(\lqt{\bb{C}^2}{\{\pm1\}}\rt)\osr\{0\},
\ee
where \(\dd^c = i\lt(\ol{\del} - \del\rt)\) as usual.  Then \(\tld{\om}_t\) can be extended smoothly over \(\fr{E}\) to define a Ricci-flat K\^^22{a}hler form on \(\tld{X}\) \cite[p.\ 160]{CMwSH}.  The forms \(\tld{\om}_t\) can be used to `extend' \(\h{\om}\) over the exceptional divisor in the following sense: for \(\ep>0\), write \(\tld{X}_\ep\) for the pre-image of \(\lqt{\B^4_\ep}{\{\pm1\}}\) under the map \(\rh:\tld{X} \to \lqt{\bb{C}^2}{\{\pm1\}}\).  Then for every \(\ep > 0\), there exists \(t\) sufficiently small (depending on \(\ep\)) and a K\^^22{a}hler form \(\ca{\om}_t\) on \(\tld{X}\) such that:
\e\label{EH-interp}
\ca{\om}_t = \h{\om} \text{ on a neighbourhood of } \lt.\tld{X}\m\osr\tld{X}_{\frac{1}{2}\ep}\rt. \et \ca{\om}_t = \tld{\om}_t \text{ on } \tld{X}_{\frac{1}{4}\ep}.
\ee
Now define a new manifold \(\ca{\M}\) by:
\ew
\ca{\M} = \lt(\h{\M}\m\osr\coprod_{{\bf a}\in\fr{A}}\h{W}_{\bf a}\rt) \underset{\tl}{\bigcup} \lt(\coprod_{{\bf a}\in\fr{A}} T \x \tld{X}_\ep\rt)
\eew
where \(\underset{\tl}{\bigcup}\) denotes that, for each \({\bf a} \in \fr{A}\), the region \(\lt.\h{U}_{\bf a}\m\osr \h{W}_{\bf a}\rt.\) should be identified with the region \(T \x \lt(\tld{X}_\ep\m\osr\tld{X}_{\frac{1}{2}\ep}\rt) \cong T \x \lt.\lt(\lqt{\B^4_\ep}{\{\pm1\}}\rt)\m\osr\lt(\lqt{\B^4_{\frac{1}{2}\ep}}{\{\pm1\}}\rt)\rt.\) using \(\h{\Ph}_{\bf a}\).  Denote the image of \(T \x \tld{X}_\ep\) in \(\ca{\M}\) corresponding to \({\bf a} \in \fr{A}\) by \(\ca{U}_{\bf a}\) and the image of \(T \x \tld{X}_{\frac{1}{2}\ep}\) by \(\ca{W}_{\bf a}\).  Define a 3-form \(\ca{\ph}\) on \(\ca{\M}\) by setting \(\ca{\ph} = \h{\ph}\) on \(\lt(\h{\M}\m\osr\coprod_{{\bf a}\in\fr{A}}\h{W}_{\bf a}\rt)\) and setting:
\ew
\ca{\ph} = \dd y^{347} + \dd y^3 \w \ca{\om}_t - \dd y^4 \w \fr{Re}\tld{\Om} + \dd y^7 \w \fr{Im}\tld{\Om}
\eew
on \(\ca{W}_{\bf a}\) for each \({\bf a}\).  \(\ca{\ph}\) is smooth and well-defined by \eref{EH-interp}.  This yields:

\begin{Thm}[{\cite[Thm.\ 21]{ACG2CMwFBNb1eq1}}]\label{FFKMconstruction}
Let \(\rh:\ca{\M}\to\h{\M}\) denote the `blow-down' map. Then there exists a smooth, closed \g\ 3-form \(\ca{\ph}\) on \(\ca{\M}\) such that \(\rh_*\ca{\ph} = \h{\ph}\) outside a neighbourhood of the singular locus \(S\).
\end{Thm}

\begin{Rk}\label{Rk-on-Coh}
It is well-known that \(\tld{\om}_t\) defines a non-zero cohomology class on \(\tld{X}\) which depends on \(t\).  Using \cite[Prop.\ 22]{ACG2CMwFBNb1eq1}, it follows that the cohomology class of \(\ca{\ph}\) also depends on the choice of \(t\) and hence on \(\ep\).  To prove the unboundedness of \(\CH\) on \(\lt(\ca{\M},\ca{\ph}\rt)\), I construct a family of closed \g\ 3-forms \(\ca{\ph}^\mu\) with unbounded volume in the fixed cohomology class \(\lt[\ca{\ph}\rt]\); thus they must all have the same `choice of \(\ep\)'.  This is an important technical subtlety in the construction of the forms \(\ca{\ph}^\mu\).
\end{Rk}

\subsection{The unboundedness of \(\CH\)}\label{FFKMUBpf}

I now prove \tref{FFKM-UA}.  Recall the statement of the theorem:\vs{3mm}

\noindent{\bf Theorem \ref{FFKM-UA}.}
\em The Hitchin functional:
\ew
\CH:\lt[\ca{\ph}\rt]_+ \to (0,\infty)
\eew
is unbounded above.\vs{2mm}\em

\(\lt(\ca{\M},\ca{\ph}\rt)\) does not satisfy the hypotheses of \pref{UB-Md-Prop}, however the manifold \((\M,\vph)\) does satisfies the hypotheses (specifically, the hypotheses of \pref{UB-Md-Prop}(1)) since, by \eref{FFKMform}, the 3-form \(\th^{123} = \dd\lt(\th^{25}\rt)\) is exact.  Thus, by \pref{UB-Md-Prop}(1), for each \(\mu \ge 1\), the 3-form:
\ew
\vph^\mu = \mu^6\th^{123} + \th^{145} + \th^{167} - \th^{246} + \th^{257} + \th^{347} + \th^{356}
\eew
is of \g-type and satisfies \(vol_{\vph^\mu} = \mu^2vol_\vph\).  By \eref{FFKMInv}, both the 3-form \(\th^{123}\) and the 2-form \(\th^{25}\) are \(\mc{I}\)-invariant and thus descend to the orbifold \(\h{\M}\).  Hence the forms \(\vph^\mu\) descend to define closed \g\ 3-forms \(\h{\vph}^\mu\) on \(\h{\M}\) with unbounded volume, which lie in the fixed cohomology class \(\lt[\h{\vph}\rt]\).

To complete the proof of \tref{FFKM-UA}, therefore, it suffices to `resolve the singularities' of \(\lt(\h{\M},\h{\vph}^\mu\rt)\).  The obvious approach is to mimic the construction of \(\lt(\ca{\M},\ca{\ph}\rt)\), by first deforming \(\h{\vph}^\mu\) into the form:
\ew
\h{\xi}^\mu = \mu^6\dd y^{123} + \dd y^{145} + \dd y^{167} - \dd y^{246} + \dd y^{257} + \dd y^{347} + \dd y^{356}
\eew
in a neighbourhood of the singular locus, and then resolving the singularity in \(\h{\xi}\) using \(\ca{\om}_t\) as above.  However this approach fails: in order to deform \(\h{\vph}^\mu\) into \(\h{\xi}^\mu\) on the region \(\h{U}_{\bf a} \cong T \x \lt(\lqt{\B^4_\ep}{\{\pm1\}}\rt)\), it is necessary for \(\ep\) to depend on \(\mu\).  This implies that the cohomology class of the resolved 3-form \(\ca{\ph}^\mu\) also depends on \(\mu\) (see \rref{Rk-on-Coh}) and thus this construction fails to demonstrate the unboundedness of the Hitchin functional \(\CH\) on the fixed cohomology class \(\lt[\ca{\ph}\rt]\).  Thus, instead, I deform \(\h{\vph}^\mu\) into the form:
\ew
\h{\xi}^\mu + y^1\dd y^{147} = \mu^6\dd y^{123} + \dd y^{145} + \dd y^{167} - \dd y^{246} + \dd y^{257} + \dd y^{347} + \dd y^{356} + y^1\dd y^{147}
\eew
near the singular locus.  This deformation can be performed on \(\h{U}_{\bf a} \cong T \x \lt(\lqt{\B^4_\ep}{\{\pm1\}}\rt)\) with \(\ep\) chosen independently of \(\mu\).  The additional term \(y^1 \dd y^{147}\) persists during the resolution of singularities, before being cut-off near the exceptional divisor, at some distance from the exceptional divisor depending on \(\mu\).  This enables the resolved 3-forms \(\ca{\ph}^\mu\) to lie in a fixed cohomology class, completing the proof of \tref{FFKM-UA}.

\begin{Rk}
The reader will recall that Joyce \cite{CR7MwHG2I,CR7MwHG2II,CMwSH} constructed numerous \g-manifolds by resolving the singularities in finite quotients of the torus \(\lt(\bb{T}^7,\ph_0\rt)\).  Despite the similarities between Joyce's construction and the construction of \(\lt(\ca{\M},\ca{\ph}\rt)\), the results of this paper do not apply to Joyce's manifolds since, unlike \((\M,\vph)\), the torus \(\lt(\bb{T}^7,\ph_0\rt)\) itself does not satisfy the hypotheses of \pref{UB-Md-Prop}.  Thus, the question of whether \(\CH\) is unbounded above on manifolds admitting torsion-free \g\ 3-forms appears to remain beyond our current understanding.
\end{Rk}

I begin with the following lemma:
\begin{Lem}\label{quadlem}
Let \({\bf a}\in\fr{A}\), let \(r \ge 0\) denote the radial distance from the singular locus in \(\h{U}_{\bf a}\), i.e.:
\ew
r^2 = \lt(y^1\rt)^2 + \lt(y^2\rt)^2 + \lt(y^5\rt)^2 + \lt(y^6\rt)^2,
\eew
where the \(y^i\) are defined in \erefs{Ph0} and \eqref{Ph1}, and the associated discussion, and define:
\ew
\h{\xi}^\mu = \mu^6\dd y^{123} + \dd y^{145} + \dd y^{167} - \dd y^{246} + \dd y^{257} + \dd y^{347} + \dd y^{356}.
\eew
Then there exist a constant \(C>0\) and a 2-form \(\h{\al}_{\bf a}\) on \(\h{U}_{\bf a}\), both independent of \(\mu\), satisfying:
\e\label{quadbd}
\lt|\h{\al}_{\bf a}\rt|_{\mns{\h{\xi}^\mu}} \le C\mu^{-1} r^2 \et \lt|\dd\h{\al}_{\bf a}\rt|_{\mns{\h{\xi}^\mu}} \le Cr
\ee
such that:
\ew
\h{\vph}^\mu - \h{\xi}^\mu = y^1\dd y^{147} +  \dd\h{\al}_{\bf a}.
\eew
(Here \(|\cdot|_{\mns{\h{\xi}^\mu}}\) denotes the pointwise norm induced by the \g\ 3-form \(\h{\xi}^\mu\). E.g.\ in the case \(\mu=1\), this is just the Euclidean norm in the \(y^i\) coordinates, denoted \(| \1 |_\E\).)
\end{Lem}

\begin{proof}
Begin by working on \(U_{\bf a}\).  Using the equation:
\ew
\Ph_{\bf a} =
\begin{dcases*}
f_{\bf a} \circ \Ph_{\bf 0} & if \(a^1 = 0\);\\
g_{\bf a} \circ \Ph_{\bf 1} & if \(a^1 = 1\),
\end{dcases*}
\eew
together with the fact that both \(f_{\bf a}\) and \(g_{\bf a}\) are induced by left-translations, and hence preserve each \(\th^i\), one sees that:
\ew
\Ph_{\bf a}^*\th^i =
\begin{dcases*}
\Ph_{\bf 0}^*\th^i & if \(a^1 = 0\);\\
\Ph_{\bf 1}^*\th^i & if \(a^1 = 1\).
\end{dcases*}
\eew
Using the explicit expressions for \(\Ph_{\bf 0}\) and \(\Ph_{\bf 1}\) given in \erefs{Ph0} and \eqref{Ph1}, together with \eref{FFKM1forms}, it follows that:
\ew
\Ph_{\bf 0}^*\bpm
\th^1\\
\th^2\\
\th^3\\
\th^4\\
\th^5\\
\th^6\\
\th^7
\epm
=
\Ph_{\bf 1}^*\bpm
\th^1\\
\th^2\\
\th^3\\
\th^4\\
\th^5\\
\th^6\\
\th^7
\epm
=
\begin{pmatrix}
\dd y^1\\
\dd y^2\\
\dd y^3\\
\dd y^4 - y^2 \dd y^1\\
\dd y^5 + y^1\dd y^3\\
\dd y^6 + y^1\dd y^4\\
\dd y^7 + y^1\dd y^5 + \tfrac{1}{2}\lt(y^1\rt)^2\dd y^3
\end{pmatrix}.
\eew

Therefore:
\ew
\vph^\mu - \xi^\mu &= y^1\lt(\dd y^{147} - \dd y^{156} - \dd y^{134} + \dd y^{237}\rt) + y^2\lt(\dd y^{137} - \dd y^{126}\rt)\\
&\hs{8mm}+ \tfrac{1}{2}\lt(y^1\rt)^2\lt(2 \dd y^{145} - \dd y^{136} + \dd y^{235}\rt) + y^1y^2\lt(\dd y^{135} - \dd y^{124}\rt) - \tfrac{1}{2}\lt(y^1\rt)^3\dd y^{134}\\
&= y^1\dd y^{147} + \dd\al_{\bf a},
\eew
where:
\ew
\al_{\bf a}  &= \dd y^1 \w \lt[\lt(y^1y^5 + \tfrac{1}{2}\lt(y^2\rt)^2\rt)\dd y^6 + \lt(\lt(y^1\rt)^2y^5 + \tfrac{1}{2}y^1\lt(y^2\rt)^2\rt)\dd y^4 - \tfrac{1}{2}\lt(y^1\rt)^2y^6\dd y^3\rt]\\
& \hs{24mm} + \dd y^3 \w \lt[\lt(-\tfrac{1}{2}\lt(y^1\rt)^2 - \tfrac{1}{8}\lt(y^1\rt)^4\rt)\dd y^4 + y^1y^2\dd y^7 + \tfrac{1}{2}\lt(y^1\rt)^2y^2\dd y^5\rt]\\
&= \dd y^1 \w \be_{\bf a} + \dd y^3 \w \ga_{\bf a}.
\eew
Observe that there exists \(C>0\) independent of \(\mu\) such that:
\ew
|\be_{\bf a}|_\E, |\ga_{\bf a}|_\E \le \frac{C}{2}r^2 \et |\dd\be_{\bf a}|_\E, |\dd\ga_{\bf a}|_\E \le \frac{C}{2}r.
\eew
Also, by solving the linear system in \eref{lin-sys}, one may verify that:
\e\label{met-expr}
g_{\xi^\mu} = \mu^4\lt(\lt(\dd y^1\rt)^{\ts2} + \lt(\dd y^2\rt)^{\ts2} + \lt(\dd y^3\rt)^{\ts2}\rt) + \mu^{-2}\lt(\lt(\dd y^4\rt)^{\ts2} + \lt(\dd y^5\rt)^{\ts2} + \lt(\dd y^6\rt)^{\ts2} + \lt(\dd y^7\rt)^{\ts2}\rt).
\ee
In particular \(g_{\xi^\mu} \ge \mu^{-2}g_\E\) when acting on vectors.  It follows that \(|\cdot|_{\mns{\xi^\mu}} \le \mu|\cdot|_\E\) when acting on 1-forms, and \(|\cdot|_{\mns{\xi^\mu}} \le \mu^2|\cdot|_\E\) when acting on 2-forms.  Hence:
\ew
|\be_{\bf a}|_{\mns{\xi^\mu}}, |\ga_{\bf a}|_{\mns{\xi^\mu}} \le \frac{C}{2}\mu r^2 \et |\dd\be_{\bf a}|_{\mns{\xi^\mu}}, |\dd\ga_{\bf a}|_{\mns{\xi^\mu}} \le \frac{C}{2}\mu^2r.
\eew
One may also compute that \(\lt|\dd y^1\rt|_{\mns{\xi^\mu}} = \lt|\dd y^3\rt|_{\mns{\xi^\mu}} = \mu^{-2}\).  Therefore:
\ew
|\al_{\bf a}|_{\mns{\xi^\mu}} &\le \lt|\dd y^1\rt|_{\mns{\xi^\mu}} |\be_{\bf a}|_{\mns{\xi^\mu}} + \lt|\dd y^3\rt|_{\mns{\xi^\mu}} |\ga_{\bf a}|_{\mns{\xi^\mu}}\\
&\le C\mu^{-1}r^2,
\eew
as required.  Likewise \(\dd\al_{\bf a} = \dd y^1 \w \dd\be_{\bf a} + \dd y^3 \w \dd\ga_{\bf a}\) and hence \(|\dd\al_{\bf a}|_{\mns{\xi^\mu}} \le Cr\).  Since \(\mc{I}^*\al_{\bf a} = \al_{\bf a}\), \(\al_{\bf a}\) descends to define the required 2-form \(\h{\al}_{\bf a}\) on \(\h{U}_{\bf a}\).

\end{proof}

\begin{Rk}
The term \(y^1\dd y^{147}\) is also exact with primitive \(\frac{1}{2}\lt(y^1\rt)^2\dd y^{47}\), however one may calculate that:
\ew
\lt|\frac{1}{2}\lt(y^1\rt)^2\dd y^{47}\rt|_{\mns{\h{\xi}^\mu}} = \frac{\mu^2}{2}\lt|y^1\rt|^2;
\eew
thus this primitive does not satisfy the bounds in \eref{quadbd}.  It is for this reason that the term \(y^1 \dd y^{147}\) is dealt with separately to the other terms in the expression for \(\h{\vph}^\mu - \h{\xi}^\mu\).
\end{Rk}

Using \lref{quadlem}, I now prove:
\begin{Prop}\label{almostprod}
There exists \(\ep_0>0\), independent of \(\mu\), such that for all \(\ep\in(0,\ep_0]\), the following is true:

For all \(\mu\ge1\), there exists a closed, orbifold \g\ 3-form \(\h{\ph}^\mu\) on \(\h{\M}\) such that:
\ew
\h{\ph}^\mu = \h{\vph}^\mu \text{ on } \lt.\h{\M}\m\osr\coprod_{{\bf a}\in\fr{A}} \h{U}_{\bf a}\rt.
\eew
and on each \(\h{W}_{\bf a}\) for \({\bf a}\in\fr{A}\), one has:
\ew
\h{\ph}^\mu = \mu^6\dd y^{123} + \dd y^{145} + \dd y^{167} - \dd y^{246} + \dd y^{257} + \dd y^{347} + \dd y^{356} + y^1\dd y^{147}.
\eew

\end{Prop}

\begin{proof}
Again, begin by working at the level of \(\M\).  Let \(f:[0,\infty)\to[0,1]\) be a smooth function such that:
\e\label{f}
\pt &\hs{5mm} f\equiv 0 \text{ on an open neighbourhood of } \lt[0,\tfrac{1}{2}\rt];\\
\pt &\hs{5mm} f\equiv 1 \text{ on an open neighbourhood of } [1,\infty);\hs{50mm}\\
\pt &\hs{5mm} \lt\|f'\rt\|_\infty \le 3.
\ee
Consider the 3-form \(\ph\) on \(U_{\bf a}\) defined by:
\e\label{ph-mu}
\ph^\mu = \xi^\mu + y^1\dd y^{147} + \dd\lt[f\lt(\frac{r}{\ep}\rt)\al_{\bf a}\rt].
\ee
Clearly \(\ph^\mu\) is closed and satisfies:
\ew
\ph^\mu = 
\begin{dcases*}
\xi^\mu + y^1\dd y^{147} & on \(W_{\bf a}\);\\
\vph^\mu & near the boundary of \(U_{\bf a}\).
\end{dcases*}
\eew
On \(U_{\bf a}\), using \erefs{quadbd}, \eqref{met-expr} and \eqref{f}, one may compute that:
\ew
\lt|\ph^\mu - \xi^\mu\rt|_{\mns{\xi^\mu}} &\le \lt|y^1\dd y^{147}\rt|_{\mns{\xi^\mu}} + \lt|\dd\al_{\bf a}\rt|_{\mns{\xi^\mu}} + \frac{\lt\|f'\rt\|_\infty}{\ep}\lt|\dd r\rt|_{\mns{\xi^\mu}}\lt|\al_{\bf a}\rt|_{\mns{\xi^\mu}}\\
&\le (4C + 1)\ep,
\eew
where \(C>0\) is as in \lref{quadlem} (recall that \(\lt|\dd r\rt|_{\mns{\xi^\mu}} \le \mu\), as in the proof of \lref{quadlem}).  Thus \(\ph^\mu\) is of \g-type for all \(\ep > 0\) sufficiently small, independent of \(\mu\), by the stability of \g\ 3-forms.  Since \(\xi^\mu\), \(y^1 \dd y^{147}\) and \(\al_{\bf a}\) are all \(\mc{I}\)-invariant, the form \(\ph^\mu\) descends to define an orbifold \g\ 3-form \(\h{\ph}^\mu\) on \(\h{\M}\), completing the proof.

\end{proof}

One can also use \lref{quadlem} to give an explicit formula for \(\h{\ph}\) described in \sref{FFKM-constr}, on the region \(\h{U}_{\bf a}\).  Explicitly, one takes:
\e\label{h-ph}
\h{\ph} = \h{\xi} + \dd\lt[f\lt(\frac{r}{\ep}\rt)\lt(\tfrac{1}{2}\lt(y^1\rt)^2\dd y^{47} + \al_{\bf a}\rt)\rt].
\ee
In particular, note that whilst \(\h{\vph}^1 = \h{\vph}\), it is not true that \(\h{\ph}^1 = \h{\ph}\).

The task now is to resolve the singularities in \(\h{\ph}^\mu\).  I begin by introducing some notation. Firstly, for \(k\in(0,\infty)\), define:
\ew
\B^4\lt(\tfrac{1}{2}\ep,k\rt) = \lt\{\lt(w^1,w^2\rt) \in \bb{C}^2~\m|~k^6\lt|w^1\rt|^2 + \lt|w^2\rt|^2 < \frac{1}{2}\ep\rt\}.
\eew
Thus \(\B^4\lt(\tfrac{1}{2}\ep,k\rt)\) is a complex ellipse with radius \(\frac{1}{2}k^{-3}\ep\) in the \(w^1\)-direction and radius \(\frac{1}{2}\ep\) in the \(w^2\)-direction. Also define \(\tld{X}\lt(\tfrac{1}{2}\ep,k\rt)\) to be the pre-image of \(\lqt{\B^4\lt(\tfrac{1}{2}\ep,k\rt)}{\{\pm1\}}\) under the blow-down map \(\rh\) and, for \(k \in \lt[2^{-\frac{1}{3}}, \infty\rt)\), define \(\ca{W}_{{\bf a},k}\) to be the subset of \(\ca{U}_{\bf a}\) corresponding to \(T \x \tld{X}\lt(\tfrac{1}{2}\ep,k\rt)\). (\(k \ge 2^{-\frac{1}{3}}\) is needed to ensure that \(\ca{W}_{{\bf a},k} \pc \ca{U}_{\bf a}\).)  Secondly, define:
\ew
\bb{T}^3_\mu = \rqt{\bb{R}^3}{\mu^3\bb{Z} \ds \bb{Z} \ds \bb{Z}}.
\eew
Analogously, let \(\La_\mu\) denote the image of \(\La\) under the map \((y^3,y^4,y^7) \in \bb{R}^3 \mt (\mu^3 y^3, y^4, y^7) \in \bb{R}^3\) and define:
\ew
\tld{\bb{T}^3_\mu} = \rqt{\bb{R}^3}{\La_\mu}.
\eew
As above, use \(T_\mu\) to denote either \(\bb{T}^3_\mu\) or \(\tld{\bb{T}^3_\mu}\) as appropriate.

Begin by considering the space \(R_\mu = \lt(T_\mu\rt)_{y^3,y^4,y^7} \x \tld{X}\lt(\frac{1}{2}\ep, \mu^{-1}\rt)_{y^1,y^2,y^5,y^6} \yy T_\mu \x \tld{X}_{\frac{1}{2}\ep}\).  Define a 3-form \(\si\) on \(R_\mu\) via:
\e\label{si-defn}
\si = \dd\lt[f\lt(\frac{2r}{\ep}\rt)\cdot\frac{1}{2}\lt(y^1\rt)^2\dd y^{47}\rt],
\ee
where \(f\) is as defined in \eref{f}.  Clearly, \(\si\) vanishes near the exceptional locus and thus \(\si\) defines a smooth 3-form over all of \(R_\mu\), via extension by zero. Moreover, outside the region \(T_\mu \x \tld{X}_{\frac{1}{2}\ep}\) (i.e.\ on the region \(\lt\{r \ge \frac{1}{2}\ep\rt\}\)) \(\si\) is simply given by \(y^1 \dd y^{147}\).

Next, define a 3-form \(\ze\) on \(R_\mu\) via:
\e\label{ze-defn}
\ze = \dd y^{347} + \dd y^3 \w \ca{\om}_t - \dd y^4 \w \fr{Re}\tld{\Om} + \dd y^7 \w \fr{Im}\tld{\Om}
\ee
for \(\ca{\om}_t\) as in \sref{FFKM-constr}. \(\ze\) defines a \g\ 3-form on \(R_\mu\) by \lref{scaling-lem}. Finally, define a 3-form \(\ze^\mu\) on \(R_\mu\) as follows:
\ew
\ze^\mu = \ze + \frac{1}{\mu^3}\si.
\eew

\begin{Lem}\label{tech-lem-for-res}
For \(\ep>0\) sufficiently small, independent of \(\mu\), and for \(\mu\) sufficiently large, \(\ze^\mu\) is of \g-type on \(R_\mu\).
\end{Lem}

\begin{proof}
The proof is, again, an application of the stability of \g\ 3-forms.  Firstly, consider the region \(T_\mu \x \lt.\tld{X}\lt(\frac{1}{2}\ep, \mu^{-1}\rt)\m\osr\tld{X}_{\frac{1}{2}\ep}\rt.\) (i.e.\ the region \(\lt\{r \ge \frac{1}{2}\ep\rt\}\)). Here, \(\ze = \h{\xi}\) and \(\si = y^1\dd y^{147}\), so:
\ew
\lt|\si\rt|_\ze = \lt|y^1\rt| \le \frac{\ep\mu^3}{2}
\eew
and thus:
\ew
\lt|\ze^\mu - \ze\rt|_\ze \le \frac{1}{2}\ep.
\eew
Hence \(\ze_\mu\) is of \g-type on \(T_\mu \x \lt.\tld{X}\lt(\frac{1}{2}\ep, \mu^{-1}\rt)\m\osr\tld{X}_{\frac{1}{2}\ep}\rt.\) for all \(\mu\) if \(\ep\) is sufficiently small, independent of \(\mu\).

Now fix \(\ep\) and consider the region \(T_\mu \x \tld{X}_{\frac{1}{2}\ep}\). On this region \(\lt|\si\rt|_\ze \le C\) for some fixed \(C>0\) independent of \(\mu\). Thus:
\ew
\lt|\ze^\mu - \ze\rt|_\ze \le \frac{C}{\mu^3}.
\eew
Thus for \(\mu\) sufficiently large, \(\ze^\mu\) is also of \g-type on the region \(T_\mu \x \tld{X}_{\frac{1}{2}\ep}\). Thus \(\ze^\mu\) is of \g-type on all of \(R_\mu\) and the result is proven.

\end{proof}

\begin{Rk}
It can be shown that, whilst the constant \(C\) is independent of \(\mu\), it does depend on \(\ep\); specifically, \(C \to 0\) as \(\ep\to0\). Thus, reducing \(\ep\) if necessary (independently of \(\mu\)), the conclusion of \lref{tech-lem-for-res} is valid for all \(\mu\ge1\). However, since unboundedness ultimately occurs as \(\mu\to\infty\), the simpler statement above suffices.
\end{Rk}

Using this lemma, the \g\ 3-forms required for the resolution can be constructed. Firstly, consider the map \(\lqt{\bb{C}^2}{\{\pm1\}} \to \lqt{\bb{C}^2}{\{\pm1\}}\) given by:
\ew
\lt(w^1,w^2\rt) \mt \lt(\mu^3w^1,w^2\rt).
\eew
Restricting this map to the region \(\lt(\lqt{\bb{C}^2}{\{\pm1\}}\rt)\osr\lt\{0\rt\}\) and using the blow-up map \(\rh\) gives rise to a map \(\fr{h}^\mu:\lt.\tld{X}\m\osr\fr{E}\rt. \to \lt.\tld{X}\m\osr\fr{E}\rt.\) which extends to all of \(\tld{X}\). Now define:
\e\label{ca-ph-mu-homthty}
\fr{H}^\mu: \ca{W}_{\bf a} \cong \lt(T\rt)_{y^3,y^4,y^7} \x \lt(\tld{X}_{\frac{1}{2}\ep}\rt)_{y^1,y^2,y^5,y^6} \to \lt(T_\mu\rt)_{y^3,y^4,y^7} \x \tld{X}\lt(\tfrac{1}{2}\ep, \mu^{-1}\rt)_{y^1,y^2,y^5,y^6},
\ee
where the action of \(\fr{H}^\mu\) on \(T\) is induced by the map \((y^3,y^4,y^7) \in \bb{R}^3 \mt (\mu^3 y^3, y^4, y^7) \in \bb{R}^3\) and \(\fr{H}^\mu\) acts on \(\tld{X}_{\frac{1}{2}\ep}\) by \(\fr{h}^\mu\), and write:
\e\label{tldze}
\ca{\ze}^\mu = \mu^{-3}\lt(\fr{H}^\mu\rt)^*\ze^\mu.
\ee
By \lref{tech-lem-for-res}, this is a smooth, closed, \g\ 3-form on \(\ca{W}_{\bf a}\). An explicit computation shows that near the boundary of \(\ca{W}_{\bf a}\) (and, more generally, on an open neighbourhood of the region \(\lt.\ca{W}_{\bf a}\m\osr\ca{W}_{{\bf a},\mu}\rt.\)):
\e\label{outside-surg}
\ca{\ze}^\mu = \mu^6\dd y^{123} + \dd y^{145} + \dd y^{167} - \dd y^{246} + \dd y^{257} + \dd y^{347} + \dd y^{356} + y^1\dd y^{147},
\ee
which matches the boundary condition required for the resolution.  Thus for each \(\mu\in[1,\infty)\), one obtains a smooth, closed \g\ 3-form \(\ca{\ph}^\mu\) on \(\ca{\M}\) by setting \(\ca{\ph}^\mu = \h{\ph}^\mu\) outside \(\ca{W}_{\bf a}\) for each \({\bf a}\in\fr{A}\) and \(\ca{\ph}^\mu = \ca{\ze}^\mu\) on each \(\ca{W}_{\bf a}\).

Now let \(\overset{\circ}{\M} = \ca{\M}\lt\osr\coprod_{{\bf a}\in\fr{A}}\ca{U}_{\bf a}\rt.\). Then on \(\overset{\circ}{\M}\) one has \(\ca{\ph}^\mu = \h{\vph}^\mu\) and hence \(vol_{\ca{\ph}^\mu} = \mu^2\th^{1...7}\) by \pref{UB-Md-Prop}.  Hence, one may compute that:
\ew
\CH\lt(\ca{\ph}^\mu\rt) \ge \bigintsss_{\overset{\circ}{\M}} vol_{\ca{\ph}^\mu} = \mu^2\bigintsss_{\overset{\circ}{\M}} \th^{1...7} \to \infty \text{ as } \mu \to \infty.
\eew

Thus, the proof of \tref{FFKM-UA} is completed by the following result:
\begin{Prop}
Let \(\lt(\ca{\M},\ca{\ph}\rt)\) be as defined in \tref{FFKMconstruction} and let \(\ca{\ph}^\mu\) be as defined above.  Then:
\ew
\lt[\ca{\ph}^\mu\rt] = \lt[\ca{\ph}\rt] \in \dR{3}\lt(\ca{\M}\rt) \text{ for all } \mu\ge 1.
\eew
\end{Prop}

\begin{proof}
It suffices to prove that the difference \(\ca{\ph}^\mu - \ca{\ph}\) is exact for each \(\mu\ge1\). The strategy is to prove that \(\ca{\ph}^\mu - \ca{\ph}\) is exact on each of the regions:
\begin{itemize}
\item \(\ca{W}_{\bf a}\) for \({\bf a}\in\fr{A}\);
\item \(\lt.\ca{\M}\m\osr\coprod_{{\bf a}\in\fr{A}}\ca{U}_{\bf a}\rt.\);
\item \(\lt.\ca{U}_{\bf a}\m\osr\ca{W}_{\bf a}\rt.\) for \({\bf a}\in\fr{A}\),
\end{itemize}
and then to verify that the primitives may be combined to define a global primitive on all of \(\ca{\M}\).\\

\noindent\ul{\(\ca{W}_{\bf a}\) for \({\bf a}\in\fr{A}\)}:\ Recall the map \(\fr{h}^\mu\) defined above. One can verify that:
\ew
\ca{\ze}^\mu = \dd y^{347} + \dd y^3 \w \lt(\fr{h}^\mu\rt)^*\ca{\om}_t - \dd y^4 \w \fr{Re}\tld{\Om} + \dd y^7 \w \fr{Im} \tld{\Om} + \dd\lt[f\lt(\frac{2\lt(\fr{h}^\mu\rt)^*r}{\ep}\rt)\cdot\frac{1}{2}\lt(y^1\rt)^2\dd y^{47}\rt]
\eew
and thus:
\ew
\ca{\ph}^\mu - \ca{\ph} =   \dd y^3 \w \lt[\lt(\fr{h}^\mu\rt)^*\ca{\om}_t - \ca{\om}_t\rt] + \dd\lt[f\lt(\frac{2\lt(\fr{h}^\mu\rt)^*r}{\ep}\rt)\cdot\frac{1}{2}\lt(y^1\rt)^2\dd y^{47}\rt] \text{ on } \ca{W}_\textbf{a}
\eew
The second term is manifestly exact. For the first term, recall the Generalised Poincar\'{e} Lemma \cite[Prop.\ 17.10]{ItSM}:
\begin{Thm}
Let \(X\), \(Y\) be smooth manifolds, let
\begin{tikzcd}
X \ar[r, shift left, "f_1"] \ar[r, shift right, "f_2" '] & Y 
\end{tikzcd}
be smooth maps and let \(F:f_1\Rightarrow f_2\) be a smooth homotopy. Then the maps
\begin{tikzcd}
\dR{\pt}(Y) \ar[r, shift left,  "f_1^*"] \ar[r,  shift right, "f_2^*" '] & \dR{\pt}(X)
\end{tikzcd}
are equal.
\end{Thm}

\noindent Since \(\fr{h}^\mu\) is homotopic to the identity on \(\tld{X}\), it follows that \(\lt(\fr{h}^\mu\rt)^*\ca{\om}_t - \ca{\om}_t = \dd \ta\) for some suitable \(\ta\). Thus, on the region \(\ca{W}_{\bf a}\), one finds that:
\ew
\ca{\ph}^\mu - \ca{\ph} &= \dd\lt[\ta \w \dd y^3 + f\lt(\frac{2\lt(\fr{h}^\mu\rt)^*r}{\ep}\rt)\cdot\frac{1}{2}\lt(y^1\rt)^2\dd y^{47}\rt]\\
&= \dd\vpi.
\eew

In order to extend \(\vpi\) to all of \(\ca{\M}\) below, it is necessary to compute \(\vpi\) explicitly near the boundary of \(\ca{W}_{\bf a}\). For the second term in \(\vpi\), since \(\lt(\fr{h}^\mu\rt)^*(r)\ge r\), one finds that:
\ew
f\lt(\frac{2\lt(\fr{h}^\mu\rt)^*r}{\ep}\rt)\cdot\frac{1}{2}\lt(y^1\rt)^2\dd y^{47} = \frac{1}{2}\lt(y^1\rt)^2\dd y^{47} \text{ near the boundary of } \ca{W}_{\bf a}.
\eew
For the first term in \(\vpi\), recall that the Generalised Poincar\'{e} Lemma stated above may be proved by constructing an explicit chain homotopy
\bcd
\Om^\pt(Y) \ar[r, bend left = 20, "f_1^*"{name=top}] \ar[r, bend right = 20, "f_2^*" '{name=bottom}] & \Om^\pt(X)
\ar[Rightarrow, shorten=1mm, from=top, to=bottom, "\fr{F}"]
\ecd
defined by:
\ew
\bcd[row sep = 0pt]
\fr{F}:\Om^\pt(Y) \ar[r]& \Om^{\pt-1}(X)\\
\om \ar[r, maps to]& \bigint_{[0,1]} \io_s^*\lt(\del_s\hk(F^*\om)\rt) \dd s
\ecd
\eew
and calculating that:
\ew
\dd\fr{F} + \fr{F}\dd = f_2^*-f_1^*,
\eew
where the homotopy \(F\) is viewed as a map \(F: [0,1] \x X \to Y\) and \(\io_s\) denotes the embedding \(X \cong \{s\} \x X \emb [0,1] \x X\).  Using the specific homotopy \(\fr{h}^{\sqrt[3]{1 + s\lt(\mu^3-1\rt)}} = F_s\) of \(\tld{X}\) connecting \(\Id\) to \(\fr{h}^\mu\), one may calculate that:
\ew
\ta = \frac{\mu^6 - 1}{2}\lt(y^1\dd y^2 - y^2\dd y^1\rt) \text{ near the boundary of } \ca{W}_{\bf a}.
\eew
Thus:
\e\label{BC1}
\vpi = \frac{\mu^6 - 1}{2}\lt(y^1\dd y^{23} - y^2\dd y^{13}\rt) +  \frac{1}{2}\lt(y^1\rt)^2\dd y^{47} \text{ near the boundary of } \ca{W}_{\bf a}.
\ee

\noindent\ul{\(\ca{\M}\lt\osr\coprod_{{\bf a}\in\fr{A}} \ca{U}_{\bf a}\rt.\)}:~ As discussed at the start of \sref{FFKMUBpf}, by \eref{FFKMform} one has:
\ew
\ca{\ph}^\mu - \ca{\ph} = \lt(\mu^6 - 1\rt)\th^{123} = \lt(\mu^6 - 1\rt)\dd\lt(\th^{25}\rt) = \dd\vpi.
\eew
Using \erefs{FFKM1forms}, \eqref{Ph0} and \eqref{Ph1}, one finds that:
\e\label{BC2}
\vpi = \lt(\mu^6 - 1\rt)\lt(\dd y^{25} + y^1 \dd y^{23}\rt) \text{ near the boundary of } \ca{U}_{\bf a}.
\ee

\noindent\ul{\(\ca{U}_{\bf a}\big\osr\ca{W}_{\bf a}\) for \({\bf a}\in\fr{A}\)}: Finally, using \erefs{ph-mu} and \eqref{h-ph}, one finds that on \(\ca{U}_{\bf a}\big\osr\ca{W}_{\bf a}\) for \({\bf a}\in\fr{A}\):
\ew
\ca{\ph}^\mu - \ca{\ph} = \lt(\mu^6 - 1\rt)\dd y^{123} + \dd\lt\{\lt[1-f\lt(\frac{r}{\ep}\rt)\rt]\lt(\tfrac{1}{2}\lt(y^1\rt)^2\dd y^{47}\rt)\rt\}.
\eew
Thus \(\ca{\ph}^\mu - \ca{\ph} = \dd\vpi\), where:
\ew
\vpi =&\ \frac{\mu^6-1}{2}\lt(y^1\dd y^{23} - y^2 \dd y^{13}\rt) + \lt(\mu^6-1\rt)\dd\lt(f\lt(\frac{r}{\ep}\rt)\lt(y^2\dd y^5 + \tfrac{1}{2}y^1y^2\dd y^3\rt)\rt)\\
& \hs{85mm} + \lt[1-f\lt(\frac{r}{\ep}\rt)\rt]\lt(\tfrac{1}{2}\lt(y^1\rt)^2\dd y^{47}\rt).
\eew
This satisfies:
\cag\label{BC3}
\vpi = \frac{\mu^6-1}{2}\lt(y^1\dd y^{23} - y^2 \dd y^{13}\rt) + \lt(\tfrac{1}{2}\lt(y^1\rt)^2\dd y^{47}\rt) \text{ near the boundary of } \ca{W}_{\bf a} \text{ and}\vs{3mm}\\
\vpi = \lt(\mu^6-1\rt)\lt(\dd y^{25} + y^1 \dd y^{23}\rt) \text{ near the boundary of } \ca{U}_{\bf a}.
\caag

Combining \erefs{BC1}, \eqref{BC2} and \eqref{BC3}, one sees that \(\vpi\) defines a smooth 2-form on all of \(\ca{\M}\) such that:
\ew
\ca{\ph}^\mu - \ca{\ph} = \dd\vpi,
\eew
as required.

\end{proof}

This completes the proof of \tref{FFKM-UA}.

\begin{Rk}\label{LF-Rk}
Recall that, for a closed \g\ 3-form \(\ph\), the Laplacian flow of \(\ph\) is the solution of the evolution PDE \cite[\S6]{RoG2S}:
\ew
\frac{\del \ph(t)}{\del t} = \De_{\ph(t)}\ph(t) = -\dd\Hs_{\ph(t)}\dd\Hs_{\ph(t)}\ph(t) \et \ph(0) = \ph.
\eew
Laplacian flow can be regarded as the gradient flow of \(\CH\) \cite[\S1.5]{LFfCG2S:STB}; in particular, \(\CH\) increases strictly along the flow.  Accordingly, Laplacian flow has been used in the literature to provide examples of 7-manifolds on which \(\CH\) is unbounded above; see, e.g., \cite[\S6]{RoG2S} and \cite[\S5]{ACoESttG2LF}.

The family \(\ph(\al,\be,\la;\mu)\) constructed in \sref{N-mfld} can also be interpreted via Laplacian flow.  Using \lref{scaling-lem}, \eref{dRI} and \eref{dSU2}, one may compute that:
\ew
-\dd\Hs_{\ph(\al,\be,\la;\mu)}\dd\Hs_{\ph(\al,\be,\la;\mu)}\ph(\al,\be,\la;\mu) = \frac{4\lt(\la\ol{\la}\rt)^\frac{2}{3}}{\al\mu^2} g^1 \w \om.
\eew
Allowing \(\mu = \mu(t)\) gives:
\ew
\frac{\del \ph(\al,\be,\la;\mu)}{\del t} = 6\mu^5\frac{\dd \mu}{\dd t} \al g^1 \w \om.
\eew
Thus \(\ph(\al,\be,\la;\mu(t))\) is a flow line of Laplacian flow starting from \(\ph(\al,\be,\la)\) if \(\mu\) satisfies the ODE:
\ew
\frac{\dd \mu}{\dd t} = \frac{2\lt(\la\ol{\la}\rt)^\frac{2}{3}}{3\al^2\mu^7} \et \mu(0) = 1.
\eew
It follows that the Laplacian flow starting from \(\ph(\al,\be,\la)\) exists for all \(t>0\) and is given by:
\ew
\ph\lt(\al,\be, \la; \sqrt[8]{\frac{16\lt(\la\ol{\la}\rt)^\frac{2}{3}t}{3\al^2} + 1}\rt).
\eew

In general, however, Laplacian flow can only be explicitly solved on manifolds with a high degree of symmetry, and thus cannot be used to investigate the unboundedness above of \(\CH\) on more complicated manifolds.  As a illustration, note that even at the level of the manifold \((\M,\vph)\):
\ew
\De_{\vph}\vph = -\dd \Hs_{\vph} \dd \Hs_\vph\vph = 2 \th^{123} + 2 \th^{145} - \th^{136} + \th^{127},
\eew
and thus the equation \(\frac{\del \vph_t}{\del t} = \De_{\vph_t}\vph_t\) cannot (to the author's knowledge) explicitly be solved starting from \(\vph_0 = \vph\).  However, \sref{FFKM} has shown that the scaling arguments described in \sref{UB-Md} can be applied successfully to prove the unboundedness above of \(\CH\) on \(\lt(\ca{\M},\ca{\ph}\rt)\).  This suggests that \pref{UB-Md-Prop} is a more widely applicable technique for proving the unboundedness above of \(\CH\) than Laplacian flow.
\end{Rk}

\section{The large volume limit of \(\lt(\ca{\M},\ca{\ph}\rt)\)}

The purpose of this section is to prove \tref{FFKM-CT}.  Essential to the proof is the main result of the author's recent paper \cite{AGCRfOvSF}.  I begin, therefore, by briefly recounting the relevant definitions therein; for further details, the reader may consult \cite{AGCRfOvSF} directly.

Let \(E\) be a closed orbifold.  The collection of smooth points of \(E\) (i.e.\ points with trivial orbifold group) comprise a smooth manifold denoted \(E({\bf 1})\), which is open and dense in \(E\).  More generally, given any isomorphism class \([\Ga]\) of finite groups, the connected components of the set \(E([\Ga])\) of all points in \(E\) with orbifold group representing the class \([\Ga]\) also comprise smooth manifolds.  These connected components partition \(E\); moreover \(E([\Ga])\) is non-empty for only finitely many classes \([\Ga]\), by compactness of \(E\).  This is the simplest example of a stratification of \(E\), and is termed the canonical stratification of \(E\).  More generally, a stratification of \(E\) is a partition of \(E\) into a finite number of disjoint, connected subsets \(\lt(E_i\rt)_i\) such that each \(E_i\) is a smooth submanifold of \(E([\Ga])\) for some \(\Ga\).

Fix a stratification \(\Si = \lt(E_i\rt)_i\) of \(E\).  A stratified Riemannian metric \(\h{g} = \{g_i\}_i\) on \(E\) is the data of a Riemannian metric \(g_i\) on each stratum \(E_i\) such that, for each \(i\), there exists a continuous orbifold Riemannian metric \(\ol{g}_i\) on \(E\) whose tangential component along \(E_i\) is \(g_i\).  Likewise, a stratified Riemannian semi-metric \(\h{g} = \{g_i\}_i\) is the data of a Riemannian semi-metric \(g_i\) on each stratum \(E_i\) satisfying the analogous condition that for each \(i\), there exists a continuous orbifold Riemannian semi-metric \(\ol{g}_i\) on \(\T E\) whose tangential component along \(E_i\) is \(g_i\).  A stratified distribution \(\mc{D}\) on \(E\) is a subbundle (in the orbifold sense) of the tangent bundle of \(E\) with the property that, for each \(i\): \(\mc{D} \cap \T E_i \cc \T E_i\) is a distribution in the usual manifold sense.  Given a stratified distribution \(\mc{D}\) and a \srsm\ \(\h{g}\), I term \(\h{g}\) regular \wrt\ \(\mc{D}\) if for each \(i = 0,...,n\) the bilinear form \(\ol{g}_i\) vanishes along \(\mc{D}\) but defines an inner product transverse to \(\mc{D}\).  (In particular, note that this implies that the kernel of the Riemannian semi-metric \(g_i\) as a bilinear form on \(E_i\) is precisely \(\mc{D}_i = \mc{D} \cap \T E_i\).)

In a similar way, a stratified quasi-Finslerian structure on \(E\) is the data of a continuous map \(\mc{L}_i: \T E_i \to [0,\0)\) for each \(i\) which is positively-homogeneous of degree 1 (i.e.\ for all \(\la \in \bb{R}\) and \(v \in \T E_i\): \(\mc{L}_i(\la v) = |\la| \mc{L}_i(v)\)) and positive definite (i.e.\ it vanishes precisely on the zero section of \(\T E_i\)) satisfying the property that for each continuous orbifold Riemannian metric \(h\) on \(\T E\), there exists a continuous function \(C_i:\ol{E_i} \to (0,\infty)\) for each \(i\) such that
\ew
\frac{1}{C_i} \|-\|_h \le \mc{L}_i \le C_i\|-\|h \on E_i.
\eew

As for usual \rmm s, both \srm s and \sqfs s can be integrated along piecewise-\(C^1\) curves in \(E\) and hence define natural metrics on the orbifold \(E\), whose induced topology coincides with the original topology on \(E\).  Likewise, a \srsm\ on \(E\) induces a semi-metric on \(E\) in the obvious way.

The following simple example of the above definitions, taken from \cite[Remark 4.8]{AGCRfOvSF}, will be of particular importance below.

\begin{Ex}\label{refinement}
Let \((E,\Si = \{E_i\}_i)\) be a stratified orbifold and let \(g\) be a Riemannian (semi-)metric on \(E\).  Setting \(g_i = g|_{E_i}\), where \(g|_{E_i}\) denotes the tangential component of \(g\) along \(E_i\), the data \(\h{g} = (g_i)_i\) defines a stratified Riemannian (semi-)metric on \(E\).  Moreover, the (semi-)metrics on \(E\) induced by \(g\) and \(\h{g}\) coincide.
\end{Ex}

Return now to the manifold \(\ca{\M}\).  Recall that there is a natural fibration \cite[\S9]{ACG2CMwFBNb1eq1}:
\ew
\bcd[row sep = 0pt]
q:\M \ar[r] & \bb{T}^3\\
\Ga\1(x^1,...,x^7) \ar[r, maps to] & \lt(\frac{x^1}{2} + \bb{Z},x^2 + \bb{Z}, x^3 + \bb{Z}\rt)
\ecd
\eew
with (non-calibrated coassociative) fibres diffeomorphic to \(\bb{T}^4\). Let \(\mc{I}\) denote the involution of \(\M\) defined in \eref{FFKMInv} and define a non-free involution \(\fr{I}\) of \(\bb{T}^3\) by acting on the first two factors of \(\bb{T}^3\) by \(-\Id\) and on the final factor by \(\Id\). Then \(q \circ \mc{I} = \fr{I} \circ q\) and so \(q\) descends to define a singular fibration:
\ew
\h{q}:\widehat{\M} \to \lqt{\bb{T}^3}{\fr{I}} = \lt(\lqt{\bb{T}^2}{\{\pm1\}}\rt) \x S^1 =B,
\eew
with \(\lqt{\bb{T}^2}{\{\pm1\}}\) being homeomorphic (although obviously not diffeomorphic) to \(\bb{CP}^1\).  The fibres of \(\h{q}\) are all path-connected, the generic fibres being 4-tori and the fibres over the singular locus of \(B\) being diffeomorphic to \(\lt(\lqt{\bb{T}^2}{\{\pm1\}}\rt) \x \bb{T}^2\).  Combining \(\h{q}\) with the natural `blow-down' map \(\rh:\ca{\M} \to \widehat{\M}\) similarly yields a fibration \(\pi\) of \(\ca{\M}\) over \(B\):
\ew
\bcd
\ca{\M} \ar[r, ^^22\rh^^22] \ar[rd, ^^22\pi^^22] & \h{\M} \ar[d, ^^22\h{q}^^22]\\
 & \lt(\lqt{\bb{T}^2}{\{\pm1\}}\rt) \x S^1.
\ecd
\eew

Away from the exceptional locus of \(\ca{\M}\), the map \(\pi\) is a smooth surjection with fibre \(\bb{T}^4\).  Near the exceptional locus, the map \(\pi\) is modelled on:
\ew
\tld{X} \x \bb{T}^3 \to \lt(\lqt{\bb{C}^2}{\{\pm1\}}\rt) \x \bb{T}^3 \oto{proj_1 \x proj_1} \lt(\lqt{\bb{C}}{\{\pm1\}}\rt) \x S^1,
\eew
where \(\tld{X}\) is the blow-up of \(\lqt{\bb{C}^2}{\{\pm1\}}\) at the origin as in \sref{FFKM-constr}.  The fibre of \(\tld{X} \to \lqt{\bb{C}^2}{\{\pm1\}} \to \lqt{\bb{C}}{\{\pm1\}}\) over 0 is the union of the proper transform of \(\lqt{(\{0\}\x\bb{C})}{\{\pm1\}}\) -- denoted \(\lt(\lqt{(\{0\}\x\bb{C})}{\{\pm1\}}\rt)_{PT}\) -- and the exceptional divisor \(\bb{CP}^1\) intersecting transversally at a single point; hence for each \(y^3 \in S^1\), the fibre of \(\pi\) over \(\{0\} \x \{y^3\}\) is the union of \(\lt(\lqt{(\{0\}\x\bb{C})}{\{\pm1\}}\rt)_{PT} \x \{y^3\} \x \bb{T}^2\) and \(\bb{CP}^1 \x \{y^3\} \x \bb{T}^2\) intersecting transversally along a single copy of \(\bb{T}^2\). It follows that the singular fibres of \(\pi\) are homeomorphic to four copies of \(\bb{CP}^1 \x \bb{T}^2\) intersecting a fifth copy of \(\bb{CP}^1 \x \bb{T}^2\) transversally along four distinct copies of \(\bb{T}^2\).\footnote{Note that \cite[p.\ 35]{ACG2CMwFBNb1eq1} contains an error in its description of the singular fibres of \(\pi\), of which the authors of \cite{ACG2CMwFBNb1eq1} have been informed.}

Using the above terminology, I now state the following, refined, version of \tref{FFKM-CT}:

\begin{Thm}\label{FFKM-CT2}
Let \(\lt(\ca{\M},\ca{\ph}^\mu\rt)_{\mu \in [1,\infty)}\) be the family constructed in the proof of \tref{FFKM-UA}.  Then the large volume limit of \(\lt(\ca{\M}, \ca{\ph}^\mu\rt)\) corresponds to an adiabatic limit of the fibration \(\pi\).  Specifically:
\ew
\lt(\ca{\M}, \mu^{-6}\ca{\ph}^\mu\rt) \to \lt(B,\h{\mc{L}}\rt) \as \mu \to \infty
\eew
in the Gromov--Hausdorff sense, where \(\h{\mc{L}}\) is the \sqfs\ on \(B\) (\wrt\ the canonical stratification of \(B\)) defined below in \eref{mcL} and is, in particular, Euclidean outside a neighbourhood of the singular locus of \(B\).
\end{Thm}

Consider the map \(\fr{f}:\M \to S^1\) given by:
\ew
\Ga\1(x^1,...,x^7) \mt x^3 + \bb{Z}.
\eew
Then \(\fr{f}\) descends to a map \(\h{\fr{f}}: \h{\M} \to S^1\).  Define \(\ca{\fr{f}} = \h{\fr{f}} \circ \rh: \ca{\M} \to S^1\) and, for each \({\bf a} \in \fr{A}\), define:
\ew
\ca{\fr{f}}_{\bf a} = \ca{\fr{f}}|_{\ca{W}_{\bf a}} \et \h{\fr{f}}_{\bf a} = \h{\fr{f}}|_{\h{W}_{\bf a}}.
\eew
Explicitly, the maps \(\ca{\fr{f}}_{\bf a}\) and \(\h{\fr{f}}_{\bf a}\) may be described as follows: writing \(\h{W}_{\bf a} \cong T_{y^3,y^4,y^7} \x \lt(\lqt{\B^4_{\frac{\ep}{2}}}{\{\pm1\}}\rt)\), one finds:
\ew
\h{\fr{f}}_{\bf a} : \h{W}_{\bf a} &\to S^1\\
\lt(y^i\rt) &\mt y^3
\eew
and similarly for \(\ca{\fr{f}}_{\bf a}\).

Now recall the map \(\h{q}: \h{\M} \to B\).  Since \(\h{q}\) is induced by the submersive map \(q: \M \to \bb{T}^3\), it provides an example of a weak submersion in the sense of \cite[Defn.\ 3.3]{AGCRfOvSF}.  In particular, the map \(\h{q}\) induces a natural stratification \(\Si\) on \(\h{\M}\) by `pulling back' the canonical stratification on \(B\): explicitly, the strata of \(\Si\) consist firstly of the pre-image under \(\h{q}\) of the smooth locus of \(B\), i.e.\ the collection of all smooth fibres of the map \(\h{q}\).  Secondly, it consists of the smooth locus of the four singular fibres of \(\h{q}\).  Finally, it consists of the 16 components of the singular locus of \(\h{\M}\). By pulling \(\Si\) back along the blow-down map \(\rh: \ca{\M} \to \h{\M}\), one obtains a stratification of \(\ca{\M}\), say \(\Si'\).  Explicitly, the stratification \(\Si'\) consists of firstly the collection of all smooth fibres of \(\ca{q}\), secondly the four singular fibres of \(\ca{q}\) with their exceptional loci removed, and thirdly the 16 exceptional loci of \(\ca{\M}\).

Now let \(g^\mu\) be the Riemannian metric on \(\ca{\M}\) induced by the \g\ 3-form \(\mu^{-6}\ca{\ph}\) and write \(\h{g}^\mu\) for the induced stratified Riemannian metric on \(\ca{\M}\) as in \exref{refinement}.  Write \(d^\mu\) for the metric on \(\ca{\M}\) induced by \(\h{g}^\mu\).  For all \(k \in [1,\infty)\), consider the space:
\ew
\ca{\M}^{(k)} = \lt.\ca{\M}\m\osr \coprod_{{\bf a} \in \fr{A}} \ca{W}_{{\bf a},k}\rt.
\eew
where \(\ca{W}_{{\bf a},k}\) was defined in \sref{FFKMUBpf} and write \(\h{W}_{{\bf a},k} = \rh\lt(\ca{W}_{{\bf a},k}\rt)\).  Write \(\mc{D}\) for the distribution over \(\h{\M}\) given by \(\ker\dd\h{q}\) and note that \(\mc{D}\) is stratified \wrt\ \(\Si\).  Suppose that there exists a stratified Riemannian semi-metric \(\h{g}^\infty\) on \(\h{\M}\) such that the following five conditions hold:\\

\begin{Conds}\label{CT-Cond}~
\begin{enumerate}
\item \(\h{g}^\infty\) is regular \wrt\ \(\mc{D}\);

\item  On each \(\ca{\M}^{(k)}\), \(k \in [1,\0)\):
\ew
\h{g}^\mu \to \h{g}^\infty \text{ uniformly as } \mu \to \infty,
\eew
and there exist constants \(\La_\mu(k) \ge 0\) such that:
\e\label{La-cvgce}
\lim_{\mu \to \infty} \La_\mu(k) = 1 \et \h{g}^\mu \ge \La_\mu(k)^2 \h{g}^\infty \text{ for all } \mu \in [1,\infty),
\ee
where \(\h{g}^\infty\) is regarded as a stratified Riemannian semi-metric on \(\ca{\M}^{(k)}\) using the blow-up map \(\rh\);

\item
\ew
\lim_{k \to \infty} ~ \limsup_{\mu \to \infty} ~\max_{{\bf a} \in \fr{A}} ~ \sup_{p \in S^1} ~ \diam_{d^\mu} \lt[ \ca{\fr{f}}_{\bf a}^{-1}(\{p\}) \cap \ca{W}_{{\bf a},k} \rt] = 0;
\eew

\item
\ew
\lim_{k \to \infty} ~ \max_{{\bf a} \in \fr{A}} ~ \sup_{p \in S^1} ~ \diam_{d^\infty}\lt[ \h{\fr{f}}_{\bf a}^{-1}(\{p\}) \cap \h{W}_{{\bf a},k} \rt] = 0;
\eew

\item
\ew
\lim_{k \to \infty} ~ \limsup_{\mu \to \infty} ~\max_{{\bf a} \in \fr{A}} ~ \sup_{\del_{{\bf a},k}} \lt| d^\mu - d^\infty \rt| = 0,
\eew
where \(d^\infty\) is the semi-metric on \(\h{\M}\) induced by \(\h{g}^\infty\) (regarded as a semi-metric on \(\del_{{\bf a},k}\) using the identification \(\rh\)) and, for simplicity of notation, I write \(\del_{{\bf a},k}\) for the subset \(\del \ca{W}_{{\bf a},k} \overset{\rh}{\cong} \del \h{W}_{{\bf a},k}\).
\end{enumerate}
\end{Conds}

Given such a \(\h{g}^\0\), \tref{FFKM-CT2} is a direct application of the main result (Theorem 5.3) of \cite{AGCRfOvSF}, with \(\h{\mc{L}} = (\mc{L}_i)\) given explicitly as follows:  fix a stratum \(B_i\) in the canonical stratification of \(B\) and write \(\pi^{-1}(B_i) = \bigcup_{j = 0}^k \h{\M}_j\), where each \(\h{\M}_j\) is a stratum in the induced stratification \(\Si\) of \(\h{\M}\).   Then given \(p\in B_i\) and \(u \in \T_pB_i\), define
\e\label{mcL}
\mc{L}_i(u) = \min_{j = 0}^k ~ \inf \lt\{\|u'\|_{g^\infty_j}~\m|~u' \in \T_x \h{\M}_j \text{ such that } \dd\h{q}(u') = u \text{ (and, in particular, } \h{q}(x) = p)\rt\}.
\ee
(For comparative purposes, the reader may wish to note that \(\ca{M}\), \(\h{\M}\) and \(\rh\) above correspond to \(E_1\), \(E_2\) and \(\Ph\) in \cite{AGCRfOvSF} respectively, and that \(\ca{W}_{{\bf a},k}\), \(\h{W}_{{\bf a},k}\) above correspond respectively to \(U_1^{\lt(k^{-1}\rt)}(j)\), \(U_2^{\lt(k^{-1}\rt)}(j)\) in \cite{AGCRfOvSF}, the rest of the notation being obviously equivalent.)

The remainder of this paper, therefore, will be devoted to establishing the five conditions in \condsref{CT-Cond}.

\subsection{Bounding the volume form induced by \(\ca{\om}_t\)}

Recall the K\"{a}hler forms \(\ca{\om}_t\) interpolating between \(\h{\om}\) and the Eguchi--Hanson metric \(\tld{\om}_t\) on \(\tld{X}\) used in the construction of \(\ca{\ph}\).  For the purpose of proving \condsref{CT-Cond}, I require a lower bound on the volume forms induced by \(\ca{\om}_t\) which is both sharp and \(t\)-independent.  The purpose of this subsection is to derive this bound.

I begin by providing an alternative perspective on the \EH\ forms \(\tld{\om}_t\).  Consider the problem of constructing Ricci-flat K\^^22{a}hler metrics on \(\lt(\lqt{\bb{C}^2}{\{\pm1\}}\rt)\osr\{0\}\).   One possible approach is as follows:  suppose one is given a closed, positive, real \((1,1)\)-form \(\tld{\om}\) on \(\lt.\lt(\lqt{\bb{C}^2}{\{\pm1\}}\rt)\m\osr\{0\}\rt.\) with the following property:
\e\label{norm}
\tld{\om}^2 = \lt(\fr{Re}\h{\Om}\rt)^2 = \lt(\fr{Im}\h{\Om}\rt)^2 = 2vol_0
\ee
(here \(\h{\Om} = \dd w^1 \w \dd w^2\) as usual and \(vol_0\) denotes the Euclidean volume form on \(\lqt{\bb{C}^2}{\{\pm1\}}\) given by \(\dd x^1 \w \dd y^1 \w \dd x^2 \w \dd y^2\), where \(w^1 = x^1 + iy^1\) and \(w^2 = x^2 + iy^2\)).   Then the triple \(\lt(\tld{\om},\fr{Re}\h{\Om},\fr{Im}\h{\Om}\rt)\) defines an \(\Sp(1)\) structure on \(\lt.\lt(\lqt{\bb{C}^2}{\{\pm1\}}\rt)\m\osr\{0\}\rt.\) and the condition \(\dd\tld{\om} = \dd\fr{Re}\h{\Om} = \dd\fr{Im}\h{\Om} = 0\) implies the vanishing of the torsion of this \(\Sp(1)\)-structure \cite[Lem.\ 6.8, p.\ 91]{TSDEoaRS}, i.e.\ the triple defines a hyper-K\^^22{a}hler structure.   This implies that the holonomy of the K\^^22{a}hler metric induced by \(\tld{\om}\) is contained in \(\Sp(1) = \SU(2)\) which is a Ricci-flat holonomy group \cite[p.\ 55]{RHG&CG}.

To ensure that \(\tld{\om}\) is closed and a \((1,1)\)-form, apply the ansatz:
\e\label{ca-om}
\tld{\om} = \frac{1}{4}\dd\dd^c\lt[a(\la)\rt] = a'(\la) \h{\om} + \frac{1}{4}a''(\la) \dd(\la) \w \dd^c(\la),
\ee
where \(\la = r^2 = \lt|w^1\rt|^2 + \lt|w^2\rt|^2\) is the radial distance squared from \(0 \in \lqt{\bb{C}^2}{\{\pm1\}}\), \(a:(0,\infty) \to \bb{R}\) is a smooth, real-valued function and \(\h{\om} = \frac{i}{2}\lt(\dd w^1 \w \dd\ol{w}^1 + \dd w^2 \w \dd\ol{w}^2\rt)\) is the standard Euclidean K\^^22{a}hler form on \(\lqt{\bb{C}^2}{\{\pm1\}}\) (in particular, note that \(\tld{\om}\) only depends on \(a'\), rather than on \(a\) itself).   A long, but elementary, calculation then verifies:
\e\label{Ricci}
\tld{\om}^2 = \frac{\frac{\dd}{\dd\la}\lt[\la^2a'(\la)^2\rt]}{\la}vol_0.
\ee
Thus \eref{norm} is reduced to the second-order ODE:
\ew
\frac{\dd}{\dd\la}\lt[\la^2a'(\la)^2\rt] = 2\la.
\eew
Integrating this equation gives:
\e\label{EH-pot}
a'_t(\la) = \sqrt{1 + \frac{t^4}{\la^2}}
\ee
for some \(t\ge0\) (the positive square root is needed to ensure that \(\tld{\om}\) is a positive \((1,1)\)-form).  In the case where \(t = 0\), \(\tld{\om}_0 = \frac{1}{4}\dd\dd^c a_0(\la)\) is simply the Euclidean form \(\h{\om}\), however in the case \(t>0\), one recovers the Eguchi--Hanson metrics \(\tld{\om}_t\) defined in \sref{FFKM-constr}, \eref{EH-defn}.

\begin{Rk}
The fact that the 1-parameter family \(\tld{\om}_t\) of \EH\ metrics can naturally be extended to include the metric \(\h{\om}\) (corresponding to the case \(t=0\)) may appear initially surprising, since the metrics \(\tld{\om}_t\) for \(t>0\) are defined on the manifold \(\tld{X}\) whereas the metric \(\h{\om}\) is defined on the orbifold \(\lqt{\bb{C}^2}{\{\pm1\}}\).  However as \(t\to0\), the diameter of the exceptional divisor \(\fr{E}\) tends to zero and so the manifolds \(\lt(\tld{X},\tld{\om}_t\rt)\) converge in the Gromov--Hausdorff sense to the orbifold \(\lt(\lqt{\bb{C}^2}{\{\pm1\}},\h{\om}\rt)\), and hence the result is not as surprising as first appears. See also \cite[p.\ 21]{K3S&SD}.
\end{Rk}

Using this perspective, I now prove the required bound on the volume form of \(\ca{\om}_t\):

\begin{Prop}\label{tld-om-const}
There exist \(R>0\) and \(\up\in(0,1)\), independent of \(t > 0\), such that the following is true:

For every \(t>0\), there exists a closed, real, positive \((1,1)\)-form \(\ca{\om}_t\) on \(\tld{X}\) satisfying the following three properties:
\begin{enumerate}
\item \(\ca{\om}_t = \tld{\om}_t\) on the region \(\lt\{p \in \tld{X}~\m|~r(p) \le \frac{tR}{2}\rt\}\);
\item \(\ca{\om}_t = \hat{\om}\) on a neighbourhood of the region \(\lt\{p \in \tld{X}~\m|~r(p) \ge tR\rt\}\);
\item \(\ca{\om}_t^2 \ge 2\up^2 vol_0\) on all of \(\tld{X}\), with equality holding at least on \(\lt\{p \in \tld{X}~\m|~r(p) = \fr{r}_t\rt\}\), where \(\fr{r}_t \in \lt(\frac{tR}{2}, tR\rt)\).
\end{enumerate}
\end{Prop}

\begin{Rk}
It is not difficult to show that any \(\ca{\om}_t\) of the form considered in \eref{ca-om} which satisfies points (1) and (2) cannot also satisfy \(\ca{\om}_t^2 \ge 2vol_0\) on all of \(\tld{X}\), and thus there is some \(\up = \up(t) \in (0,1)\) such that \(\ca{\om}_t^2 \ge 2\up(t)^2vol_0\), with equality realised at some point.  The significance of the above result is that \(\up\) can be taken to be independent of \(t\).
\end{Rk}

\begin{proof}
Let \(\up \in (0,1)\) be chosen later and write \(c = 2 - 2\up^2\).   I begin with the following auxiliary claim:

\begin{Cl}\label{h-prop}
There exists \(R_0 = R_0(c)>0\) (i.e.\ depending on \(c\) but independent of \(t\)) such that the following is true:

For all \(R \ge R_0\) and all \(t>0\), there exists a smooth function \(k_t:[0,t^2 R^2] \to (-\infty,0]\) satisfying the following three properties:
\begin{itemize}
\item[(i)] \(k_t \equiv 0\) on \(\lt[0,\frac{t^2 R^2}{4}\rt]\) and on a neighbourhood of \(t^2 R^2\) in \([0,t^2 R^2]\);
\item[(ii)] \(\bigintsss_0^{t^2 R^2} k_t(\la) \dd \la = -t^4\);
\item[(iii)] \(k_t(\la) \ge -c\la\) for all \(\la \in [0,t^2 R^2]\), with equality holding at some point \(\fr{r}_t^2 \in [0,t^2 R^2]\).
\end{itemize}
\end{Cl}

\begin{proof}[Proof of Claim]
For such a function \(k_t\) to exist, it is necessary and sufficient that:
\ew
\bigintsss_{\frac{t^2 R^2}{4}}^{t^2 R^2} c \la\ \dd \la > t^4.
\eew
(Given this, one constructs \(k_t\) by appropriately smoothing out the piecewise constant function:
\ew
\hat{k}_t(\la) =
\begin{dcases*}
0 & on \(\lt[0,\frac{t^2 R^2}{4}\rt)\)\\
-c\la & on \(\lt[\frac{t^2 R^2}{4},t^2 R^2\rt)\)\\
0 & at \(\la = t^2 R^2\)
\end{dcases*}
\eew
ensuring that \(\bigintsss_0^{t^2 R^2} k_t(\la) \dd \la = -t^4\) and that \(k_t(\fr{r}^2_t) = -c\fr{r}^2_t\) still holds at some point \(\fr{r}^2_t \in [0,t^2R^2]\).  The converse is clear.)   However:
\ew
\bigintsss_{\frac{t^2 R^2}{4}}^{t^2 R^2} c \la\ \dd \la = \frac{15ct^4 R^4}{32} > t^4
\eew
whenever \(R_0(c) > \sqrt[4]{\frac{32}{15c}}\), completing the proof {\color{red}of the claim.}

\let\qed\relax
\end{proof}

The proof of \pref{tld-om-const} now proceeds as follows.  Define:
\ew
h_t(\la) = \bigintsss_0^\la k_t(s) \dd s \ge -t^4
\eew
and define:
\ew
\ca{\om}_t = \frac{1}{4}\dd\dd^c\lt[\al_t(r^2)\rt],
\eew
where \(\al_t:(0,t^2 R^2] \to \bb{R}\) satisfies:
\e\label{al-defn}
\al'_t(\la) = \sqrt{1 + \frac{t^4}{\la^2} + \frac{h_t(\la)}{\la^2}}.
\ee
For \(\la \in \lt(0,\frac{t^2 R^2}{4}\rt]\), \(h_t \equiv 0\) by \clref{h-prop}(i), hence \(\al'_t = a'_t\) (see \eref{EH-pot}) and whence \(\ca{\om}_t = \tld{\om}_t\) for \(r \in \lt(0,\frac{t R}{2}\rt)\).  It follows that \(\ca{\om}_t\) extends over the exceptional divisor in \(\tld{X}\) and satisfies property (1) in \pref{tld-om-const}.  Similarly, for \(\la\) in a neighbourhood of \(t^2 R^2\) in \(\lt(0,t^2 R^2\rt]\), \(h_t \equiv -t^4\) by \clref{h-prop}(i)--(ii), hence \(\al'_t = a_0\) and whence \(\ca{\om}_t = \h{\om}\) for \(r\) in a neighbourhood of \(t R\) in \(\lt(0,t R\rt]\).  Thus \(\ca{\om}_t\) extends over the whole of \(\tld{X}\) and satisfies property (2) in \pref{tld-om-const}.  Thus to complete the proof of \pref{tld-om-const}, it suffices to prove that \(\ca{\om}_t\) is positive and satisfies property (3).

To prove positivity, I use the following well-known fact: if \(\om\) is a real, positive \((1,1)\)-form on an almost complex manifold \(\M\) and \(\om'\) is a real \((1,1)\)-form on \(\M\) with \(|\om'-\om|_\om<1\), then \(\om'\) is also positive (this can be verified by working in local coordinates).  To apply this to \(\ca{\om}_t\), firstly note that, since \(\tld{\om}_t\) and \(\h{\om}\) are both positive, one can restrict attention to the region \(r \in \lt(\frac{t R}{2},t R\rt)\).   Thus it suffices to prove that:
\ew
\lt|\h{\om} - \lt(\al'_t(r^2) \h{\om} + \frac{1}{4}\al''_t(r^2) \dd(r^2) \w \dd^c(r^2)\rt)\rt|_{\h{\om}} <1 \hs{5mm} \text{for }r \in \lt(\frac{t R}{2},t R\rt).
\eew
Using the triangle inequality:
\e\label{om-bd-1}
\lt|\h{\om} - \lt(\al'_t(r^2) \h{\om} + \frac{1}{4}\al''_t(r^2) \dd(r^2) \w \dd^c(r^2)\rt)\rt|_{\h{\om}} \le |\al'_t(r^2) - 1|\cdot \lt|\h{\om}\rt|_{\h{\om}} + \frac{1}{4}\lt|\al''_t(r^2)\rt|\cdot\lt|\dd(r^2) \w \dd^c(r^2)\rt|_{\h{\om}}.
\ee
Using \eref{al-defn}, it follows that:
\e\label{om-bd-2}
\lt|\al'_t(r^2)-1\rt| &\le \frac{\sqrt{t^4 + h_t(r^2)}}{r^2} \le \frac{4}{R^2} \hs{5mm} \text{for } r \in \lt(\frac{t R}{2}, t R\rt),
\ee
since \(h_t \le 0\).  Using \eref{al-defn} once more, one sees:
\ew
\al''_t(r^2) = \frac{-\frac{t^4 + h_t(r^2)}{r^6} + \frac{h_t'(r^2)}{2r^4}}{\sqrt{1 + \frac{t^4}{r^4} + \frac{h_t(r^2)}{r^4}}}
\eew
and hence:
\ew
r^2\lt|\al''_t(r^2)\rt| \le \frac{\lt|t^4 + h_t(r^2)\rt|}{r^4} + \frac{|h'_t(r^2)|}{2r^2}.
\eew
Since \(-t^4 \le h_t(r^2) \le 0\), \(-cr^2 \le h_t'(r^2) \le 0\) and \(\frac{t R}{2} \le r \le t R\), it follows that:
\e\label{om-bd-3}
r^2\lt|\al''_t(r^2)\rt| \le \frac{16}{R^4} + \frac{c}{2}.
\ee
A simple calculation shows that \(\lt|\h{\om}\rt|_{\h{\om}} = \sqrt{2}\) and \(\lt|\dd(r^2) \w \dd^c(r^2)\rt|_{\h{\om}} \le 4Cr^2\) for some \(C>0\) independent of \(r^2\), \(R\) and \(t\). Thus, by combining \erefs{om-bd-1}, \eqref{om-bd-2} and \eqref{om-bd-3}:
\ew
\lt|\h{\om} - \ca{\om}_t\rt|_{\h{\om}} \le \frac{4\sqrt{2}}{R^2} + \frac{16C}{R^4} + \frac{Cc}{2}.
\eew
Define \(c = C^{-1}\) (note that \(C\) is independent of \(R\) and \(t\)) and choose \(R>R_0(c)\) such that:
\ew
\frac{4\sqrt{2}}{R^2} + \frac{16C}{R^4} < \frac{1}{2}.
\eew
Then \(\lt|\h{\om} - \ca{\om}_t\rt|_{\h{\om}} <1\) for \(r \in \lt(\frac{t R}{2}, t R\rt)\) and thus the positivity of \(\ca{\om}_t\) has been verified.\\

Finally, consider property (3) of \pref{tld-om-const}.   Clearly for \(r \notin \lt(\frac{t R}{2}, t R\rt)\), one has \(\ca{\om}_t^2 = 2vol_0\), since both \(\h{\om}\) and \(\tld{\om}_t\) have this property.  For \(r \in \lt(\frac{t R}{2}, t R\rt)\), by combining \eref{Ricci} and \eqref{al-defn}, one computes:
\ew
\tld{\om}_t^2 = \lt(2 + \frac{h'_t(r^2)}{r^2}\rt)vol_0.
\eew
Now by property (iii) in \clref{h-prop}, \(\frac{h'_t(r^2)}{r^2} \ge -c = 2\up^2 - 2\) with equality holding at \(r = \fr{r}_t \in [0,tR]\), as required.  This completes the proof.

\end{proof}

\subsection{Defining a suitable \(\h{g}^\infty\)}\label{bulk-conv}

For simplicity of notation, write \(\ca{\upphi}^\mu = \mu^{-6}\ca{\ph}^\mu\), so that \(g^\mu = g_{\ca{\upphi}^\mu}\), and write \(\ca{S} = \rh^{-1}\lt(\h{S}\rt)\) for the exceptional locus of \(\ca{\M}\).  The next task is to understand the limit of the Riemannian metrics \(g^\mu\) away from \(\ca{S}\).

Define:
\ew
\ca{\fr{W}}_{\bf a} = \bigcap_{k\ge1} \ca{W}_{{\bf a},k} \supset \ca{S},
\eew
so that in local coordinates:
\ew
\ca{\fr{W}}_{\bf a} = \lt\{y^1,y^2 = 0, \lt(y^5\rt)^2 + \lt(y^6\rt)^2 \le \frac{1}{2}\ep\rt\},
\eew
and write:
\ew
\lt.\ca{\M}\m\osr\coprod_{{\bf a}\in\fr{A}}\ca{\fr{W}}_{\bf a}\rt. &= \lt(\ca{\M}\m\osr\coprod_{{\bf a}\in\fr{A}}\ca{U}_{\bf a}\rt) \bigcup \lt(\coprod_{{\bf a}\in\fr{A}}\lt.\ca{U}_{\bf a}\m\osr\ca{W}_{\bf a}\rt.\rt) \bigcup \lt(\coprod_{{\bf a}\in\fr{A}}\lt.\ca{W}_{\bf a}\m\osr\ca{\fr{W}}_{\bf a}\rt.\rt)\\
&= \hs{8.5mm} \itr{\M} \hs{8.5mm} \bigcup \hs{7mm} \ca{\M}_{int} \hs{7mm} \bigcup \lt(\coprod_{{\bf a}\in\fr{A}}\lt.\ca{W}_{\bf a}\m\osr\ca{\fr{W}}_{\bf a}\rt.\rt).
\eew
I shall consider the behaviour of \(g^\mu\) on each of these three regions in turn.

\subsubsection{The region \(\itr{\M}\)}

Recall that on \(\itr{\M}\):
\ew
\ca{\vph}^\mu = \th^{123} + \mu^{-6}\lt(\th^{145} + \th^{167} - \th^{246} + \th^{257} + \th^{347} + \th^{356}\rt).
\eew
Using the \g-basis \(\lt(\th^1,\th^2,\th^3,\mu^{-3}\th^4,\mu^{-3}\th^5,\mu^{-3}\th^6,\mu^{-3}\th^7\rt)\), it follows that
\ew
g^\mu &= \lt(\lt(\th^1\rt)^{\ts2} + \lt(\th^2\rt)^{\ts2} + \lt(\th^3\rt)^{\ts2}\rt) + \mu^{-6}\lt(\lt(\th^4\rt)^{\ts2} + \lt(\th^5\rt)^{\ts2} + \lt(\th^6\rt)^{\ts2} + \lt(\th^7\rt)^{\ts2}\rt)\\
&\to \lt(\th^1\rt)^{\ts2} + \lt(\th^2\rt)^{\ts2} + \lt(\th^3\rt)^{\ts2} \text{ uniformly on } \itr{\M} \text{ as } \mu \to \infty.
\eew
Thus define:
\e\label{g-infty-1}
g^\infty = \lt(\th^1\rt)^{\ts2} + \lt(\th^2\rt)^{\ts2} + \lt(\th^3\rt)^{\ts2} \text{ on } \itr{\M}.
\ee
Then \(g^\mu - g^\infty\) is non-negative definite for all \(\mu\in[1,\infty)\) and thus \eref{La-cvgce} holds on the region \(\itr{\M}\) with \(\La_\mu = 1\) for all \(\mu\).  Moreover (see \cite[\S9]{ACG2CMwFBNb1eq1}) it can be shown that \(\ker\dd\h{q} = \<e_4,e_5,e_6,e_7\?\), where \((e_i)\) denotes the basis of left-invariant vector fields on \(\M\) dual to the basis of left-invariant 1-forms \((\th^i)\).  (Note that whilst the forms \(e_i\) do not themselves descend to the orbifold \(\h{\M}\), the distribution \(\<e_4,...,e_7\?\) is invariant under the involution \(\mc{I}\) defined in \eref{FFKMInv} and thus does descend to \(\h{\M}\).)  Thus from \eref{g-infty-1}, one sees that \(g^\infty\) is positive definite transverse to \(\ker\dd\h{q}\) on \(\itr{\M}\).

\subsubsection{The region \(\coprod_{{\bf a}\in\fr{A}}\lt.\ca{W}_{\bf a}\m\osr\ca{\fr{W}}_{\bf a}\rt.\)}

For simplicity, fix some choice of \({\bf a} \in \fr{A}\). Recall from \eref{outside-surg} that on the region \(\lt.\ca{W}_{\bf a}\m\osr\ca{W}_{{\bf a},\mu}\rt.\):
\ew
\ca{\upphi}^\mu =  \dd y^{123} + \mu^{-6}\lt\{\dd y^{145} + \dd y^{167} - \dd y^{246} + \dd y^{257} + \dd y^{347} + \dd y^{356} + y^1\dd y^{147}\rt\}.
\eew
In particular, for any given \(k\in[1,\infty)\) and all \(\mu\ge k\), since \(\lt.\ca{W}_{\bf a}\m\osr\ca{W}_{{\bf a},k}\rt. \pc \lt.\ca{W}_{\bf a}\m\osr\ca{W}_{{\bf a},\mu}\rt.\) one may calculate that on \(\lt.\ca{W}_{\bf a}\m\osr\ca{W}_{{\bf a},k}\rt.\):
\ew
g^\mu = \lt(1 - \frac{\lt(y^1\rt)^2}{4}\rt)^\frac{-1}{3}\bigg\{&\lt[\lt(\dd y^1\rt)^2 + \lt(\dd y^2\rt)^2 + \lt(\dd y^3\rt)^2 + y^1\dd y^1 \s \dd y^3\rt]\\
& + \mu^{-6}\lt[\lt(\dd y^4\rt)^2 + \lt(\dd y^5\rt)^2 + \lt(\dd y^6\rt)^2 + \lt(\dd y^7\rt)^2 + y^1 \dd y^4 \s \dd y^6 + y^1 \dd y^5 \s \dd y^7\rt]\bigg\}.
\eew
Define:
\e\label{g-infty-inner}
g^\infty = \lt(1 - \frac{\lt(y^1\rt)^2}{4}\rt)^\frac{-1}{3}\lt\{\lt(\dd y^1\rt)^2 + \lt(\dd y^2\rt)^2 + \lt(\dd y^3\rt)^2 + y^1\dd y^1 \s \dd y^3\rt\} \text{ on } \lt.\ca{W}_{\bf a}\m\osr\ca{\fr{W}}_{\bf a}\rt..
\ee
Then \(g^\mu \to g^\infty\) uniformly on \(\lt.\ca{W}_{\bf a}\m\osr\ca{W}_{{\bf a},k}\rt.\) as \(\mu \to \infty\). Moreover, when \(\mu \ge k\), \(g^\mu - g^\infty\) is non-negative definite for \(\ep>0\) sufficiently small, independent of \(\mu\) (where \(\ep\) is the size of the surgery region used in the construction of \(\ca{\M}\)).  Thus \eref{La-cvgce} holds on \(\lt.\ca{W}_{\bf a} \m\osr \ca{W}_{{\bf a},k}\rt.\) by setting \(\La_\mu(k) = 1\) for all \(\mu\in[k,\infty)\).  Moreover, using \erefs{Ph0} and \eqref{Ph1}, one can show that on the region \(\h{U}_{\bf a}\) for \({\bf a}\in\fr{A}\):
\e\label{ker-in-y's}
\ker\dd\h{q} = \<\frac{\del}{\del y^4}, \frac{\del}{\del y^5}, \frac{\del}{\del y^6}, \frac{\del}{\del y^7}\?.
\ee
Thus by \eref{g-infty-inner}, \(g^\infty\) is positive definite transverse to \(\ker\dd\h{q}\) on \(\lt. \ca{W}_{\bf a} \m\osr \ca{W}_{{\bf a},k} \rt.\) for all \(\ep>0\) sufficiently small, independent of \(\mu\).

\subsubsection{The region \(\ca{\M}_{int}\)}\label{Mint}

By analogy with the notation \(\ca{\upphi}^\mu = \mu^{-6}\ca{\ph}^\mu\), define \(\h{\upxi}^\mu\) on \(\ca{\M}_{int}\) by:
\ew
\h{\upxi}^\mu = \mu^{-6}\h{\xi}^\mu = \dd y^{123} + \mu^{-6}\lt\{\dd y^{145} + \dd y^{167} - \dd y^{246} + \dd y^{257} + \dd y^{347} + \dd y^{356}\rt\}.
\eew

Recall also that, by the proof of \pref{almostprod}, on the region \(\ca{\M}_{int}\):
\e\label{upphi}
\ca{\upphi}^\mu  = \h{\upxi}^\mu + \mu^{-6}\lt\{y^1 \dd y^{147} + \dd\lt[f\lt(\frac{r}{\ep}\rt)\al_{\bf a}\rt]\rt\},
\ee
where \(f\) satisfies \eref{f} and \(\al_{\bf a}\) (defined in \lref{quadlem}) is independent of \(\mu\) and at least quadratic in \((y^1,y^2,y^5,y^6)\).

To analyse the behaviour of \(g^\mu\) on \(\ca{\M}_{int}\) as \(\mu\to\infty\), it is useful to introduce a third \g\ 3-form \(\Xi^\mu\) on \(\ca{\M}_{int}\). To define \(\Xi^\mu\), firstly write \(\ca{\upphi}^\mu = \h{\upxi}^\mu + \mu^{-6}\sum_{1 \le i < j < k \le 7} \si_{ijk} \dd y^{ijk}\), where each coefficient \(\si_{ijk}\) is a smooth function on \(\ca{\M}_{int}\) independent of \(\mu\) and satisfying:
\e\label{lin-bd}
|\si_{ijk}| \le Cr
\ee
for some fixed \(C>0\), independent of \(\mu\), \(\ep\) and \(r\). Then define:
\e\label{ze}
\Xi^\mu = \h{\upxi}^\mu + \mu^{-6}\sum_{\substack{1 \le i \le 3\\ 4 \le j < k \le 7}} \si_{ijk}\dd y^{ijk}.
\ee
\begin{Lem}\label{upxi-vs-ze}
There exist constants \(C_1\) and \(C_2\), independent of \(\mu\) and \(r\), such that:
\begin{enumerate}
\item \(\lt\|\h{\upxi}^\mu - \Xi^\mu\rt\|_{\h{\upxi}^\mu} \le C_1r\);
\item \(\lt\|\ca{\upphi}^\mu - \Xi^\mu\rt\|_{\h{\upxi}^\mu} \le C_2\mu^{-3}r\).
\end{enumerate}
\end{Lem}
\begin{proof}
Observe that \(\h{\upxi}^\mu\) has the \g-basis \(\lt(\vth^1,...,\vth^7\rt) = \lt(\dd y^1, \dd y^2, \dd y^3, \mu^{-3}\dd y^4, \mu^{-3}\dd y^5, \mu^{-3}\dd y^6, \mu^{-3}\dd y^7\rt)\). \Wrt\ this basis, one can write:
\e\label{ze-vth}
\Xi^\mu = \vth^{123} + \vth^{145} + \vth^{167} - \vth^{246} + \vth^{257} + \vth^{347} + \vth^{356} + \sum_{\substack{1 \le i \le 3\\ 4 \le j < k \le 7}} \si_{ijk}\vth^{ijk}.
\ee
(1) then immediately follows from \eref{lin-bd}.

For (2), note that from \eref{ze}:
\ew
\ca{\upphi}^\mu - \Xi^\mu = \mu^{-6}\lt(\si_{123}\dd y^{123} + \sum_{\substack{1 \le i < j \le 3\\ 4 \le k \le 7}} \si_{ijk}\dd y^{ijk}\rt).
\eew
(The fact that there are no terms of the form \(\dd y^{ijk}\) with \(4 \le i < j < k \le 7\) follows from \eref{upphi} and the precise expression for \(\al_{\bf a}\) given in \lref{quadlem}.) Writing this in terms of the \g-basis \((\vth^i)\) gives:
\ew
\ca{\upphi}^\mu - \Xi^\mu = \mu^{-3}\lt(\mu^{-3}\si_{123} \vth^{123} + \sum_{\substack{1 \le i < j \le 3\\ 4 \le k \le 7}} \si_{ijk}\vth^{ijk}\rt),
\eew
from which, together with \eref{lin-bd}, (2) is immediately clear.

\end{proof}

Informally, \lref{upxi-vs-ze} says that (1) \(\Xi^\mu\) is of \g-type and `close' to \(\h{\upxi}^\mu\) if \(\ep>0\) is sufficiently small, uniformly in \(\mu\), and (2) the difference between \(\Xi^\mu\) and \(\ca{\upphi}^\mu\) is negligible as \(\mu\to\infty\).  Thus, insight into the behaviour of \(\ca{\upphi}^\mu\) as \(\mu\to\infty\) may be gained by studying \(\Xi^\mu\) as \(\mu\to\infty\).

\begin{Lem}
One can write:
\ew
g_{\Xi^\mu} = \sum_{1 \le i,j \le 3} (\de_{ij} + g_{ij}) \dd y^i \s \dd y^j + \mu^{-6}\sum_{4 \le i,j \le 7} (\de_{ij} + g_{ij}) \dd y^i \s \dd y^j
\eew
for some smooth functions \(g_{ij}\) on \(\ca{\M}_{int}\) independent of \(\mu\) and satisfying:
\e\label{lin-bd-2}
|g_{ij}| \le C_3r
\ee
for some constant \(C_3 > 0\) independent of \(\mu\), \(\ep\) and \(r\).
\end{Lem}

\begin{proof}
Recall \eref{ze-vth}:
\ew
\Xi^\mu = \vth^{123} + \vth^{145} + \vth^{167} - \vth^{246} + \vth^{257} + \vth^{347} + \vth^{356} + \sum_{\substack{1 \le i \le 3\\ 4 \le j < k \le 7}} \si_{ijk}\vth^{ijk}.
\eew
By this equation, one can automatically write:
\ew
g_{\Xi^\mu} = \sum_{1 \le i,j \le 7} (\de_{ij} + g_{ij}) \vth^i \s \vth^j
\eew
for some \(g_{ij}\) independent of \(\mu\) and satisfying \eref{lin-bd-2}. Recalling the definition of the \(\vth^i\), to complete the proof it suffices to prove that \(g_{ij} = 0\) if \(1 \le i \le 3\) and \(4 \le j \le 7\).

To this end, recall from \cite[\S2, Thm. 1]{MwEH} that:
\ew
g_{\Xi^\mu} vol_{\Xi^\mu} = \lt[(-)\hk \Xi^\mu \rt] \w \lt[(-)\hk \Xi^\mu \rt] \w \Xi^\mu.
\eew
Thus, it suffices to prove that for all \(1 \le i \le 3\) and \(4 \le j \le 7\):
\ew
\lt[\vth_i\hk \Xi^\mu \rt] \w \lt[\vth_j\hk \Xi^\mu \rt] \w \Xi^\mu = 0,
\eew
where \((\vth_1,...,\vth_7)\) is the basis of vectors dual to \((\vth^1,...,\vth^7)\). Define:
\ew
\mc{D} = \<\vth_1,\vth_2,\vth_3\? \et \mc{T} = \<\vth_4,...,\vth_7\?
\eew
so that \(\T\ca{\M}_{int} = \mc{D} \ds \mc{T}\). By examining \eref{ze-vth}, one can verify that:
\ew
\Xi^\mu \in \ww{3}\mc{D}^* + \mc{D}^* \ts \ww{2}\mc{T}^*.
\eew
It follows that for \(1 \le i \le 3\):
\ew
\vth_i \hk \Xi^\mu \in \ww{2}\mc{D}^* + \ww{2}\mc{T}^* \pc \ww{2}\T^*\ca{\M}_{int}
\eew
and that for \(4 \le j \le 7\):
\ew
\vth_j \hk \Xi^\mu \in \mc{D}^* \ts \mc{T}^* \pc \ww{2}\T^*\ca{\M}_{int}.
\eew
Thus:
\ew
\lt[\vth_i\hk \Xi^\mu \rt] \w \lt[\vth_j\hk \Xi^\mu \rt] \w \Xi^\mu \in \ww{6}\mc{D}^* \ts \mc{T}^* + \ww{4}\mc{D}^* \ts \ww{3}\mc{T}^* + \ww{2}\mc{D}^* \ts \ww{5}\mc{T}^*,
\eew
with all three summands vanishing since \(\operatorname{rank}\mc{D} = 3\) and \(\operatorname{rank}\mc{T} = 4\). This completes the proof.

\end{proof}

It follows at once that \(g_{\Xi^\mu}\) converges uniformly to:
\ew
g^\infty = \sum_{1 \le i,j \le 3} (\de_{ij} + g_{ij}) \dd y^i \s \dd y^j
\eew
on \(\ca{\M}_{int}\) as \(\mu \to \infty\), and moreover that \(g_{\Xi^\mu} - g^\infty\) is non-negative definite for all \(\mu\).  Once again, recalling that:
\ew
\ker\dd\h{q} = \<\frac{\del}{\del y^4}, \frac{\del}{\del y^5}, \frac{\del}{\del y^6}, \frac{\del}{\del y^7}\?.
\eew
(see \eref{ker-in-y's}) one sees that \(g^\infty\) is positive definite transverse to \(\ker\dd\h{q}\) on \(\ca{\M}_{int}\).

I now return to the metrics \(g^\mu = g_{\ca{\upphi}^\mu}\):

\begin{Prop}
The metrics \(g^\mu\) converge uniformly on \(\ca{\M}_{int}\) to \(g^\infty\) (as defined above), and moreover there exist constants \(\La_\mu' \to 1\) as \(\mu \to \infty\) such that \(g^\mu \ge \lt(\La_\mu'\rt)^2 g^\infty\) on \(\ca{\M}_{int}\), for all \(\mu \in [1,\infty)\).
\end{Prop}

\begin{proof}
Let \(\lt(a^1,...,a^7\rt)\) denote the canonical basis of \(\lt(\bb{R}^7\rt)^*\) and consider the \g\ 3-form:
\ew
\vph_0 = a^{123} + a^{145} + a^{167} + a^{246} - a^{257} - a^{347} - a^{356}.
\eew
Since the assignment \(\ph \in \ww[+]{3}\lt(\bb{R}^7\rt)^* \mt g_\ph \in \ss[+]{2}\lt(\bb{R}^7\rt)^*\) of a \g\ 3-form \(\ph\) to the metric it induces is smooth, there exist constants \(\de_1, \De_0 > 0\) such that if \(\lt|\ph - \vph_0\rt|_{\vph_0} < \de_1\), then:
\e\label{De0}
\lt|g_\ph - g_{\vph_0}\rt|_{\vph_0} \le \De_0\lt|\ph - \vph_0\rt|_{\vph_0}.
\ee
Since every \g\ 3-form on a 7-manifold is pointwise isomorphic to \(\vph_0\), it follows that \eref{De0} holds for general \g\ 3-forms on manifolds, with the same values of \(\de_1\) and \(\De_0\).

The proof now proceeds via repeated application of \eref{De0}, together with the following result, taken from the author's recent paper \cite{AGCRfOvSF}:
\begin{Lem}[{\cite[Lem.\ 7.2]{AGCRfOvSF}}]\label{LC-vs-norm}
Let \((\bb{A},g)\) be a finite-dimensional inner product space and write \(\|-\|_g\) for the norm on \(\ss{2}\bb{A}^*\) induced by \(g\).  Then for any \(h \in \ss{2}\bb{A}^*\):
\ew
h \le \|h\|_g\cdot g.
\eew
\end{Lem}
Firstly, by combining \lref{upxi-vs-ze}(1) with \eref{De0}, for all \(\ep \le \frac{\de_1}{C_1}\) (a condition which is independent of \(\mu\)):
\ew
\lt\|g_{\Xi^\mu} - g_{\h{\upxi}^\mu}\rt\|_{\h{\upxi}^\mu} \le \De_0 \lt\|\Xi^\mu - \h{\upxi}^\mu\rt\|_{\h{\upxi}^\mu} \le \De_0C_1\ep.
\eew
Applying \lref{LC-vs-norm}, it follows that:
\ew
g_{\Xi^\mu} - g_{\h{\upxi}^\mu} \le \De_0 C_1 \ep g_{\h{\upxi}^\mu} \et g_{\h{\upxi}^\mu} - g_{\Xi^\mu} \le \De_0 C_1 \ep g_{\h{\upxi}^\mu}
\eew
and hence:
\ew
(1 - \De_0 C_1 \ep) g_{\h{\upxi}^\mu} \le g_{\Xi^\mu} \le (1 + \De_0 C_1 \ep) g_{\h{\upxi}^\mu}.
\eew
In particular, for \(\ep < \frac{1}{\De_0 C_1}\) (a condition which is independent of \(\mu\)) \(g_{\h{\upxi}^\mu}\) and \(g_{\Xi^\mu}\) are Lipschitz equivalent on \(\ca{\M}_{int}\), uniformly in \(\mu\). Now by \lref{upxi-vs-ze}(2): \(\lt\|\ca{\upphi}^\mu - \Xi^\mu\rt\|_{\h{\upxi}^\mu} \to 0\) as \(\mu\to\infty\) and hence, by the Lipschitz equivalence just established, \(\lt\|\ca{\upphi}^\mu - \Xi^\mu\rt\|_{\Xi^\mu} \to 0\) as \(\mu \to \infty\).  Using \eref{De0} again, one sees that for all \(\mu\) sufficiently large:
\ew
\lt\|g^\mu - g_{\Xi^\mu}\rt\|_{\Xi^\mu} \le \De_0 \lt\|\ca{\upphi}^\mu - \Xi^\mu\rt\|_{\Xi^\mu}
\eew
and hence, by using \lref{LC-vs-norm} again:
\ew
g^\mu \ge \lt(1 - \De_0\lt\|\ca{\upphi}^\mu - \Xi^\mu\rt\|_{\Xi^\mu}\rt) g_{\Xi^\mu} \ge \lt(1 - \De_0\lt\|\ca{\upphi}^\mu - \Xi^\mu\rt\|_{\Xi^\mu}\rt) g^\infty,
\eew
where in the final line I have used that \(g_{\Xi^\mu} \ge g^\infty\) for all \(\mu\) as above. Thus, setting \(\La_\mu' = \sqrt{1 - \De_0\lt\|\ca{\upphi}^\mu - \Xi^\mu\rt\|_{\Xi^\mu}}\) for all \(\mu\) sufficiently large, one has \(\La_\mu' \to 1\) as required. Therefore, to conclude the proof, it suffices to prove that \(g^\mu \to g^\infty\) uniformly on \(\ca{\M}_{int}\) as \(\mu\to\infty\).

To this end, fix a reference metric \(g\) on \(\ca{\M}_{int}\). Since \(g_{\Xi^\mu} \to g^\infty\) uniformly, it follows that \(\lt\|g_{\Xi^\mu}\rt\|_g \le 2\lt\|g^\infty\rt\|_g = D\) for all \(\mu\) sufficiently large.  Thus by applying \lref{LC-vs-norm} one final time it follows that:
\ew
g_{\Xi^\mu} \le D g
\eew
for all \(\mu\) sufficiently large. Thus \(g_{\Xi^\mu} \ge D^{-3}g\) when acting on 3-forms. In particular:
\ew
\lt\|g^\mu - g_{\Xi^\mu}\rt\|_g &\le D^3\lt\|g^\mu - g_{\Xi^\mu}\rt\|_{\Xi^\mu}\\
&\le D^3\De_0 \lt\|\ca{\upphi}^\mu - \Xi^\mu\rt\|_{\Xi^\mu} \to 0 \text{ as } \mu \to \infty.
\eew
Thus, since \(g_{\Xi^\mu}\) tends to \(g^\infty\) uniformly, it follows that \(g^\mu\) also tends to \(g^\infty\) uniformly. This completes the proof.

\end{proof}

\noindent In summary, it has been shown that one may define a Riemannian semi-metric \(g^\infty\) on \(\lt.\ca{\M}\m\osr\coprod_{{\bf a}\in\fr{A}}\ca{\fr{W}}_{\bf a}\rt.\) with kernel precisely \(\ker\dd\h{q}\) such that on each \(\ca{\M}^{(k)}\), \(g^\mu \to g^\infty\) uniformly and there exist constants \(\La_\mu(k) \ge 0\)  such that:
\ew
\lim_{\mu \to \infty} \La_\mu(k) = 1 \et g^\mu \ge \La_\mu(k)^2 g^\infty \text{ for all } \mu \in [1,\infty).
\eew

I now explain how to define the limiting stratified Riemannian semi-metric \(\h{g}^\infty\) on all of \(\h{\M}\).  Using \(\rh\), one may identify \(\lt. \ca{\M} \m\osr \coprod_{{\bf a} \in \fr{A}} \ca{\fr{W}}_{\bf a} \rt.\) with a subset of \(\h{\M}\).  By examining \eref{g-infty-inner}, one may verify that \(g^\infty\) can be smoothly extended to a semi-metric on all of \(\h{\M}\).  Now recall that the strata of \(\Si\) (the stratification of \(\h{\M}\) induced by \(\h{q}\)) consist firstly of the preimage under \(\h{q}\) of the smooth locus of \(B\), secondly of the smooth loci of the singular fibres of \(\h{q}\) and thirdly of the components of the singular locus \(\h{S} \pc \h{\M}\).  On the first two types of strata of \(\Si\), simply define \(\h{g}^\infty\) to be the restriction of the semi-metric \(g^\infty\) to the stratum.  On the strata \(\h{S}_{\bf a}\), define \(\h{g}^\infty\) to be the semi-metric:
\ew
g^\infty_{\bf a} = \up^\frac{4}{3} \lt(\dd y^3\rt)^{\ts 2},
\eew
where \(\up\) is defined in \pref{tld-om-const}.  (The motivation for this definition will become apparent in the next section.)  Again, this can be extended to a semi-metric on all of \(\h{\M}\); indeed, it is easy to verify that:
\ew
\up^\frac{4}{3} g^\infty|_{\h{S}_{\bf a}} = g^\infty_{\bf a}.
\eew
Thus \(\h{g}^\infty\) defines a stratified Riemannian semi-metric on \(\h{\M}\).  Moreover, since \(g^\0\) (and hence \(\up^\frac{4}{3}g^\0\)) is positive definite transverse to \(\ker\dd\h{q}\), it follows that \(\h{g}^\infty\) is regular \wrt\ the stratified distribution \(\mc{D}\).

Thus, in summary, the stratified Riemannian semi-metric \(\h{g}^\0\) has been defined on all of \(\h{\M}\), and has been shown to satisfy \condsref{CT-Cond}(1)\&(2).

\subsection{Estimates on \(\h{g}^\infty\)}

The purpose of this subsection is to verify \condsref{CT-Cond}(3)--(5).  Specifically:
\begin{Prop}\label{hyps-2-4}
~
\begin{enumerate}
\setcounter{enumi}{2}
\item
\ew
\lim_{k \to \infty} ~ \limsup_{\mu \to \infty} ~\max_{{\bf a} \in \fr{A}} ~ \sup_{p \in S^1} ~ \diam_{d^\mu} \lt[ \ca{\fr{f}}_{\bf a}^{-1}(\{p\}) \cap \ca{W}_{{\bf a},k} \rt] = 0;
\eew

\item
\ew
\lim_{k \to \infty} ~ \max_{{\bf a} \in \fr{A}} ~ \sup_{p \in S^1} ~ \diam_{d^\infty}\lt[ \h{\fr{f}}_{\bf a}^{-1}(\{p\}) \cap \h{W}_{{\bf a},k} \rt] = 0;
\eew

\item
\ew
\lim_{k \to \infty} ~ \limsup_{\mu \to \infty} ~\max_{{\bf a} \in \fr{A}} ~ \sup_{\del_{{\bf a},k}} \lt| d^\mu - d^\infty \rt| = 0,
\eew
where \(d^\infty\) is the metric on \(\h{\M}\) induced by \(\h{g}^\infty\) (regarded as a metric on \(\del_{{\bf a},k}\) using the identification \(\rh\)) and, for simplicity of notation, I write \(\del_{{\bf a},k}\) for the subset \(\del \ca{W}_{{\bf a},k} \overset{\rh}{\cong} \del \h{W}_{{\bf a},k}\).

\end{enumerate}
\end{Prop}

\begin{proof}
~

(3) For each \(p\in S^1\), write:
\ew
\diam\lt[\ca{\fr{f}}_{\bf a}^{-1}(\{p\}) \cap \ca{W}_{{\bf a},k}, g^\mu\rt]
\eew
for the diameter of the space \(\ca{\fr{f}}_{\bf a}^{-1}(\{p\}) \cap \ca{W}_{{\bf a},k}\) \wrt\ the intrinsic metric induced by the Riemannian metric \(g^\mu\), i.e.\ the metric defined using paths contained entirely within \(\ca{\fr{f}}_{\bf a}^{-1}(\{p\}) \cap \ca{W}_{{\bf a},k}\). Then clearly:
\ew
\diam_{d^\mu}\lt[\ca{\fr{f}}_{\bf a}^{-1}(\{p\}) \cap \ca{W}_{{\bf a},k} \rt] \le \diam\lt[\ca{\fr{f}}_{\bf a}^{-1}(\{p\}) \cap \ca{W}_{{\bf a},k}, g^\mu\rt]
\eew
and so it suffices to prove that:
\ew
\lim_{k \to \infty} ~ \limsup_{\mu \to \infty} ~\max_{{\bf a} \in \fr{A}} ~ \sup_{p \in S^1} ~ \diam \lt[ \ca{\fr{f}}_{\bf a}^{-1}(\{p\}) \cap \ca{W}_{{\bf a},k}, g^\mu \rt] = 0.
\eew

Initially, fix \(k\in[1,\infty)\), \({\bf a} \in \fr{A}\) and consider \(\mu \ge k\). Recall from \eref{ca-ph-mu-homthty} that there is a homothety:
\ew
\fr{H}^\mu: \ca{W}_{\bf a} \cong \lt(T\rt)_{y^3,y^4,y^7} \x \lt(\tld{X}_{\frac{1}{2}\ep}\rt)_{y^1,y^2,y^5,y^6} \to \lt(T_\mu\rt)_{y^3,y^4,y^7} \x \tld{X}\lt(\tfrac{1}{2}\ep, \mu^{-1}\rt)_{y^1,y^2,y^5,y^6}
\eew
(given by rescaling the \(y^1\), \(y^2\) and \(y^3\) directions by \(\mu^3\)) which identifies the \g\ 3-form \(\ca{\ph}^\mu\) on the left-hand side with the \g\ 3-form \(\mu^{-3} \ze^\mu = \mu^{-3}\lt(\ze + \mu^{-3}\si\rt)\) on the right-hand side, where \(\ze\) and \(\si\) are defined in \erefs{ze-defn} and \eqref{si-defn} respectively. In particular, the homothety identifies \(\ca{\upphi}^\mu = \mu^{-6}\ca{\ph}^\mu\) on the left-hand side with the \g\ 3-form \(\mu^{-9}\ze^\mu\) on the right-hand side.

Note also that there is a natural map \(f: T_\mu \to \mu^3S^1\) given by projecting onto the first coordinate (i.e.\ \(y^3\)).  Given \(p \in S^1\), write \(T_{\mu,p}\) for the fibre of this map over the point \(\mu^3p \in \mu^3S^1\), which can naturally be identified with the torus \(\bb{T}^2\), irrespective of whether \(T_\mu = \bb{T}^3_\mu\) or \(\tld{\bb{T}^3_\mu}\).  Then, using the diffeomorphism invariance of intrinsic diameter, one may compute that:
\e\label{diam-1-1}
\diam\lt[\ca{\fr{f}}_{\bf a}^{-1}(\{p\}) \cap \ca{W}_{{\bf a},k}, g^\mu\rt] &= \diam\lt[\tld{X}\lt(\frac{1}{2}\ep,\frac{k}{\mu}\rt) \x T_{\mu,p} , \mu^{-9}\ze^\mu\rt]\\
&= \mu^{-3}\diam\lt[\tld{X}\lt(\frac{1}{2}\ep,\frac{k}{\mu}\rt) \x T_{\mu,p}, \ze^\mu\rt].
\ee
Now, as seen in \lref{tech-lem-for-res}, \(\lt\|\mu^{-3}\si\rt\|_{\ze}\) is bounded on \(\tld{X}\lt(\frac{1}{2}\ep,\mu^{-1}\rt) \x T_\mu\) by some absolute constant \(C'>0\) independent of \(\mu\), which may be made arbitrarily small by choosing \(\ep\) sufficiently small (independent of \(\mu\)) and \(\mu\) sufficiently large.  \Wlg\ one may assume that \(C' < \de_1\) for \(\de_1\) as in \eref{De0} and thus:
\ew
\|g_{\ze^\mu} - g_\ze\|_\ze \le \De_0C' \text{ on } \tld{X}\lt(\tfrac{1}{2}\ep,\mu^{-1}\rt) \x T_\mu.
\eew
Since \(\tld{X}\lt(\frac{1}{2}\ep,\frac{k}{\mu}\rt) \x T_{\mu,p} \pc \tld{X}\lt(\frac{1}{2}\ep,\mu^{-1}\rt) \x T_\mu\) it follows by \lref{LC-vs-norm} that:
\ew
g_{\ze^\mu} \le (1+\De_0C')g_\ze \text{ on } \tld{X}\lt(\frac{1}{2}\ep,\frac{k}{\mu}\rt) \x T_{\mu,p}.
\eew
Therefore:
\e\label{diam-1-2}
\diam\lt[\tld{X}\lt(\frac{1}{2}\ep,\frac{k}{\mu}\rt) \x T_{\mu,p}, \ze^\mu\rt] \le \sqrt{1+ \De_0C'}\diam\lt[\tld{X}\lt(\frac{1}{2}\ep,\frac{k}{\mu}\rt) \x T_{\mu,p}, \ze\rt].
\ee
Now outside the region \(\tld{X}_{\frac{1}{2}\ep} = \tld{X}\lt(\tfrac{1}{2}\ep,1\rt)\), the metric induced by \(\ze\) is Euclidean. It follows that:
\e\label{diam-1-3}
\diam\lt[\tld{X}\lt(\frac{1}{2}\ep,\frac{k}{\mu}\rt) \x T_{\mu,p}, \ze\rt] \le C''\lt(\frac{\mu}{k}\rt)^3
\ee
for some \(C''>0\) independent of \(k\), \(\mu\) and \(p\). Combining \erefs{diam-1-1}, \eqref{diam-1-2} and \eqref{diam-1-3} gives:
\ew
\diam\lt[\ca{\fr{f}}_{\bf a}^{-1}(\{p\}) \cap \ca{W}_{{\bf a},k}, g^\mu\rt] \le \mu^{-3}\sqrt{1+\De_0C'}C''\lt(\frac{\mu}{k}\rt)^3 = C''\sqrt{1+\De_0C'}k^{-3}.
\eew
Taking supremum over \(p\in S^1\), maximum over \({\bf a} \in \fr{A}\), limit superior over \(\mu\to\infty\) and the limit over \(k\to\infty\) then gives the required result.

(4) Since every point of \(\h{S}_{\bf a}\) is a limit point of \(\h{W}_{{\bf a},k} \osr \h{S}_{\bf a}\), one clearly has:
\ew
\lim_{k \to \infty} ~ \max_{{\bf a} \in \fr{A}} ~ \sup_{p \in S^1} ~ \diam_{d^\infty}\lt[ \h{\fr{f}}_{\bf a}^{-1}(\{p\}) \cap \h{W}_{{\bf a},k} \rt] = \lim_{k \to \infty} ~ \max_{{\bf a} \in \fr{A}} ~ \sup_{p \in S^1} ~ \diam_{d^\infty}\lt[ \h{\fr{f}}_{\bf a}^{-1}(\{p\}) \cap \lt. \h{W}_{{\bf a},k} \m\osr \h{S}_{\bf a} \rt. \rt].
\eew

Now \(\h{g}^\infty\) is simply given by \(g^\infty\) on \(\lt. \h{\M} \m\osr \h{S} \rt.\).  Thus, one has:
\ew
\lim_{k \to \infty} ~ \max_{{\bf a} \in \fr{A}} ~ \sup_{p \in S^1} ~ \diam_{d^\infty}\lt[ \h{\fr{f}}_{\bf a}^{-1}(\{p\}) \cap \lt. \h{W}_{{\bf a},k} \m\osr \h{S}_{\bf a} \rt. \rt] \le \lim_{k \to \infty} ~ \max_{{\bf a} \in \fr{A}} ~ \sup_{p \in S^1} ~ \diam \lt[ \h{\fr{f}}_{\bf a}^{-1}(\{p\}) \cap \lt. \h{W}_{{\bf a},k} \m\osr \h{S}_{\bf a} \rt., g^\infty \rt]
\eew
where, again, \(\diam \lt[ \h{\fr{f}}_{\bf a}^{-1}(\{p\}) \cap \lt. \h{W}_{{\bf a},k} \m\osr \h{S}_{\bf a} \rt., g^\infty \rt]\) denotes the diameter of \(\h{\fr{f}}_{\bf a}^{-1}(\{p\}) \cap \lt. \h{W}_{{\bf a},k} \m\osr \h{S}_{\bf a} \rt.\) \wrt\ the intrinsic semi-metric defined by \(g^\infty\), i.e.\ the semi-metric defined using paths contained entirely within \(\h{\fr{f}}_{\bf a}^{-1}(\{p\}) \cap \lt. \h{W}_{{\bf a},k} \m\osr \h{S}_{\bf a} \rt.\).

Now fix \({\bf a} \in \fr{A}\) and recall that:
\ew
\h{\fr{f}}_{\bf a}^{-1}(\{p\}) \cap \lt. \h{W}_{{\bf a},k} \m\osr \h{S}_{\bf a} \rt. = \lt[\lt(\lqt{\B^4\lt(\tfrac{1}{2}\ep,k\rt)}{\{\pm1\}}\rt) \m\osr \{0\}\rt] \x T_p,
\eew
where \(T_p\) is the fibre over \(p \in S^1\) of the projection \(T \to S^1\) onto the first coordinate, and:
\ew
g^\infty = \lt(1 - \frac{\lt(y^1\rt)^2}{4}\rt)^\frac{-1}{3}\lt\{\lt(\dd y^1\rt)^2 + \lt(\dd y^2\rt)^2 + \lt(\dd y^3\rt)^2 + y^1\dd y^1 \s \dd y^3\rt\},
\eew
on this region.  Write \(g^\infty = g_{s\E} + \vpi\) where:
\ew
g_{s\E} = \lt(\dd y^1\rt)^2 + \lt(\dd y^2\rt)^2 + \lt(\dd y^3\rt)^2 \et \vpi = g^\infty - g_{s\E}.
\eew
Note that \(g^\infty\) and \(g_{s\E}\) are not Riemannian metrics on \(\lt. \h{W}_{{\bf a},k} \m\osr \h{S}_{\bf a} \rt.\) and so, {\it a priori}, it is not clear that \lref{LC-vs-norm} applies. However, if one restricts attention to the distribution \(\<\frac{\del}{\del y^1}, \frac{\del}{\del y^2}, \frac{\del}{\del y^3}\?\), then both \(g^\infty\) and \(g_{s\E}\) are non-degenerate (i.e.\ inner products) and thus \lref{LC-vs-norm} applies.  One may compute that, on this distribution, over the region \(\lt. \h{W}_{{\bf a},k} \m\osr \h{S}_{\bf a} \rt.\):
\ew
\|\vpi\|_{g_{s\E}} \le Dk^{-3}
\eew
for some \(D >0\) independent of \(k\).  Thus by \lref{LC-vs-norm}:
\ew
g^\infty \le \lt(1+Dk^{-3}\rt)g_{s\E}.
\eew
Hence:
\ew
\diam \lt[ \h{\fr{f}}_{\bf a}^{-1}(\{p\}) \cap \lt. \h{W}_{{\bf a},k} \m\osr \h{S}_{\bf a} \rt., g^\infty \rt] \le \sqrt{1 + D k^{-3}}\diam \lt[ \h{\fr{f}}_{\bf a}^{-1}(\{p\}) \cap \lt. \h{W}_{{\bf a},k} \m\osr \h{S}_{\bf a} \rt., g_{s\E} \rt].
\eew
Clearly, \(\diam \lt[ \h{\fr{f}}_{\bf a}^{-1}(\{p\}) \cap \lt. \h{W}_{{\bf a},k} \m\osr \h{S}_{\bf a} \rt., g_{s\E} \rt]\) is bounded by \(D'k^{-3}\) for some \(D'>0\) independent of \(k\) and \(p\). Thus one has:
\ew
\diam \lt[ \h{\fr{f}}_{\bf a}^{-1}(\{p\}) \cap \lt. \h{W}_{{\bf a},k} \m\osr \h{S}_{\bf a} \rt., g^\0 \rt] \le D'k^{-3}\sqrt{1+Dk^{-3}}.
\eew
Taking supremum over \(p\in S^1_{y^3}\), maximum over \({\bf a} \in \fr{A}\) and the limit as \(k\to\infty\) gives the required result.

(5)  To prove this result, it is useful to introduce a third semi-metric on the region \(\del_{{\bf a},k}\) as follows.  Equip \(S^1\) with the metric \(\up^\frac{2}{3}d_\E\), where \(\up\) is as defined in \pref{tld-om-const}.  Pulling this metric back along the restriction \(\ca{\fr{f}}_{\bf a} \cong \h{\fr{f}}_{\bf a}: \del_{{\bf a},k} \to S^1\) defines a (\(\mu\)-independent) semi-metric on \(\del_{{\bf a},k}\), which I shall denote \(d_k\).  Explicitly:
\ew
d_k(x,y) = \up^\frac{2}{3}d_\E\lt(\ca{\fr{f}}_{\bf a}(x),\ca{\fr{f}}_{\bf a}(y)\rt).
\eew
To prove \pref{hyps-2-4}(5), I shall prove the following two statements:
\begin{itemize}
\item[(\(5i\))]
\ew
\lim_{k \to \infty} ~ \limsup_{\mu \to \infty} ~ \max_{{\bf a} \in \fr{A}} ~ \sup_{\del_{{\bf a},k}} \lt|d^\mu - d_k\rt| = 0;
\eew

\item[(\(5ii\))]
\ew
\lim_{k \to \infty} ~ \max_{{\bf a} \in \fr{A}} ~ \sup_{\del_{{\bf a},k}} \lt|d_k - d^\infty\rt| = 0.
\eew
\end{itemize}
Clearly, these collectively imply \pref{hyps-2-4}(5).

To prove (\(5i\)), I employ the following strategy: it is necessary to bound both \(d_k - d^\mu\) and \(d^\mu - d_k\) from above.   For the first of these two quantities, it is shown that, on all of \(\ca{\M}\): `\(g^\mu \ge \up^\frac{4}{3}\ca{\fr{f}}^*g_\E\)' in the limit as \(\mu\to\infty\).   Thus `\(d^\mu \ge d^k\)' (again, in a limiting sense) and hence \(d_k - d^\mu\) can be bounded above.   For the second, for any two points \(x,y \in \del_k\), an explicit path \(\ga:x \to y\) is constructed whose length with respect to \(g^\mu\) is approximately \(d_k(x,y)\), with this approximation becoming exact in the limit as first \(\mu\to\infty\) and then \(k\to\infty\).  The strategy for (\(5ii\)) is similar.

(\(5i\)) Recall the following decomposition of \(\ca{\M}\):
\ew
\ca{\M} = \itr{\M} \bigcup \ca{\M}_{int} \bigcup \lt(\coprod_{{\bf a} \in \fr{A}} \ca{W}_{\bf a}\rt).
\eew
As in \sref{bulk-conv}, I shall consider each region in turn.\\

\noindent \ul{\(\itr{\M}\)}:   Here:
\ew
g^\mu = \lt(\lt(\th^1\rt)^{\ts2} + \lt(\th^2\rt)^{\ts2} + \lt(\th^3\rt)^{\ts2}\rt) + \mu^{-6}\lt(\lt(\th^4\rt)^{\ts2} + \lt(\th^5\rt)^{\ts2} + \lt(\th^6\rt)^{\ts2} + \lt(\th^7\rt)^{\ts2}\rt)
\eew
and thus evidently:
\e\label{itr-bd}
g^\mu \ge \lt(\th^3\rt)^{\ts2} = \lt(\dd x^3\rt)^{\ts2} = \ca{\fr{f}}^*g_\E \ge \up^\frac{4}{3}\ca{\fr{f}}^*g_\E
\ee
on \(\itr{\M}\) (recall that \(\up < 1\)).\\

\noindent \ul{\(\ca{\M}_{int}\)}:   On this region, recall that \(\ca{\upphi}^\mu = \mu^{-6}\h{\ph}^\mu\), where \(\h{\ph}^\mu\) was defined in \pref{almostprod}.   Moreover, recall from the proof of the same proposition that:
\ew
\lt|\h{\ph}^\mu - \h{\xi}^\mu\rt|_{\h{\xi}^\mu} \le (4C+1)\ep
\eew
for some constant \(C>0\) independent of \(\mu\), where \(\h{\xi}^\mu\) is given by:
\ew
\h{\xi}^\mu = \mu^6\dd y^{123} + \dd y^{145} + \dd y^{167} - \dd y^{246} + \dd y^{257} + \dd y^{347} + \dd y^{356}.
\eew
Hence, by applying a simple rescaling:
\ew
\lt|\ca{\upphi}^\mu - \h{\upxi}^\mu\rt|_{\h{\upxi}^\mu} \le (4C+1)\ep,
\eew
where \(\h{\upxi}^\mu = \mu^{-6}\h{\xi}^\mu\) as in \sref{Mint}. It follows from \eref{De0} that, for all \(\ep > 0\) sufficiently small (independent of \(\mu\)), the metric \(g^\mu\) induced by \(\ca{\upphi}^\mu\) satisfies:
\ew
\lt|g^\mu - g_{\h{\upxi}^\mu}\rt|_{g_{\h{\upxi}^\mu}} \le C'\ep
\eew
for some constant \(C' > 0\), independent of \(\mu\).   By applying \lref{LC-vs-norm}, it follows that on \(\ca{\M}_{int}\):
\ew
g^\mu \ge \lt(1 - C'\ep\rt)g_{\h{\upxi}^\mu}.
\eew
However an explicit calculation shows that:
\ew
g_{\h{\upxi}^\mu} = \lt(\lt(\dd y^1\rt)^{\ts2} + \lt(\dd y^2\rt)^{\ts2} + \lt(\dd y^3\rt)^{\ts2}\rt) + \mu^{-6}\lt(\lt(\dd y^4\rt)^{\ts2} + \lt(\dd y^5\rt)^{\ts2} + \lt(\dd y^6\rt)^{\ts2} + \lt(\dd y^7\rt)^{\ts2}\rt)
\eew
and thus:
\ew
g_{\h{\upxi}^\mu} \ge \lt(\dd y^3\rt)^{\ts2} = \lt(\dd x^3\rt)^{\ts2} = \ca{\fr{f}}^*g_\E.
\eew
Thus on \(\ca{\M}_{int}\):
\ew
g^\mu \ge \lt(1 - C'\ep\rt)\ca{\fr{f}}^*g_\E.
\eew
Now recall that \(\up<1\) is independent of \(\mu\) and \(t\) and hence also independent of \(\ep\).   Thus for \(\ep\) sufficiently small independent of \(\mu\), one has that:
\ew
1 - C'\ep \ge \up^\frac{4}{3}.
\eew
Thus, reducing \(\ep\) equally for all \(\mu\) if necessary, one has that on \(\ca{\M}_{int}\):
\e\label{int-bd}
g^\mu \ge \up^\frac{4}{3}\ca{\fr{f}}^*g_\E.
\ee\\

\noindent \ul{\(\coprod_{{\bf a}\in\fr{A}} \ca{W}_{\bf a}\)}:   Fix some \(\ca{W}_{\bf a}\), \({\bf a} \in \fr{A}\), and, as for the proof of (3), begin by considering the homothetic region \(T_\mu \x \tld{X}\lt(\tfrac{1}{2}\ep,\mu^{-1}\rt)\) equipped with the \g\ 3-form \(\ze^\mu = \ze + \mu^{-3}\si\) (cf.\ \lref{tech-lem-for-res}).

Initially, consider the region outside \(T_\mu \x \tld{X}_{\frac{1}{2}\ep}\), i.e.\ the region \(\lt\{r \ge \frac{1}{2}\ep\rt\}\).  On this region \(\ze = \hat{\xi}\) is just the standard (Euclidean) \g\ 3-form in the coordinates \((y^i)\) and so:
\ew
g_\ze \ge \lt(\dd y^3\rt)^{\ts2} = f^*g_\E,
\eew
where \(f\) denotes the composite:
\ew
T_\mu \x \tld{X}\lt(\tfrac{1}{2}\ep,\mu^{-1}\rt) \oto{proj_1} T_\mu \oto{\lt(y^3,y^4,y^7\rt) \mt y^3} \mu^3S^1,
\eew
as above.  Moreover, recall from \lref{tech-lem-for-res} that outside \(T_\mu \x \tld{X}_{\frac{1}{2}\ep}\):
\ew
\lt\|\ze^\mu - \ze\rt\|_\ze \le \frac{1}{2}\ep.
\eew
and thus by \eref{De0}, for all \(\ep>0\) sufficiently small (independent of \(\mu\)):
\ew
\lt\|g_{\ze^\mu} - g_\ze\rt\|_\ze \le \frac{\De_0\ep}{2}.
\eew
Hence by applying \lref{LC-vs-norm}:
\ew
g_{\ze^\mu} \ge \lt(1 - \frac{\De_0\ep}{2} \rt)g_\ze \ge \lt(1 - \frac{\De_0\ep}{2} \rt) f^*g_\E.
\eew
As before, reducing \(\ep\) equally for all \(\mu\) if necessary, one may assume that \(1 - \frac{\De_0\ep}{2} \ge \up^\frac{4}{3}\) and thus that:
\e\label{os-est}
g_{\ze^\mu} \ge \up^\frac{4}{3} f^*g_\E
\ee
outside \(T_\mu \x \tld{X}_{\frac{1}{2}\ep}\).

Next, consider the region \(T_\mu \x \tld{X}_{\frac{1}{2}\ep}\).   By \lref{scaling-lem}, on this region one can write:
\ew
g_\ze = \nu^\frac{4}{3}\lt(\dd y^3\rt)^{\ts2} + \nu^{-\frac{2}{3}}\lt[\lt(\dd y^4\rt)^{\ts2} + \lt(\dd y^7\rt)^{\ts2}\rt] + \nu^{-\frac{2}{3}}g_{\ca{\om}_t}
\eew
where \(\nu\) is defined by the equation:
\ew
\nu^2\ca{\om}_t^2 = \fr{Re}\tld{\Om}^2 = \fr{Im}\tld{\Om}^2.
\eew
Then clearly:
\ew
g_\ze \ge \nu^\frac{4}{3}\lt(\dd y^3\rt)^{\ts2} \ge \up^\frac{4}{3}\lt(\dd y^3\rt)^{\ts2},
\eew
where the final equality follows from the fact that \(\up\) was defined precisely as the minimum value of \(\nu\).   Now recall that on the region \(T_\mu \x \tld{X}_{\frac{1}{2}\ep}\) one has:
\ew
\lt\|\frac{1}{\mu^3}\si\rt\|_\ze \le \frac{C}{\mu^3}
\eew
for some \(C>0\) independent of \(\mu\).   Thus, by \eref{De0}, for all \(\mu\) sufficiently large one can write:
\e\label{effect-of-si}
\lt\|g_{\ze^\mu} - g_\ze\rt\|_{g_\ze} \le \frac{C\De_0}{\mu^3}.
\ee
Applying \lref{LC-vs-norm} yields:
\e\label{is-est}
g_{\ze^\mu} \ge \lt(1 - \frac{C\De_0}{\mu^3}\rt)g_\ze \ge\lt(1 - \frac{C\De_0}{\mu^3}\rt)\up^\frac{4}{3}f^*g_\E
\ee
on the region \(T_\mu \x \tld{X}_{\frac{1}{2}\ep}\).   Combining \erefs{os-est} and \eqref{is-est} then yields the estimate:
\e\label{ze^mu-bd}
g_{\ze^\mu} \ge \lt(1 - \frac{C\De_0}{\mu^3}\rt)\up^\frac{4}{3}f^*g_\E
\ee
on all of \(T_\mu \x \tld{X}\lt(\tfrac{1}{2}\ep,\mu^{-1}\rt)\).

Now recall the homothety \(\ca{W}_\textbf{a} \cong T \x \tld{X}_{\frac{1}{2}\ep} \oto{\fr{H}^\mu} T_\mu \x \tld{X}\lt(\tfrac{1}{2}\ep,\mu^{-1}\rt)\) and recall from \eref{tldze} that the 3-form \(\ca{\upphi}^\mu\) on \(\ca{W}_{\bf a}\) is defined by:
\ew
\ca{\upphi}^\mu = \mu^{-9}\lt(\fr{H}^\mu\rt)^*\ze^\mu.
\eew
Note also that the diagram:
\ew
\bcd
\ca{W}_{\bf a} \cong T \x \tld{X}_{\frac{1}{2}\ep} \ar[r, ^^22 \fr{H}^\mu ^^22] \ar[d, ^^22 \ca{\fr{f}} ^^22] & T_\mu \x \tld{X}\lt(\tfrac{1}{2}\ep,\mu^{-1}\rt) \ar[d, ^^22 f ^^22]\\
S^1 \ar[r, ^^22 \x \mu^3 ^^22] & \mu^3 S^1
\ecd
\eew
commutes.  Applying the homothety \(\fr{H}^\mu\) to \eref{ze^mu-bd} yields:
\ew
g^\mu = g_{\ca{\upphi}^\mu} &\ge \lt(1 - \frac{C\De_0}{\mu^3}\rt)\up^\frac{4}{3}\ca{\fr{f}}^*g_\E
\eew
on the region \(\ca{W}_{\bf a}\).

Combining this with \eref{itr-bd} and \eqref{int-bd} yields:
\e\label{full-bd}
g^\mu \ge \lt(1 - \frac{C\De_0}{\mu^3}\rt)\up^\frac{4}{3}\ca{\fr{f}}^*g_\E
\ee
on all of \(\ca{\M}\), hence:
\ew
d^\mu \ge \sqrt{1 - \frac{C\De_0}{\mu^3}}d_k
\eew
and whence:
\e\label{1sthalf}
d_k - d^\mu &\le \frac{C\De_0}{\mu^3},
\ee
since manifestly \(d_k(x,y) \le \up^\frac{4}{3} < 1\) for all \(x,y\in \ca{\M}\).   Thus, the quantity \(d_k - d^\mu\) has been bounded above uniformly on all of \(\ca{\M}\).

Now fix \({\bf a} \in \fr{A}\) and turn attention to bounding the quantity \(d^\mu - d_k\) from above on the subset \(\del_{{\bf a},k}\).   Recall the distance \(\fr{r}\) defined in \pref{tld-om-const}.   For each \(p \in S^1\), define a point \(\ca{p} \in \ca{W}_{{\bf a},k}\) using the local coordinates \((y^1,...,y^7)\) on \(\ca{W}_{{\bf a},k}\) as:
\ew
\ca{p} = \lt(\frac{\fr{r}}{\mu^3},0,p,0,0,0,0\rt) \in \ca{\fr{f}}_{\bf a}^{-1}(\{p\}) \cap \ca{W}_{{\bf a},\mu} \cc \ca{\fr{f}}_{\bf a}^{-1}(\{p\}) \cap \ca{W}_{{\bf a},k},
\eew
where the final inclusion holds when \(\mu \ge k\).  The first task is to understand the distance between points of the form \(\ca{p}\) \wrt\ the metric \(d^\mu\).

Given \(p,q \in S^1\), pick the shorter segment \(\ga\) connecting \(p \to q\) in \(S^1\) and use it to define a path \(\ca{\ga}\) in \(\ca{W}_{{\bf a},\mu}\) connecting \(\ca{p} \to \ca{q}\) via:
\ew
\ca{\ga} = \lt(\frac{\fr{r}}{\mu^3},0,\ga,0,0,0,0\rt).
\eew
Then clearly:
\ew
d^\mu\lt(\ca{p},\ca{q}\rt) \le \ell_{g^\mu}(\ca{\ga}).
\eew
To compute \(\ell_{g^\mu}(\ca{\ga})\), consider the homothety \(\ca{W}_{\textbf{a},\mu} \cong T \x \tld{X}_{\frac{1}{2}\ep,\mu} \oto{\fr{H}^\mu} T_\mu \x \tld{X}_{\frac{1}{2}\ep}\) and recall that \(\ca{\upphi}^\mu = \mu^{-9}\lt(\fr{H}^\mu\rt)^*\ze^\mu\).   Under \(\fr{H}^\mu\), \(\ca{\ga}\) becomes the path:
\ew
\ca{\ga}' = \lt(\fr{r},0,\mu^3\ga,0,0,0,0\rt).
\eew
Begin by considering the \g\ 3-form \(\ze\).   By \lref{scaling-lem}, this induces the metric:
\ew
g_\ze = \nu^\frac{4}{3}\lt(\dd y^3\rt)^{\ts2} + \nu^{-\frac{2}{3}}\lt[\lt(\dd y^4\rt)^{\ts2} + \lt(\dd y^7\rt)^{\ts2}\rt] + \nu^{-\frac{2}{3}}g_{\tld{\om}_t},
\eew
where \(\nu|_{r = \fr{r}} = \up\).   The length of \(\ca{\ga}'\) \wrt\ the metric induced by \(\ze\) is then clearly \(\up^\frac{2}{3}\mu^3d_\E(p,q)\).   Applying \lref{LC-vs-norm} to \eref{effect-of-si} yields:
\ew
g_{\ze^\mu} \le \lt(1 + \frac{C\De_0}{\mu^3}\rt)g_\ze
\eew
and so:
\ew
\ell_{g_{\ze^\mu}}\lt(\ca{\ga}'\rt) \le \sqrt{1 + \frac{C\De_0}{\mu^3}}\up^\frac{2}{3}\mu^3d_\E(p,q).
\eew

Pulling this equation back along the homothety \(\fr{H}^\mu\) (and rescaling by \(\mu^{-3}\)) gives:
\e\label{d-4-sp}
d^\mu\lt(\ca{p},\ca{q}\rt) &\le \ell_{g^\mu}(\ca{\ga})\\
&\le \sqrt{1 + \frac{C\De_0}{\mu^3}}\up^\frac{2}{3}d_\E(p,q)\\
&\le d_k\lt(\ca{p},\ca{q}\rt) + \frac{C\De_0}{\mu^3},
\ee
where, again, \(\up^\frac{2}{3}d_\E(p,q) \le 1\) for all \(p,q \in S^1\) has been used.

Now pick two arbitrary points \(x,y \in \del_k\).   Since, writing \(p = \ca{\fr{f}}_\textbf{a}(x)\) and \(q = \ca{\fr{f}}_\textbf{a}(y)\), one has \(d^\mu\lt(x, \ca{p}\rt) \le \diam_{d^\mu}\lt[\ca{\fr{f}}_{\bf a}^{-1}(\{p\}) \cap \ca{W}_{{\bf a},k}\rt]\) and \(d^\mu\lt(\ca{q},y\rt) \le \diam_{d^\mu}\lt[\ca{\fr{f}}_{\bf a}^{-1}(\{q\}) \cap \ca{W}_{{\bf a},k}\rt]\), it follows that:
\e\label{2ndhalf}
d^\mu(x,y) &\le d^\mu\lt(x, \ca{p}\rt) + d^\mu\lt(\ca{p},\ca{q}\rt) + d^\mu\lt(\ca{q},y\rt)\\
&\le d^\mu\lt(\ca{p},\ca{q}\rt) + 2\sup_{s \in S^1} \diam_{d^\mu}\lt[\ca{\fr{f}}_{\bf a}^{-1}(\{s\}) \cap \ca{W}_{{\bf a},k}\rt]\\
& \le d_k(x,y) + \frac{C\De_0}{\mu^3} + 2\sup_{s \in S^1} \diam_{d^\mu}\lt[\ca{\fr{f}}_{\bf a}^{-1}(\{s\}) \cap \ca{W}_{{\bf a},k}\rt],
\ee
where \eref{d-4-sp} was used in passing to the final line.  Combining \erefs{1sthalf} and \eqref{2ndhalf} gives:
\ew
\lt|d^\mu(x,y) - d_k(x,y)\rt| \le \frac{C\De_0}{\mu^3} + 2\sup_{s \in S^1} \diam_{d^\mu}\lt[\ca{\fr{f}}_{\bf a}^{-1}(\{s\}) \cap \ca{W}_{{\bf a},k}\rt]
\eew
uniformly in \(x\) and \(y\).   Taking \(\max\) over \({\bf a} \in \fr{A}\), \(\limsup\) over \(\mu \to \infty\) and subsequently the limit over \(k \to \infty\), together with \pref{hyps-2-4}(3) gives the required result.

(\(5ii\)) The argument in this case is similar but easier.  Firstly, by taking the (pointwise) limit of \eref{full-bd} as \(\mu \to \infty\), one obtains that on \( \lt. \h{\M} \m\osr \coprod_{{\bf a} \in \fr{A}}  \h{\fr{W}}_{\bf a} \rt.\):
\ew
g^\infty \ge \up^\frac{4}{3} \h{\fr{f}}^*g_\E
\eew
and hence, by continuity, this inequality holds on all of \(\lt. \h{\M} \m\osr \h{S} \rt.\).  Moreover, for each \({\bf a} \in \fr{A}\):
\ew
g^\infty_{\bf a} = \up^\frac{4}{3} \h{\fr{f}}^*g_\E.
\eew
It follows that:
\e\label{1sthalf-2}
d^\infty \ge d_k
\ee
on all of \(\h{\M}\).

For the converse bound, given \(x,y \in \del_k\), define \(p = \h{\fr{f}}_\textbf{a}(x)\), \(q = \h{\fr{f}}_\textbf{a}(y)\) and consider the points:
\ew
\h{p} = \lt(0,0,p,0,0,0,0\rt) \et \h{q} = \lt(0,0,q,0,0,0,0\rt) \in \h{W}_{{\bf a},k}.
\eew
Since \(d^\infty\lt(x, \h{p}\rt) \le \diam_{d^\infty}\lt[ \h{\fr{f}}_{\bf a}^{-1}(\{p\}) \cap \h{W}_{{\bf a},k} \rt]\) and \(d^\infty\lt(\h{q},y\rt) \le \diam_{d^\infty}\lt[ \h{\fr{f}}_{\bf a}^{-1}(\{q\}) \cap \h{W}_{{\bf a},k} \rt]\), it follows that:
\ew
d^\infty(x,y) &\le d^\infty\lt(x, \h{p}\rt) + d^\infty\lt(\h{p},\h{q}\rt) + d^\infty\lt(\h{q},y\rt)\\
&\le d^\infty\lt(\h{p},\h{q}\rt) + 2\sup_{s \in S^1}  \diam_{d^\infty}\lt[ \h{\fr{f}}_{\bf a}^{-1}(\{s\}) \cap \h{W}_{{\bf a},k} \rt] .
\eew
However \(d^\infty\lt(\h{p},\h{q}\rt)\) can easily be bounded as follows: choose the shorter segment \(\ga:p \to q\) in \(S^1\).  This defines a path \(\h{\ga}\) from \(\h{p}\) to \(\h{q}\) in \(\h{W}_{{\bf a},k}\) via \((0,0,\ga,0,0,0,0)\) which has length \(\up^\frac{2}{3}\dd_\E(p,q) = d_k(x,y)\) \wrt\ the \(g^\infty_{\bf a}\).  Thus:
\e\label{2ndhalf-2}
d^\infty(x,y) \le d_k(x,y) + 2\sup_{s \in S^1}  \diam_{d^\infty}\lt[ \h{\fr{f}}_{\bf a}^{-1}(\{s\}) \cap \h{W}_{{\bf a},k} \rt] .
\ee

Finally, combining \erefs{1sthalf-2} and \eqref{2ndhalf-2} gives:
\ew
\lt|d^\infty(x,y) - d_k(x,y)\rt| \le 2\sup_{s \in S^1}  \diam_{d^\infty}\lt[ \h{\fr{f}}_{\bf a}^{-1}(\{s\}) \cap \h{W}_{{\bf a},k} \rt]
\eew
uniformly in \(x, y \in \del_{{\bf a},k}\).   Taking \(\max\) over \({\bf a} \in \fr{A}\), letting \(k \to \infty\) and recalling \pref{hyps-2-4}(4) gives the required result.  This completes the proof of \tref{FFKM-CT2}.

\end{proof}

~\vs{5mm}

\noindent Laurence H.\ Mayther\\
University of Cambridge\\
United Kingdom\\
{\it lhm32@cam.ac.uk}

\end{document}